%% file: sn_paper_journal_black.tex
\pdfoutput=1

\documentclass[11pt]{article}
\usepackage[UKenglish]{babel}
\usepackage[a4paper,top=3cm,bottom=2cm,left=3cm,right=3cm,marginparwidth=1.75cm]{geometry}
\usepackage{amsfonts,amsmath,amsthm,amssymb} 
\usepackage{mathrsfs} 
\usepackage{stmaryrd} 

\RequirePackage[OT1]{fontenc}
\usepackage[square,numbers]{natbib}
\usepackage[bookmarks=false,draft=false]{hyperref}

\usepackage{comment} 
\usepackage{graphicx} 
\usepackage{color}
\usepackage[dvipsnames]{xcolor}

\usepackage{multirow}

\allowdisplaybreaks[1]

\input{originalCommand.tex}
\input{theorems_for_journal.tex}

\title{Inference for ergodic diffusions plus noise}
\author{Shogo H. Nakakita\\
	Graduate School of Engineering Science, Osaka University
	\and Masayuki Uchida\\
	Graduate School of Engineering Science and MMDS, Osaka University}
\date{\today}

\begin{document}
\maketitle

\input{abstract.tex}
\input{section_1_to_6_black.tex}

\input{proof_journal_revised_black.tex}

\section*{Acknowledgement}
This work 
was partially supported by 
Overseas Study Program of MMDS,
JST CREST,
JSPS KAKENHI Grant Number 
JP17H01100 
and Cooperative Research Program
of the Institute of Statistical Mathematics.

\bibliography{bibliography}

\end{document}

%% file: originalCommand.tex
\newcommand{\lm}[2]{\bar{#1}_{#2}}
\newcommand{\dop}{\mathrm{d}}
\renewcommand{\Re}{\mathbf{R}}
\renewcommand{\epsilon}{\varepsilon}
\newcommand{\vech}{\mathrm{vech}}

\newcommand{\E}[1]{\mathbf{E}\left[{#1}\right]}
\newcommand{\CE}[2]{\mathbf{E}\left[\left.{#1}\right|{#2}\right]}

\newcommand{\parens}[1]{\left({#1}\right)}
\newcommand{\lparens}[1]{\left({#1}\right.}
\newcommand{\rparens}[1]{\left.{#1}\right)}
\newcommand{\crotchet}[1]{\left[{#1}\right]}
\newcommand{\lcrotchet}[1]{\left[{#1}\right.}
\newcommand{\rcrotchet}[1]{\left.{#1}\right]}
\newcommand{\tuborg}[1]{\left\{{#1}\right\}}

\newcommand{\norm}[1]{\left\|{#1}\right\|}
\newcommand{\abs}[1]{\left|{#1}\right|}
\newcommand{\labs}[1]{\left|{#1}\right.}
\newcommand{\rabs}[1]{\left.{#1}\right|}

\newcommand{\cp}{\overset{P}{\to}}
\newcommand{\cl}{\overset{\mathcal{L}}{\to}}
\newcommand{\tr}{\mathrm{tr}}
\newcommand{\ip}[2]{#1\left\llbracket #2\right\rrbracket}

%% file: theorems_for_journal.tex
\newtheorem{theorem}{Theorem}
\newtheorem{proposition}[theorem]{Proposition}
\newtheorem{lemma}[theorem]{Lemma}
\newtheorem{corollary}[theorem]{Corollary}

\theoremstyle{remark}
\newtheorem{remark}{Remark}

%% file: abstract.tex
\begin{abstract}
	We research adaptive maximum likelihood-type estimation for an ergodic diffusion process where the observation is contaminated by noise. This methodology leads to the asymptotic independence of the estimators for the variance of observation noise, the diffusion parameter and the drift one of the latent diffusion process. 
	Moreover, it can lessen the computational burden compared to simultaneous 
	maximum likelihood-type  estimation. In addition to adaptive estimation, we propose a test to see if noise exists or not, and analyse real data as the example such that data contains observation noise with statistical significance.
\end{abstract}

%% file: section_1_to_6_black.tex
\section{Introduction}

We consider a $d$-dimensional ergodic diffusion process defined by the following stochastic differential equation
\begin{align}
\dop X_t = b(X_t, \beta)\dop t + a(X_t, \alpha)\dop w_t,\ X_0 = x_0,
\end{align}
where $\left\{w_t\right\}_{t\ge 0}$ is a $r$-dimensional standard Wiener process, $x_0$ is a $\Re^d$-valued random variable independent of $\left\{w_t\right\}_{t\ge0}$, $\alpha\in\Theta_1\subset \Re ^{m_{1}}$, $\beta\in\Theta_2\subset \Re ^{m_{2}}$ with $\Theta_1$ and $\Theta_2$ being compact and convex. Moreover, $b:\Re^d\times\Theta_2\to \Re ^d$, $a:\Re^d\times \Theta_1\to \Re^d\otimes\Re^r$ are known functions. We denote $\theta:=(\alpha,\beta)\in\Theta_1\times \Theta_2 =:\Theta$ and $\theta^{\star}=(\alpha^{\star},\beta^{\star})$ as the true value of $\theta$ which belongs to $\mathrm{Int}(\Theta)$.

We deal with the problem of parametric inference for $\theta$ with $\left\{Y_{ih_{n}}\right\}_{i=1,\ldots,n}$ defined by the following model
\begin{align}
Y_{ih_{n}}=X_{ih_{n}}+\Lambda^{1/2} \epsilon_{ih_{n}},\ i=0,\ldots,n,
\end{align}
where $h_{n}>0$ is the {discretisation step}, $\Lambda\in \Re^d\otimes\Re^{d}$ is a positive semi-definite matrix and $\left\{\epsilon_{ih_{n}}\right\}_{i=1,\ldots,n}$ is an i.i.d. sequence of $\Re^d$-valued random variables such that $\E{\epsilon_{ih_{n}}}=\mathbf{0}$, $\mathrm{Var}\parens{\epsilon_{ih_{n}}}=I_d$, and each component is independent of other components, $\tuborg{w_t}$ and $x_0$. {Hence the term $\Lambda^{1/2}\epsilon_{ih_{n}}$ indicates the exogenous noise}. Let $\Theta_{\epsilon}\in\Re^{d(d+1)/2}$ be the convex and compact parameter space such that $\theta_{\epsilon}:=\mathrm{vech}(\Lambda)\in\Theta_{\epsilon}$ and $\Lambda_{\star}$ be the true value of $\Lambda$ such that $\theta_{\epsilon}^{\star}:=\mathrm{vech}(\Lambda_{\star})\in\mathrm{Int}(\Theta_{\epsilon})$, where $\vech$ is the half-vectorisation operator. We denote $\vartheta := (\theta, \theta_{\epsilon})$ and $\Xi:=\Theta\times\Theta_{\epsilon}$. With respect to the sampling scheme, we assume that $h_{n}\to0$ and $nh_{n}\to\infty$ as $n\to\infty$. 


Our main concern with these settings is the adaptive maximum likelihood (ML)-type estimation scheme in the form of $
\hat{\Lambda}_{n}=\frac{1}{2n}\sum_{i=0}^{n-1}\parens{Y_{(i+1)h_{n}}-Y_{ih_{n}}}^{\otimes2}$,
\begin{align}
\mathbb{H}_{1,n}^{\tau}(\hat{\alpha}_{n}|\hat{\Lambda}_{n})&=\sup_{\alpha\in\Theta_1}\mathbb{H}_{1,n}^{\tau}(\alpha|\hat{\Lambda}_{n}),\\
\mathbb{H}_{2,n}(\hat{\beta}_{n}|\hat{\alpha}_{n})&=\sup_{\beta\in\Theta_2}\mathbb{H}_{2,n}(\beta|\hat{\alpha}_{n}),
\end{align}
where $A^{\otimes 2} = AA^T$ for any matrix $A$, $A^T$ indicates the transpose of $A$, $\mathbb{H}_{1,n}^{\tau}$ and $\mathbb{H}_{2,n}$ are quasi-likelihood functions, which are defined in Section 3.


The composition of the model above is quite analogous to that of discrete-time state space models (e.g., see \citep{PPC09}) in terms of expression of endogenous perturbation in the system of interest and exogenous noise attributed to observation separately. As seen in the assumption $h_{n}\to0$, this model that we consider is for the situation where high-frequency observation holds, and this requirement enhances the flexibility of modelling since our setting includes the models with non-linearity, dependency of the innovation on state space itself. In addition, adaptive estimation which also becomes possible through the high-frequency setting has the advantage in easing computational burden in comparison to simultaneous one. Fortunately, the number of situations where requirements are satisfied has been grown gradually, and will continue to soar because of increase in the amount of real-time data and progress of observation technology these days.

The idea of modelling with diffusion process concerning observational noise is no new phenomenon. For instance, in the context of high-frequency financial data analysis, the researchers have addressed the existence of "microstructure noise" with large variance with respect to time increment {questioning} the premise that what we observe are purely diffusions. The energetic research of the modelling with "diffusion + noise" has been conducted in the decade: some research have examined the asymptotics of this model in the framework of fixed time interval such that {$nh_{n}=1$} (e.g., \citep{GlJ01a}, \citep{GlJ01b}, \citep{JLMPV09}, \citep{PV09} and \citep{O17}); and \citep{Fa14} and \citep{Fa16} research the parametric inference of this model with ergodicity and the asymptotic framework ${nh_{n}}\to \infty$. For parametric estimation for {discretely} observed diffusion processes without measurement errors, see \citep{Fl89}, \citep{Y92}, \citep{Y11}, \citep{BS95}, \citep{K97} and references therein.

Our research is focused on the statistical inference for an ergodic diffusion plus noise. 
We give the estimation methodology with adaptive estimation that relaxes computational burden 
and that has been researched for ergodic diffusions so far (see \citep{Y92}, \citep{Y11}, \citep{K95}, \citep{UY12}, \citep{UY14}) in comparison to the simultaneous estimation of \citep{Fa14} and \citep{Fa16}. In previous researches the simultaneous asymptotic normality of $\hat{\Lambda}_{n}$, $\hat{\alpha}_{n}$ and $\hat{\beta}_{n}$ has not been shown, but our method allows us to see asymptotic normality and asymptotic independence of them with the different convergence rates. {Our methods also broaden the applicability of modelling with stochastic differential equations since it is more robust for the existence of noise than the existent results in {discretely observed} diffusion processes with ergodicity not concerning observation noise.}


As the real data analysis, we analyse the 2-dimensional wind data \citep{NWTC} and try to model the dynamics with 2-dimensional Ornstein-Uhlenbeck process. We utilise the fitting of our diffusion-plus-noise modelling and that of diffusion modelling with estimation methodology called local Gaussian approximation method (LGA method) which has been investigated for these decades (for instance, see \citep{Y92}, \citep{K95} and \citep{K97}).
{The result (see Section 5)} seems that there is considerable difference between these estimates: however, we cannot evaluate which is the more trustworthy fitting only with these results. It results from the fact that we cannot distinguish a diffusion from a diffusion-plus-noise; if $\Lambda_{\star}=O$, then the observation is not contaminated by noise and the estimation of LGA should be adopted for its asymptotic efficiency; but if $\Lambda_{\star}\neq O$, what we observe is no more a diffusion process and the LGA method loses its theoretical validity. Therefore, it is necessary to compose the statistical hypothesis test with $H_0: \Lambda = O$ and $H_1: \Lambda \neq O$. In addition to estimation methodology, we also research this problem of hypothesis test and propose a test which has the consistency property.



In Section 2, we {gather} the assumption and notation across the paper. Section 3 gives the main 
results 
of this paper. Section 4 examines the result of Section 3 with simulation. In Section 5 {we analyse the real data for wind velocity} named MetData with our estimators and LGA as discussed above and test whether noise does exist.

\section{Local means, notations and assumptions}
\subsection{Local means}
We partition the observation into $k_{n}$ blocks containing $p_{n}$ {observations} and examine the property of the following local means such that
\begin{align}
\lm{Z}{j}{}=\frac{1}{p_{n}}\sum_{i=0}^{p_{n}-1}Z_{j\Delta_{n}+ih_{n}},\ j=0,\ldots,k_{n}-1,
\end{align}
where $\tuborg{Z_{ih_{n}}}_{i=1,\ldots,n}$ is an arbitrary sequence of random variables on the mesh $\{ih_{n}\}_i$ as $\tuborg{Y_{ih_{n}}}_{i=1,\ldots,n}$, $\tuborg{X_{ih_{n}}}_{i=1,\ldots,n}$ and $\tuborg{\epsilon_{ih_{n}}}_{i=1,\ldots,n}$; 
and $\Delta_{n}=p_{n}h_{n}$. 
Note that $k_{n} p_{n} =n$ and $k_{n} \Delta_{n} =n h_{n}$.

In the same way as \citep{Fa14} and \citep{Fa16},
our estimation method is based on these local means with respect to the observation $\tuborg{Y_{ih_{n}}}_{i=1,\ldots,n}$. The idea is so straightforward; taking means of the data $\tuborg{Y_{ih_{n}}}$ in each partition should reduce the influence of the noise term $\tuborg{\epsilon_{ih_{n}}}$ because of the law of large numbers and then we will obtain the information of the latent process $\tuborg{X_{ih_{n}}}$. 

We show how local means work to extract the information of the latent process. 
The first plot on next page (Figure \ref{figPLP}) is the simulation of a 1-dimensional Ornstein-Uhlenbeck process such that
\begin{align}
\dop X_t = -\parens{X_t-1}\dop t + \dop w_t, X_0=1
\end{align}
where $r=1$, $n=10^4$ and $h_{n}=n^{-0.7}$. Secondly, we contaminate the observation with normally-distributed noise $\tuborg{\epsilon_{ih_{n}}}$ and $\Lambda_{\star}= 0.1$ and plot the observation on next page (Figure \ref{figPOB}). Finally we make the sequence of local means $\tuborg{\lm{Y}{j}}$ where $p_{n}=25$ and $k_{n}=400$ plot at the bottom of the next page (Figure \ref{figPLM}).

With these plots, it seems that the local means recover rough states of the latent processes, and actually it is possible to compose the quantity which converges to each state $\tuborg{X_{j\Delta_{n}}}$ on the mesh $\tuborg{j\Delta_{n}}$ for Proposition \ref{cor736} with the assumptions below.
\begin{figure}[p]
\centering
\includegraphics[bb= 0 0 720 480,width=10cm]{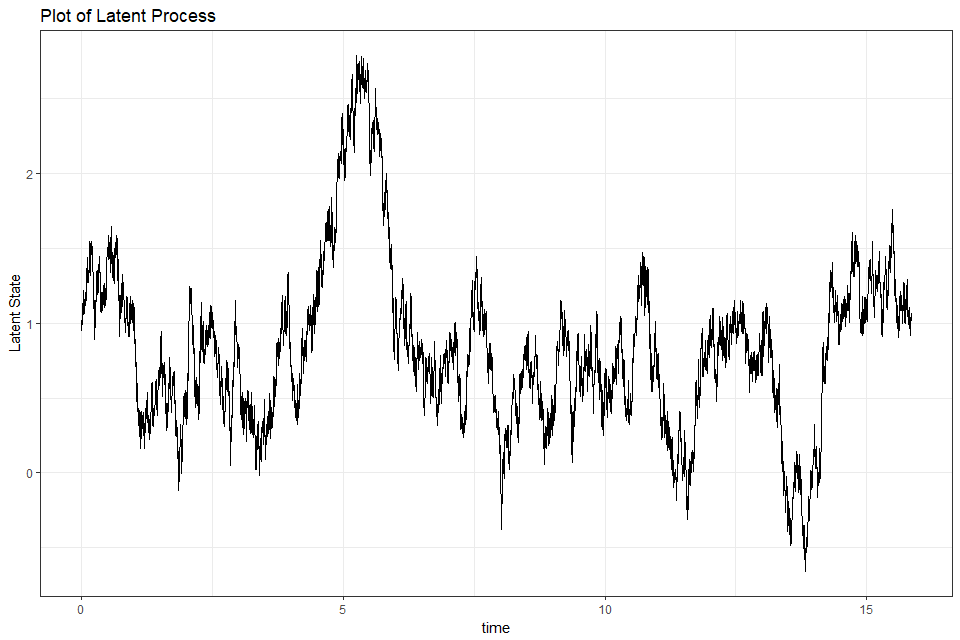}
\caption{plot of the latent process}\label{figPLP}
\includegraphics[bb=0 0 720 480,width=10cm]{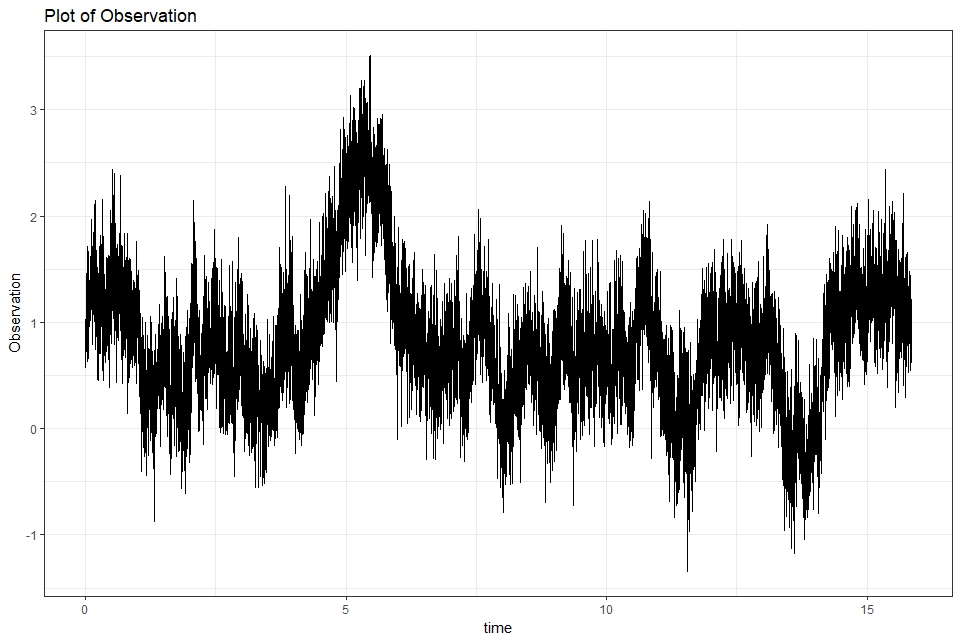}
\caption{plot of the contaminated observation}\label{figPOB}
\includegraphics[bb=0 0 720 480,width=10cm]{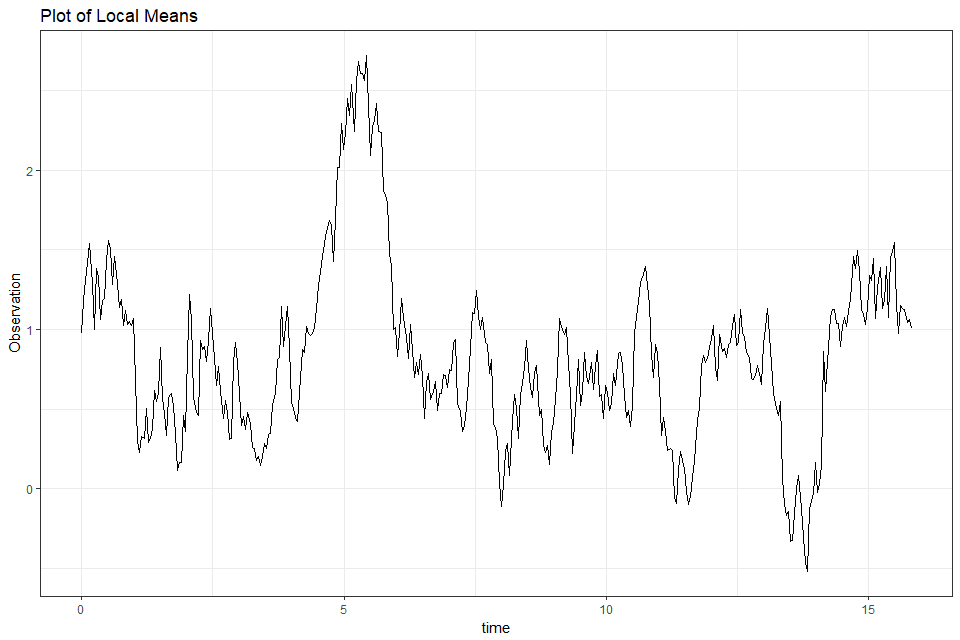}
\caption{plot of the local means}\label{figPLM}
\end{figure}

\subsection{Notations and assumptions}

We set the following notations.
\begin{enumerate}
	\item For a matrix $A$, $A^T$ denotes the transpose of $A$ and $A^{\otimes 2}:=AA^T$. For same size matrices $A$ and $B$, $\ip{A}{B}:=\tr\parens{AB^T}$. 
	\item For any vector $v$, $v^{(i)}$ denotes the $i$-th component of $v$. Similarly, $M^{(i,j)}$, $M^{(i,\cdot)}$ and $M^{(\cdot,j)}$ denote the $(i,j)$-th component, the $i$-th row vector and $j$-th column vector of a matrix $M$ respectively.
	\item For any vector $v$, $\norm{v}:=\sqrt{\sum_i\parens{v^{(i)}}^2}$, and for any matrix $M$, $\norm{M}:=\sqrt{\sum_{i,j}\parens{M^{(i,j)}}^2}$.
	\item $A(x,\alpha):=\parens{a(x,\alpha)}^{\otimes 2}.$
	\item $C$ is a positive generic constant independent of all other variables. If it depends on fixed other variables, e.g. an integer $k$, we will express as $C(k)$.
	\item $a(x):=a(x,\alpha^{\star})$ and $b(x):=b(x,\beta^{\star})$.
	\item Let us define $\vartheta:=\parens{\theta,\theta_{\epsilon}}\in \Xi$.
	\item A $\Re$-valued function $f$ on $\Re^d$ is a \textit{polynomial growth function} if for all $x\in \Re^d$,
	\begin{align*}
	\abs{f(x)}\le C\parens{1+\norm{x}}^C.
	\end{align*}
	$g:\Re^d\times \Theta\to \Re$ is a \textit{polynomial growth function uniformly in $\theta\in\Theta$} if for all $x\in \Re^d$,
	\begin{align*}
	\sup_{\theta\in\Theta}\abs{g(x,\theta)}\le C\parens{1+\norm{x}}^C.
	\end{align*}
	Similarly we say $h:\Re^d\times \Xi\to \Re$ is a \textit{polynomial growth function uniformly in $\vartheta\in\Xi$} if for all $x\in \Re^d$,
	\begin{align*}
	\sup_{\vartheta\in\Xi}\abs{h(x,\vartheta)}\le C\parens{1+\norm{x}}^C.
	\end{align*}
	\item Let us denote for any $\mu$-integrable function $f$ on $\Re^d$, $
	\mu(f(\cdot)) := \int f(x)\mu(\dop x).$
	\item 	We set \begin{align*}
	\mathbb{Y}_1^{\tau}(\alpha)&:=-\frac{1}{2}\nu_0\parens{\mathrm{tr}\parens{\parens{A^{\tau}(\cdot,\alpha,\Lambda_{\star})}^{-1}A^{\tau}(\cdot,\alpha^{\star},\Lambda_{\star})-I_d}+\log\frac{\det A^{\tau}(\cdot,\alpha,\Lambda_{\star})}{\det A^{\tau}(\cdot,\alpha^{\star},\Lambda_{\star})}},\\
	\mathbb{Y}_2(\beta)&:=-\frac{1}{2}\nu_0\parens{\ip{\parens{A(\cdot,\alpha^{\star})}^{-1}}{\parens{b(\cdot,\beta)-b(\cdot,\beta^{\star})}^{\otimes2}}},
	\end{align*}
	where $
	A^{\tau}(\cdot,\alpha,\Lambda):=A(\cdot,\alpha)+3\Lambda\mathbf{1}_{\left\{2\right\}}\left(\tau\right)$ and $\nu_0$ is the invariant measure of $X$.
	\item 
	Let
	\begin{align*}
	&\left\{B_{\kappa}(x)\left|\kappa=1,\ldots,m_1,\ B_{\kappa}=(B_{\kappa}^{(j_1,j_2)})_{j_1,j_2}\right.\right\},\\
	&\left\{f_{\lambda}(x)\left|\lambda=1,\ldots,m_2,\ f_{\lambda}=(f^{(1)}_{\lambda},\ldots,f^{(d)}_{\lambda})\right.\right\}
	\end{align*}
	be
	sequences of $\Re^d\otimes \Re^d$-valued functions and $\Re^d$-valued ones respectively such that the components of themselves and their derivative with respect to $x$ are polynomial growth functions for all $\kappa$ and $\lambda$. 
	Then we define the following matrix 
	\begin{align*}
	W_1^{(l_1,l_2),(l_3,l_4)}&:=\sum_{k=1}^{d}\parens{\Lambda_{\star}^{1/2}}^{(l_1,k)}\parens{\Lambda_{\star}^{1/2}}^{(l_2,k)}\parens{\Lambda_{\star}^{1/2}}^{(l_3,k)}\parens{\Lambda_{\star}^{1/2}}^{(l_4,k)}
	\parens{\E{\abs{\epsilon_0^k}^4}-3}\notag\\
	&\qquad+\frac{3}{2}\parens{\Lambda_{\star}^{(l_1,l_3)}\Lambda_{\star}^{(l_2,l_4)}+\Lambda_{\star}^{(l_1,l_4)}\Lambda_{\star}^{(l_2,l_3)}},
	\end{align*}
	and matrix-valued functionals, for $\bar{B}_{\kappa}:=\frac{1}{2}\parens{B_{\kappa}+B_{\kappa}^T}$,
	\begin{align*}
	\parens{W_2^{(\tau)}(\tuborg{B_{\kappa}})}^{(\kappa_1,\kappa_2)}&:=\begin{cases}
	\nu_0\parens{\tr\tuborg{\parens{\bar{B}_{\kappa_1}A\bar{B}_{\kappa_2}A}(\cdot)}} \\
	\qquad\text{ if }\tau\in(1,2),\\
	\nu_0\parens{\tr\tuborg{\parens{\bar{B}_{\kappa_1}A\bar{B}_{\kappa_2}A+4\bar{B}_{\kappa_1}A\bar{B}_{\kappa_2}\Lambda_{\star}+12\bar{B}_{\kappa_1}\Lambda_{\star}\bar{B}_{\kappa_2}\Lambda_{\star}}(\cdot)}}\\
	\qquad\text{ if }\tau=2,
	\end{cases}\\
	\parens{W_3(\tuborg{f_{\lambda}})}^{(\lambda_1,\lambda_2)}&:=\nu_0\parens{\parens{f_{\lambda_1}A\parens{f_{\lambda_2}}^T}(\cdot)}.
	\end{align*}
	\item $\cp$ and $\cl$ indicate convergence in probability and convergence in law respectively.
	\item For $f(x)$, $g(x,\theta)$ and $h(x,\vartheta)$, $f'(x):=\frac{\dop }{\dop x}f(x)$, $f''(x):=\frac{\dop^2}{\dop x^2}f(x)$, $\partial_{x}g(x,\theta):=\frac{\partial}{\partial x}g(x,\theta)$, $\partial_{\theta}g(x,\theta):=\frac{\partial}{\partial \theta}g(x,\theta)$, $\partial_{x}h(x,\vartheta):=\frac{\partial}{\partial x}h(x,\vartheta)$ and  $\partial_{\vartheta}h(x,\vartheta):=\frac{\partial}{\partial \vartheta}h(x,\vartheta)$.
\end{enumerate}

We make the following assumptions.
\begin{enumerate}
	\item[(A1)] $b$ and $a$ are continuously differentiable for 4 times,  and the components of themselves as well as their derivatives are polynomial growth functions uniformly in $\theta\in\Theta$. Furthermore, there exists $C>0$ such that for all $x\in\Re^d$,
	\begin{align*}
	&\norm{b(x)}+\norm{b'(x)}+\norm{b''(x)}\le C(1+\norm{x}), \\ 
	&\norm{a(x)}+\norm{a'(x)}+\norm{a''(x)}\le C(1+\norm{x}). 
	\end{align*}
	\item[(A2)] $X$ is ergodic and the invariant measure $\nu_0$ has $k$-th moment for all $k>0$.
	\item[(A3)] For all $k>0$, $\sup_{t\ge0}\E{\norm{X_t}^k}<\infty$.
	\item[(A4)] For any $k>0$, $\epsilon_{ih_{n}}$ has $k$-th moment and the component of $\epsilon_{ih_{n}}$ are independent of the other components for all $i$, $\tuborg{w_t}$ and $x_0$. In addition, the marginal distribution of each component is symmetric.
	\item[(A5)] $\inf_{x,\alpha} \det c(x,\alpha)>0$.
	\item[(A6)] There exist $\chi>0$ and $\tilde{\chi}>0$ such that for all $\tau$, $\alpha$ and $\beta$, $\mathbb{Y}_1^{\tau}(\alpha)\le -\chi\norm{\alpha-\alpha^{\star}}^2$, and $\mathbb{Y}_2^{\tau}(\beta)\le -\tilde{\chi}\norm{\beta-\beta^{\star}}^2$.
	\item[(A7)] The components of $b$, $a$, $\partial_xb$, $\partial_{\beta}b$, $\partial_xa$, $\partial_{\alpha}a$, $\partial_x^2b$, $\partial_{\beta}^2b$, $\partial_x\partial_{\beta}b$, $\partial_x^2a$, 
	$\partial_{\alpha}^2a$ and $\partial_x\partial_{\alpha}a$ are polynomial growth functions uniformly in $\theta\in\Theta$.
	\item[(AH)] $h_{n}=p_{n}^{-\tau},\ \tau\in(1,2]$ and $h_{n}\to0$, $p_{n}\to\infty$, $k_{n}\to\infty$, $\Delta_{n}=p_{n}h_{n}\to0$, $nh_{n}\to\infty$ as $n\to\infty$.
	\item[(T1)] If the index set $\mathcal{K}:=\tuborg{i\in\tuborg{1,\ldots,d}:\Lambda_{\star}^{(i,i)}>0}$ is not null, then the submatrix of $\Lambda_{\star}$ such that $\Lambda_{\star,\mathrm{sub}}:=\crotchet{\Lambda_{\star}^{(i_1,i_2)}}_{i_1,i_2\in \mathcal{K}}$
	is positive definite.
\end{enumerate}
{
\begin{remark}
	Some of the assumptions are ordinary ones in statistical inference for ergodic diffusions: (A1) and (A7) which indicates local Lipschitz continuity lead to existence and uniqueness of the solution of the stochastic differential equation. 
	(A2), (A3) and (A5) can be seen in the literature such as \citep{K97} and some sufficient conditions are shown in \citep{UY12}. 
	(A6) is the identifiability condition adopted in \citep{UY12} and \citep{Y11}, supporting {non-degeneracy} of information matrices simultaneously.
	
	(A4) is a stronger assumption with respect to integrability of $\left\{\epsilon_{ih_{n}}\right\}_{i=0,\ldots,n}$ compared to \citep{Fa14} and \citep{Fa16}. We consider multi-dimensional parameters in variance of noise, the diffusion coefficient and the drift one, and it results in the necessity of stronger integrability of $\epsilon_{ih_{n}}$ when we prove Theorem \ref{thm742}, which shows that some empirical functionals uniformly converge to 0 in probability.
\end{remark}
}
\section{Main results}
\subsection{Adaptive ML-type estimation}
Firstly, we construct an estimator for $\Lambda$ such that
$\hat{\Lambda}_{n}:=\frac{1}{2n}\sum_{i=0}^{n-1}\parens{Y_{(i+1)h_{n}}-Y_{ih_{n}}}^{\otimes2}$, {which takes the form similar to the quadratic variation of $\left\{Y_{ih_{n}}\right\}_{i=0,\ldots,n}$.}

\begin{lemma}\label{lem311}
	Under (A1)-(A4), $h_{n}\to0$ and $nh_{n}\to\infty$ as $n\to\infty$, $\hat{\Lambda}_{n}$ is consistent.
\end{lemma}

{
\begin{remark}
	The result of Lemma \ref{lem311} can be understood intuitively: note the order evaluation $\sum_{i=0}^{n-1}\left(X_{\left(i+1\right)h_{n}}-X_{ih_{n}}\right)^{\otimes2}=O_{P}\left(nh_{n}\right)$ and 
	$\sum_{i=0}^{n-1}\left(\epsilon_{\left(i+1\right)h_{n}}-\epsilon_{ih_{n}}\right)^{\otimes2}=O_{P}\left(n\right)$; and the cross term being merely a residual because of independence of $\left\{X_{t}\right\}_{t\ge0}$ and $\left\{\epsilon_{ih_{n}}\right\}_{i=0,\ldots,n}$.
\end{remark}
}

We propose the following {Gaussian-type} quasi-likelihood functions such that
{\small\begin{align}
	&\mathbb{H}_{1,n}^{\tau}(\alpha|\Lambda):=-\frac{1}{2}\sum_{j=1}^{k_{n}-2}
	\parens{\ip{\parens{\frac{2}{3}\Delta_{n} A_{n}^{\tau}(\lm{Y}{j-1},\alpha,\Lambda)}^{-1}}{\parens{\lm{Y}{j+1}-\lm{Y}{j}}^{\otimes 2}}
		+\log\det \parens{ A_{n}^{\tau}(\lm{Y}{j-1},\alpha,\Lambda)}}, \\
	&\mathbb{H}_{2,n}(\beta|\alpha)
	:=-\frac{1}{2}\sum_{j=1}^{k_{n}-2}
	\ip{\parens{\Delta_{n}A(\lm{Y}{j-1},\alpha)}^{-1}}{\parens{\lm{Y}{j+1}-\lm{Y}{j}-\Delta_{n}b(\lm{Y}{j-1},\beta)}^{\otimes 2}},
	\end{align}}
where $A_{n}^{\tau}(x,\alpha,\Lambda):=A(x,\alpha)+3\Delta_{n}^{\frac{2-\tau}{\tau-1}}\Lambda$. {They are quite similar to the quasi-likelihood functions for discretely-observed ergodic diffusion processes where $nh_{n}^{2}\to0$ proposed by \citep{UY12} except for the scaling $2/3$ seen in $\mathbb{H}_{1,n}^{\tau}$. It is because
\begin{align}
	\mathbf{E}\left[\left(\lm{Y}{j+1}-\lm{Y}{j}\right)^{\otimes2}|X_{j\Delta_{n}},\left\{\epsilon_{ih_{n}}\right\}_{i=0,\ldots,jp_{n}-1}\right]\approx\frac{2}{3}\Delta_{n}A_{n}^{\tau}\left(X_{j\Delta_{n}},\alpha^{\star},\Lambda_{\star}\right)
\end{align} in some sense (see Proposition \ref{pro737} or Theorem \ref{thm743}), which can be contrasted with that $\mathbf{E}\left[\left(X_{\left(i+1\right)h_{n}}-X_{ih_{n}}\right)^{\otimes2}|X_{ih_{n}}\right]\approx h_{n}A\left(X_{ih_{n}},\alpha^{\star}\right)$}. We define the {adaptive ML-type} estimators $\hat{\alpha}_{n}$ and $\hat{\beta}_{n}$ {corresponding to $\mathbb{H}_{1,n}^{\tau}$ and $\mathbb{H}_{2,n}$},  where
\begin{align}
\mathbb{H}_{1,n}^{\tau}(\hat{\alpha}_{n}|\hat{\Lambda}_{n})&=\sup_{\alpha\in\Theta_1}\mathbb{H}_{1,n}^{\tau}(\alpha|\hat{\Lambda}_{n}),\\
\mathbb{H}_{2,n}(\hat{\beta}_{n}|\hat{\alpha}_{n})&=\sup_{\beta\in\Theta_2}\mathbb{H}_{2,n}(\beta|\hat{\alpha}_{n}).
\end{align}
The consistency of these estimators is given by the next theorem.

\begin{theorem}\label{thm312}
	Under (A1)-(A7) and (AH), $\hat{\alpha}_{n}$ and $\hat{\beta}_{n}$ are consistent.
\end{theorem}
{
\begin{remark}
These adaptive ML-type estimators have the advantage that the computation burden in optimisation is reduced compared to that of simultaneous ML-type ones maximising the simultaneous quasi-likelihood $\mathbb{H}_{n}\left(\alpha,\beta|\Lambda\right)$ such that
\begin{align}
	\mathbb{H}_{n}\left(\hat{\alpha}_{n},\hat{\beta}_{n}|\hat{\Lambda}_{n}\right)
	=\sup_{\alpha\in\Theta_1,\beta\in\Theta_2}\mathbb{H}_{n}\left(\alpha,\beta|\hat{\Lambda}_{n}\right)
\end{align}
studied in \citep{Fa14} and \citep{Fa16}.
\end{remark}
}

{To argue the asymptotic normality of the estimators, let us denote the limiting information matrices}
\begin{align}
	\mathcal{I}^{\tau}(\vartheta^{\star})&:=\mathrm{diag}\left\{W_{1}, \mathcal{I}^{(2,2),\tau},\mathcal{I}^{(3,3)}\right\}(\vartheta^{\star})\\
	\mathcal{J}^{\tau}(\vartheta^{\star})&:=\mathrm{diag}\left\{I_{d(d+1)/2},\mathcal{J}^{(2,2),\tau},\mathcal{J}^{(3,3)}\right\}(\vartheta^{\star})
\end{align}
where for $i_1,i_2\in\tuborg{1,\ldots,m_1}$,
\begin{align}
\mathcal{I}^{(2,2),\tau}(\vartheta^{\star})&:=
W_2^{(\tau)}\parens{\tuborg{\frac{3}{4}\parens{ A^{\tau}}^{-1}\parens{\partial_{\alpha^{(k_1)}}A}\parens{ A^{\tau}}^{-1}(\cdot,\vartheta^{\star})}_{k_1}},\\
\mathcal{J}^{(2,2),\tau}(\vartheta^{\star})&:=
\crotchet{\frac{1}{2}\nu_0\parens{\tr\tuborg{\parens{A^{\tau}}^{-1}\parens{\partial_{\alpha^{(i_1)}}A}\parens{A^{\tau}}^{-1}
			\parens{\partial_{\alpha^{(i_2)}}A}}(\cdot,\vartheta^{\star})}}_{i_1,i_2},
\end{align}
{which are the information for the diffusion parameter $\alpha$, and for $j_1,j_2\in\tuborg{1,\ldots,m_2}$,}
\begin{align}
\mathcal{I}^{(3,3)}(\theta^{\star})&=\mathcal{J}^{(3,3)}(\theta^{\star}):=
\crotchet{\nu_0\parens{\ip{\parens{A}^{-1}}{\parens{\partial_{\beta^{(j_1)}}b}\parens{\partial_{\beta^{(j_2)}}b}^T}(\cdot,\theta^{\star})}}_{j_1,j_2}
\end{align}
{which are that for the drift one $\beta$.
To ease the notation, let us also }denote $\hat{\theta}_{\epsilon,n}:=\vech\hat{\Lambda}_{n}$ and $\theta_{\epsilon}^{\star}:=\vech \Lambda_{\star}$.

\begin{theorem}\label{thm313}
	Under (A1)-(A7), (AH) and $k_{n}\Delta_{n}^2\to0$, the following convergence in distribution holds:
	{
	\begin{align*}
	&\left[
		\sqrt{n}\left(\hat{\theta}_{\epsilon,n}-\theta_{\epsilon}^{\star}\right),\ 
		\sqrt{k_{n}}\left(\hat{\alpha}_{n}-\alpha^{\star}\right),\ 
		\sqrt{nh_{n}}\left(\hat{\beta}_{n}-\beta^{\star}\right)
	\right]\\[5pt]
	&\hspace{3cm}\cl N\parens{\mathbf{0},\parens{\mathcal{J}^{\tau}(\vartheta^{\star})}^{-1}\mathcal{I}^{\tau}(\vartheta^{\star})\parens{\mathcal{J}^{\tau}(\vartheta^{\star})}^{-1}}.
	\end{align*}}
\end{theorem}
{

This theorem shows the difference of the convergence rates with respect to $\hat{\theta}_{\epsilon,n}$, $\hat{\alpha}_{n}$ and $\hat{\beta}_{n}$ which is essentially significant to construct adaptive estimation approach. 
The difference among these convergence rates can be intuitively comprehended: the estimator for $\Lambda$ has $\sqrt{n}$-consistency as ordinary i.i.d. case because it is estimated with the quadratic variation of observation masking the influence of the latent process $\left\{X_{t}\right\}_{t\ge0}$ as noted in the remark for Lemma \ref{lem311}; $\hat{\alpha}$ which is estimated with the quasi-likelihoods composed by $k_{n}$ local means $\left\{\lm{Y}{j}\right\}_{j=0,\ldots,k_{n}-1}$ has $\sqrt{k_{n}}$-consistency corresponding to $\sqrt{n}$-consistency in the inference for discretely-observed diffusion processes; and $\hat{\beta}_{n}$ has $\sqrt{nh_{n}}$-consistency which is ordinary in statistics for diffusion processes.

Note that our estimator $\hat{\beta}_{n}$ is asymptotically efficient for all $\tau\in\left(1,2\right]$: the limiting variance of $\sqrt{nh_{n}}\left(\hat{\beta}_{n}-\beta^{\star}\right)$ is the inverse of the Fisher information. It is because we construct adaptive quasi-likelihood functions for the diffusion parameter $\alpha$ and the drift one $\beta$ separately; to the contrary, a simultaneous quasi-likelihood functions proposed in \citep{Fa14} and \citep{Fa16} cannot achieve the asymptotic efficiency under $\tau=2$.
}

{
\begin{remark}
(AH) inherits the assumption in both \citep{Fa14} and \citep{Fa16}. 
The tuning parameter $\tau$ controls the size and the number of partitions denoted as $p_{n}$ and $k_{n}$ given for observations, and for all $\tau\in\left(1,2\right]$ we have shown our estimators have asymptotic normality.
It can be tuned to get optimal in each application, but the following discussion may give guide to determine $\tau$. Generally speaking, larger $\tau$ has an advantage in asymptotics of our estimators 
since higher $\tau$ indicates smaller $k_{n}\Delta_{n}^{2}$
whose convergence to 0 is one of the conditions to show asymptotic normality of the estimators. Let us consider the case where it holds $nh_{n}^{\gamma}\to 0$ for some $\gamma\in\left(1,3/2\right]$; then $k_{n}\Delta_{n}^{2}=nh_{n}^{2-1/\tau}\to0$ if $\gamma>2-1/\tau$, which is the condition can be eased with larger $\tau$.
On the other hand, there does not exist any $C>0$ such that $\mathcal{I}^{(2,2),\tau}=C\mathcal{J}^{(2,2),\tau}$ if $\tau=2$ which makes it difficult to compose test statistics like likelihood-ratio ones (see \citep{NU18}).
Hence in practice, $\tau$ sufficiently close to $2$ can be optimal, 
but it would be hard to discuss goodness-of-fit of models when $\tau=2$ exactly.
\end{remark}
}

\subsection{Test for noise detection}
We formulate the statistical hypothesis test such that
{\begin{align*}
	H_0: \Lambda_{\star}=O,\ H_1: \Lambda_{\star}\neq O.
\end{align*}}
We define $S_{t}:= \sum_{l=1}^{d}X_t^{(l)}$ and $\mathscr{S}_{ih_{n}} := \sum_{l=1}^{d}Y_{ih_{n}}^{(l)}$
and $\crotchet{X_t^{(1)},\ldots,X_t^{(d)},S_{t}}$ is also an ergodic diffusion.
Furthermore, 
{\small
\begin{align}
	Z_{n}:=\sqrt{\frac{2p_{n}}{3\sum_{j=1}^{k_{n}-2}\parens{\lm{\mathscr{S}}{j+1}-\lm{\mathscr{S}}{j}}^4}}\parens{\sum_{i=0}^{n-1}\parens{\mathscr{S}_{(i+1)h_{n}}-\mathscr{S}_{ih_{n}}}^2
		-\sum_{0\le 2i\le n-2}\parens{\mathscr{S}_{(2i+2)h_{n}}-\mathscr{S}_{2ih_{n}}}^2},
\end{align}}
where $\lm{\mathscr{S}}{j}:=\frac{1}{p_{n}}\sum_{i=0}^{p_{n}-1}\mathscr{S}_{j\Delta_{n}+ih_{n}}$, and consider the hypothesis test with rejection region $Z_{n}\ge z_{\alpha}$ 
where $z_{\alpha}$ is the {$(1-\alpha)$-quantile} of $N(0,1)$.

\begin{theorem}\label{thm321} Under $H_0$, (A1)-(A5), (AH) and $nh_{n}^2\to0$, 
	\begin{align*}
	Z_{n}\cl N(0,1).
	\end{align*}
\end{theorem}

\begin{theorem}\label{thm322} Under $H_1$, (A1)-(A5), (AH), (T1) and $nh_{n}^2\to0$, the test is consistent, i.e., for all $\alpha\in(0,1)$,
	\begin{align*}
	P(Z_{n}\ge z_{\alpha})\to 1.
	\end{align*}
\end{theorem}

{
\begin{remark}
The consistency shown above utilises the well-known fact in financial econometrics that the quadratic variation of process as diverges to infinity as observation frequency increases when the observation is contaminated by exogenous noise.
The first quadratic variation in the bracket of $Z_{n}$ composes of the entire observation, and the second one halves the number of samples with doubling the frequency from $h_{n}$ to $2h_{n}$.
It results in that the first quadratic variation divided by $2n$ converges to $\sum_{\ell_{1}}\sum_{\ell_{2}}\Lambda_{\star}^{\left(\ell_{1},\ell_{2}\right)}$ in probability while the second one divided by $2n$ converges to $(1/2)\sum_{\ell_{1}}\sum_{\ell_{2}}\Lambda_{\star}^{\left(\ell_{1},\ell_{2}\right)}$ in probability. This difference is sufficiently large so that $Z_{n}$ diverges in the sense of $P\left(Z_{n}\ge z_{\alpha}\right)\to1$ for any $\alpha\in(0,1)$. 
\end{remark}
}

{
In the next place, we concern approximation of the power in the finite sample scheme and consider the following sequence of alternatives
\begin{align*}
	\left(\Lambda_{\star}\right)_{n}=\frac{h_{n}}{\sqrt{n}}\mathfrak{M},
\end{align*}
where $\mathfrak{M}\ge0$ and $\delta:=\sum_{\ell_{1}}\sum_{\ell_{2}}\mathfrak{M}^{\left(\ell_{1},\ell_{2}\right)}>0$. Then we obtain the next result for approximation of power.

\begin{theorem}\label{thm323} 
	Under the sequence of the alternatives $\left\{\left(\Lambda_{\star}\right)_{n}\right\}$, (A1)-(A5), (AH) and $nh_{n}^{2}\to0$, the limit power of the test is $\Phi\left(\delta-z_{\alpha}\right)$.
\end{theorem}

\begin{remark}
(i) The order of the alternatives above might seem to be peculiar, but can be comprehended as follows: the convergence in distribution in Theorem 4 can result from that of the following quantity
\begin{align*}
	\sqrt{n}\left(\frac{1}{nh_{n}}\sum_{i=0}^{n-1}\left(S_{\left(i+1\right)h_{n}}-S_{ih_{n}}\right)^{2}-\frac{1}{nh_{n}}\sum_{0\le 2i\le n-2}\left(S_{\left(2i+2\right)h_{n}}-S_{2ih_{n}}\right)^{2}\right),
\end{align*}
which converges to normal distribution with mean 0, that is to say, is $O_{P}\left(1\right)$. If we replace $S_{ih_{n}}$ with $\mathscr{S}_{ih_{n}}$ and $\Lambda_{\star}$ with $\sum_{\ell_{1}}\sum_{\ell_{2}}\Lambda_{\star}^{\left(\ell_{1},\ell_{2}\right)}>0$ {is} fixed, then this quantity has the order $O_{P}\left(\sqrt{n}/h_{n}\right)$ as discussed in Lemma \ref{lem311} or Theorem \ref{thm322}. Here it is possible to understand the role of $h_{n}/\sqrt{n}$ in the sequence of the alternatives: it let the quantity remain to $O_{P}\left(1\right)$ and in fact $Z_{n}$ converges to normal distribution.

(ii) We should note that the dependency of $\Lambda_{\star}$ on $n$ is simply aimed at approximation of the power as claimed above; some literatures also let $\Lambda_{\star}$ depend on $n$ even in estimation framework while we have set it as a constant matrix except for Theorem \ref{thm323}.
\end{remark}

}

\bgroup
\def\arraystretch{1.5}
\setlength\tabcolsep{0.3cm}

\section{Example and simulation results}
\subsection{Case of small noise}
First of all, we consider the following 2-dimensional Ornstein-Uhlenbeck process
\begin{align}
\dop \crotchet{\begin{matrix}
	X_t^{\left(1\right)}\\
	X_t^{\left(2\right)}
	\end{matrix}}=\parens{\crotchet{\begin{matrix}
		\beta_1 & \beta_3\\
		\beta_2 & \beta_4
		\end{matrix}}\crotchet{\begin{matrix}
		X_t^{\left(1\right)}\\
		X_t^{\left(2\right)}
		\end{matrix}}+\crotchet{\begin{matrix}
		\beta_5\\
		\beta_6
		\end{matrix}}}\dop t
+ \crotchet{\begin{matrix}
	\alpha_1 & \alpha_2\\
	\alpha_2 & \alpha_3
	\end{matrix}}\dop w_t,\ \crotchet{\begin{matrix}
	X_0^{\left(1\right)}\\
	X_0^{\left(2\right)}
	\end{matrix}}=\crotchet{\begin{matrix}
	1\\
	1
	\end{matrix} }, 
\end{align}
where the true values of the diffusion parameters $\parens{\alpha_1^{\star}, \alpha_2^{\star},\alpha_3^{\star}}=\parens{1, 0.1,1}$ and the drift one $\parens{\beta_1^{\star}, \beta_2^{\star},\beta_3^{\star},\beta_4^{\star},\beta_5^{\star},\beta_6^{\star}}=\parens{-1,-0.1,-0.1,-1,1,1}$,
and the multivariate normal noise and the several levels of $\Lambda$ such that $\Lambda_{\star,-\infty}=O, \Lambda_{\star,-i}=10^{-i}I_2$
for all $i=\{4,5,6,7,8\}$. We check the performance of our estimator and the test constructed in Section 3, and compare our estimator (local mean method, LMM) with the estimator by LGA.
We show the setting and result of simulation in the following tables{; note that we simulate with $\tau=1.8$, $1.9$, and $2.0$ and denote empirical means {and root-mean-squared errors} in 1000 iterations without and with bracket $(\cdot)$ respectively}. With respect to the estimator for the noise variance, let us check the case of $\Lambda_{\star,-4}$. The empirical mean and {root-mean-squared errors} of $\hat{\Lambda}_{n}^{\left(1,1\right)}$ with $\Lambda_{\star}^{\left(1,1\right)}=10^{-4}$ are $1.32\times10^{-4}$ and $3.21\times 10^{-5}$; those of $\hat{\Lambda}_{n}^{\left(1,2\right)}$ with $\Lambda_{\star}^{\left(1,2\right)}=0$ are $6.29\times 10^{-6}$ and $6.31\times 10^{-6}$; and those of $\hat{\Lambda}_{n}^{\left(2,2\right)}$ with $\Lambda_{\star}^{\left(2,2\right)}=10^{-4}$ are $1.33\times10^{-4}$ and $3.25\times 10^{-5}$.

\begin{table}[h]
	\centering
	\caption{Setting in Section 4}
	\begin{tabular}{c|ccc}
		quantity & \multicolumn{3}{c}{approximation} \\\hline
		$n$ & \multicolumn{3}{c}{$10^6$}\\
		$h_{n}$ & \multicolumn{3}{c}{$6.309573\times 10^{-5}$}\\
		${nh_{n}}$ & \multicolumn{3}{c}{$63.09573$}\\
		$nh_{n}^{2}$ & \multicolumn{3}{c}{$0.003981072$}\\
		$\tau$ & $1.8$ & $1.9$ & $2.0$\\
		$p_{n}$ & $215$ & $162$ & $125$\\
		$k_{n}$ & $4651$ & $6172$ & $8000$\\
		$\Delta_{n}$ & $0.01356558$ & $0.01022151$ & $0.007886967$\\
		$k_{n}\Delta_{n}^{2}$ & $0.8559005$ & $0.6448459$ & $0.497634$\\
		iteration & \multicolumn{3}{c}{1000}
	\end{tabular}
\end{table}

{Firstly, we compare the simulation results for test statistics for all $\tau=1.8$, $1.9$ and $2.0$: see Table 2--4. It is observed that there are little differences among the results with the different value of tuning parameter $\tau$; hence we can conclude that $\tau$ does not matter at least in hypothesis testing proposed {in} section 3.2 as the results are proved for all $\tau\in\left(1,2\right]$.}

In the {second} place, we examine the performance of the diffusion estimators {(see Table 5--7)}. It can be seen that neither estimator with our method nor LGA dominates the {others} in terms of {root-mean-squared errors} where {$\Lambda_{\star}=\Lambda_{\star,-\infty}$, $\Lambda_{\star,-8}$, $\Lambda_{\star,-7}$}. Note that the powers of the test statistics are not large in these settings. It reflects that it is indifferent to choose either our estimators which are consistent even if there is no noise or the estimators with LGA by counting on the result of noise detection test which are asymptotically efficient if observation is not contaminated by noise. In contrast to these sizes of variance of noise, the results of simulation with the setting $\Lambda_{\star,-6}$, $\Lambda_{\star,-5}$ and $\Lambda_{\star,-4}$ shows that our estimators dominate the estimators with LGA in terms of {root-mean-squared errors}, and simultaneously the test for noise detection performs high power. {We should refer to the differences among LMM with different values of $\tau$; the larger $\tau$ clearly lessen the {root-mean-squared errors} since the influence of observational noise is not so large under these settings and our estimator $\hat{\alpha}_{n}$ has $\sqrt{k_{n}}$-consistency.}

We also see the same behaviour in estimation for drift parameters {(see Table 8--13)}. In this case, our estimators are dominant in all the setting of noise variance, but the performance of LGA estimators are close to them where $\Lambda_{\star,-\infty}$, $\Lambda_{\star,-8}$ and $\Lambda_{\star,-7}$. With the larger variance of noise, the estimators with local means method are far fine compared to the others. {Contrary to estimation for diffusion parameters, the difference of {root-mean-squared errors} among LMM with various value of $\tau$ is not obvious because our estimator $\hat{\beta}_{n}$ is $\sqrt{nh_{n}}$-consistent, i.e., the convergence rate does not depend on $\tau$.}

\begin{remark}
	\normalfont With these results,  we can see that the test works well as a criterion to select the estimation methods with local means and LGA: when adopting $H_0:\Lambda_{\star}=O$, we are essentially free to adopt either estimation; if rejecting $H_0$, we are strongly motivated to select our estimator.
\end{remark}


\subsection{Case of large noise}
Secondly we consider the problem with the identical setting as the previous one except for the variance of noise. We set the variance as $\Lambda_{\star}=I_2$ which is much larger than those in the previous subsection. In simulation, the empirical power of the test for noise detection is 1. We compare the estimation with our method (local mean method, LMM) and that with local Gaussian approximation (LGA) again.

Obviously all the estimators with LMM
{and all $\tau=1.8$, $1.9$, and $2.0$ }dominate the {those with LGA, diverging clearly (see Table 14)}. Moreover, the {root-mean-squared errors} of our estimators are close to those with settings of small noise 
in the subsection above. It shows that our estimator is robust even if the variance of noise is so large that we cannot imagine the undermined diffusion process seemingly.

{Remarkably, the {root-mean-squared errors} for $\hat{\alpha}_{n}$ decreases as $\tau$ declines contrary to what we see in the previous section 4.1. We can consider several causes for this tendency: the difference between the asymptotic variance with $\tau\in\left(1,2\right)$ and that with $\tau=2$ dependent on $\Lambda_{\star}$; the variety in $p_{n}$ denoting the number of samples in each local means, whose value results in the degree of diminishing the influence of noise. Anyway, we can observe the approach to tune $\tau$ should be problem-centric as mentioned in section 3.1. }

\begin{table}[!h]
	\caption{test statistics performance with small noise (section 4.1, $\tau=1.8$)}
	\centering
	\begin{tabular}{c|ccc}
		& ratio of $Z_{n}>z_{0.05}$ & ratio of $Z_{n}>z_{0.01}$ & ratio of $Z_{n}>z_{0.001}$ \\\hline 
		$\Lambda_{\star}=O$ & 0.050 & 0.004 & 0.001\\ 
		$\Lambda_{\star}=10^{-8}I_2$ & 0.062 & 0.010 & 0.001\\ 
		$\Lambda_{\star}=10^{-7}I_2$ & 0.256 & 0.088 & 0.017\\ 
		$\Lambda_{\star}=10^{-6}I_2$ & 1.000 & 1.000 & 1.000\\ 
		$\Lambda_{\star}=10^{-5}I_2$ & 1.000 & 1.000 & 1.000\\ 
		$\Lambda_{\star}=10^{-4}I_2$ & 1.000 & 1.000 & 1.000\\
	\end{tabular}
\end{table}

\begin{table}[!h]
	\caption{test statistics performance with small noise (section 4.1, $\tau=1.9$)}
	\centering
	\begin{tabular}{c|ccc}
		& ratio of $Z_{n}>z_{0.05}$ & ratio of $Z_{n}>z_{0.01}$ & ratio of $Z_{n}>z_{0.001}$ \\\hline 
		$\Lambda_{\star}=O$ & 0.051 & 0.007 & 0.002\\ 
		$\Lambda_{\star}=10^{-8}I_2$ & 0.063 & 0.011 & 0.002\\ 
		$\Lambda_{\star}=10^{-7}I_2$ & 0.263 & 0.087 & 0.017\\ 
		$\Lambda_{\star}=10^{-6}I_2$ & 1.000 & 1.000 & 1.000\\ 
		$\Lambda_{\star}=10^{-5}I_2$ & 1.000 & 1.000 & 1.000\\ 
		$\Lambda_{\star}=10^{-4}I_2$ & 1.000 & 1.000 & 1.000\\
	\end{tabular}
\end{table}

\begin{table}[!h]
	\caption{test statistics performance with small noise (section 4.1, $\tau=2.0$)}
	\centering
	\begin{tabular}{c|ccc}
		& ratio of $Z_{n}>z_{0.05}$ & ratio of $Z_{n}>z_{0.01}$ & ratio of $Z_{n}>z_{0.001}$ \\\hline 
		$\Lambda_{\star}=O$ & 0.050 & 0.008 & 0.002\\ 
		$\Lambda_{\star}=10^{-8}I_2$ & 0.065 & 0.010 & 0.002\\ 
		$\Lambda_{\star}=10^{-7}I_2$ & 0.257 & 0.088 & 0.016\\ 
		$\Lambda_{\star}=10^{-6}I_2$ & 1.000 & 1.000 & 1.000\\ 
		$\Lambda_{\star}=10^{-5}I_2$ & 1.000 & 1.000 & 1.000\\ 
		$\Lambda_{\star}=10^{-4}I_2$ & 1.000 & 1.000 & 1.000\\
	\end{tabular}
\end{table}

\begin{table}[h]
	\caption{comparison of estimator for $\alpha_{1}^{\star}=1$ with small noise (section 4.1).}
	{Topside values in cells denote empirical means; downside ones denote RMSE.}
	\centering
	\begin{tabular}{c|c|c|c|c}
		\multirow{2}{*}{$\Lambda_{\star}$} & \multicolumn{3}{c|}{$\hat{\alpha}_{1,\mathrm{LMM}}$ ($1$)} & \multirow{2}{*}{$\hat{\alpha}_{1,\mathrm{LGA}}$ ($1$)} \\ 
		& \multicolumn{1}{c}{$\tau=1.8$} & \multicolumn{1}{c}{$\tau=1.9$} & \multicolumn{1}{c|}{$\tau=2.0$} & \\ \hline
		
		\multirow{2}{*}{$O         $} & $0.996318$ & $0.997492$ & $0.998099$ & $1.003940$ \\ 
		& ( $0.0120$ ) & ( $0.0101$ ) & ( $0.0090$ ) & ( $0.0067$ )\\ \hline
		
		\multirow{2}{*}{$10^{-8}I_2$} & $0.996318$ & $0.997492$ & $0.998099$ & $1.004100$ \\
		& ( $0.0120$ ) & ( $0.0101$ ) & ( $0.0090$ ) & ( $0.0068$ ) \\ \hline
		
		\multirow{2}{*}{$10^{-7}I_2$} & $0.996318$ & $0.997492$ & $0.998099$ & $1.005534$\\ 
		& ( $0.0120$ ) & ( $0.0101$ ) & ( $0.0090$ ) & ( $0.0077$ )\\ \hline
		
		\multirow{2}{*}{$10^{-6}I_2$} & $0.996318$ & $0.997492$ & $0.998100$ & $1.019757$ \\
		& ( $0.0120$ ) & ( $0.0101$ ) & ( $0.0090$ ) & ( $0.0205$ )\\ \hline
		
		\multirow{2}{*}{$10^{-5}I_2$} & $0.996319$ & $0.997492$ & $0.998101$ & $1.152084$ \\
		& ( $0.0120$ ) & ( $0.0101$ ) & ( $0.0090$ ) & ( $0.1522$ )\\ \hline
		
		\multirow{2}{*}{$10^{-4}I_2$} & $0.996322$ & $0.997493$ & $0.998108$ & $2.045903$ \\
		& ( $0.0120$ ) & ( $0.0101$ ) & ( $0.0090$ ) & ( $1.0459$ )
		
	\end{tabular}
\end{table}

\begin{table}[h]
	\caption{comparison of estimator for $\alpha_{2}^{\star}=0.1$ with small noise (section 4.1).}
	{Topside values in cells denote empirical means; downside ones denote RMSE.}
	\centering
	\begin{tabular}{c|c|c|c|c}
		\multirow{2}{*}{$\Lambda_{\star}$} & \multicolumn{3}{c|}{$\hat{\alpha}_{2,\mathrm{LMM}}$ ($0.1$)} & \multirow{2}{*}{$\hat{\alpha}_{2,\mathrm{LGA}}$ ($0.1$)} \\ 
		& \multicolumn{1}{c}{$\tau=1.8$} & \multicolumn{1}{c}{$\tau=1.9$} & \multicolumn{1}{c|}{$\tau=2.0$} & \\ \hline
		
		\multirow{2}{*}{$O         $} & $0.094735$ & $0.095539$ & $0.096314$ & $0.098900$\\ 
		& ( $0.0087$ ) & ( $0.0073$ ) & ( $0.0064$ ) & ( $0.0067$ ) \\ \hline
		
		\multirow{2}{*}{$10^{-8}I_2$} & $0.094735$ & $0.095539$ & $0.096314$ & $0.098885$ \\
		& ( $0.0087$ ) & ( $0.0073$ ) & ( $0.0064$ ) & ( $0.0067$ ) \\ \hline
		
		\multirow{2}{*}{$10^{-7}I_2$} & $0.094735$ & $0.095539$ & $0.096314$ & $0.098745$ \\ 
		& ( $0.0087$ ) & ( $0.0073$ ) & ( $0.0064$ ) & ( $0.0067$ ) \\ \hline
		
		\multirow{2}{*}{$10^{-6}I_2$} & $0.094735$ & $0.095539$ & $0.096314$ & $0.097379$ \\
		& ( $0.0087$ ) & ( $0.0073$ ) & ( $0.0064$ ) & ( $0.0070$ )\\ \hline
		
		\multirow{2}{*}{$10^{-5}I_2$} & $0.094736$ & $0.095539$ & $0.096313$ & $0.086260$ \\
		& ( $0.0087$ ) & ( $0.0073$ ) & ( $0.0064$ ) & ( $0.0149$ )\\ \hline
		
		\multirow{2}{*}{$10^{-4}I_2$} & $0.094736$ & $0.095540$ & $0.096311$ & $0.048684$ \\
		& ( $0.0087$ ) & ( $0.0073$ ) & ( $0.0064$ ) & ( $0.0514$ )
		
	\end{tabular}
\end{table}

\begin{table}[h]
	\caption{comparison of estimator for $\alpha_{3}^{\star}=1$ with small noise (section 4.1).}
	{Topside values in cells denote empirical means; downside ones denote RMSE.}
	\centering
	\begin{tabular}{c|c|c|c|c}
		\multirow{2}{*}{$\Lambda_{\star}$} & \multicolumn{3}{c|}{$\hat{\alpha}_{3,\mathrm{LMM}}$ ($1$)} & \multirow{2}{*}{$\hat{\alpha}_{3,\mathrm{LGA}}$ ($1$)} \\ 
		& \multicolumn{1}{c}{$\tau=1.8$} & \multicolumn{1}{c}{$\tau=1.9$} & \multicolumn{1}{c|}{$\tau=2.0$} & \\ \hline
		
		\multirow{2}{*}{$O         $} & $0.997063$ & $0.997764$ & $0.998626$ & $1.010689$ \\ 
		& ( $0.0118$ ) & ( $0.0103$ ) & ( $0.0089$ ) & ( $0.0156$ )\\ \hline
		
		\multirow{2}{*}{$10^{-8}I_2$} & $0.997063$ & $0.997764$ & $0.998626$ & $1.010847$ \\
		& ( $0.0118$ ) & ( $0.0103$ ) & ( $0.0089$ ) & ( $0.0157$ ) \\ \hline
		
		\multirow{2}{*}{$10^{-7}I_2$} & $0.997063$ & $0.997764$ & $0.998626$ & $1.012273$ \\ 
		& ( $0.0118$ ) & ( $0.0103$ ) & ( $0.0089$ ) & ( $0.0167$ ) \\ \hline
		
		\multirow{2}{*}{$10^{-6}I_2$} & $0.997063$ & $0.997765$ & $0.998626$ & $1.026402$ \\
		& ( $0.0118$ ) & ( $0.0103$ ) & ( $0.0089$ ) & ( $0.0287$ ) \\ \hline
		
		\multirow{2}{*}{$10^{-5}I_2$} & $0.997064$ & $0.997766$ & $0.998627$ & $1.157960$ \\
		& ( $0.0118$ ) & ( $0.0103$ ) & ( $0.0089$ ) & ( $0.1583$ ) \\ \hline
		
		\multirow{2}{*}{$10^{-4}I_2$} & $0.997068$ & $0.997770$ & $0.998632$ & $2.049110$ \\
		& ( $0.0118$ ) & ( $0.0103$ ) & ( $0.0089$ ) & ( $1.0491$ ) 
		
	\end{tabular}
\end{table}

\begin{table}[h]
	\caption{comparison of estimator for $\beta_{1}^{\star}=-1$ with small noise (section 4.1).}
	{Topside values in cells denote empirical means; downside ones denote RMSE.}
	\centering
	\begin{tabular}{c|c|c|c|c}
		\multirow{2}{*}{$\Lambda_{\star}$} & \multicolumn{3}{c|}{$\hat{\beta}_{1,\mathrm{LMM}}$ ($-1$)} & \multirow{2}{*}{$\hat{\beta}_{1,\mathrm{LGA}}$ ($-1$)} \\ 
		& \multicolumn{1}{c}{$\tau=1.8$} & \multicolumn{1}{c}{$\tau=1.9$} & \multicolumn{1}{c|}{$\tau=2.0$} & \\ \hline
		
		\multirow{2}{*}{$O         $} & $-1.069994$ & $-1.073383$ & $-1.075441$ & $-1.097705$ \\ 
		& ( $0.1998$ ) & ( $0.2056$ ) & ( $0.1992$ ) & ( $0.2099$ ) \\ \hline
		
		\multirow{2}{*}{$10^{-8}I_2$} & $-1.069994$ & $-1.073383$ & $-1.075442$ & $-1.098047$ \\
		& ( $0.1998$ ) & ( $0.2056$ ) & ( $0.1992$ ) & ( $0.2101$ ) \\ \hline
		
		\multirow{2}{*}{$10^{-7}I_2$} & $-1.069994$ & $-1.073382$ & $-1.075443$ & $-1.101166$ \\ 
		& ( $0.1998$ ) & ( $0.2056$ ) & ( $0.1992$ ) & ( $0.2121$ ) \\ \hline
		
		\multirow{2}{*}{$10^{-6}I_2$} & $-1.069994$ & $-1.073384$ & $-1.075443$ & $-1.132415$ \\
		& ( $0.1998$ ) & ( $0.2056$ ) & ( $0.1992$ ) & ( $0.2330$ ) \\ \hline
		
		\multirow{2}{*}{$10^{-5}I_2$} & $-1.069994$ & $-1.073386$ & $-1.075444$ & $-1.446044$ \\
		& ( $0.1998$ ) & ( $0.2056$ ) & ( $0.1992$ ) & ( $0.5088$ ) \\ \hline
		
		\multirow{2}{*}{$10^{-4}I_2$} & $-1.069996$ & $-1.073397$ & $-1.075444$ & $-4.587123$ \\
		& ( $0.1998$ ) & ( $0.2056$ ) & ( $0.1992$ ) & ( $3.6698$ ) 
		
	\end{tabular}
\end{table}

\begin{table}[h]
	\caption{comparison of estimator for $\beta_{2}^{\star}=-0.1$ with small noise (section 4.1).}
	{Topside values in cells denote empirical means; downside ones denote RMSE.}
	\centering
	\begin{tabular}{c|c|c|c|c}
		\multirow{2}{*}{$\Lambda_{\star}$} & \multicolumn{3}{c|}{$\hat{\beta}_{2,\mathrm{LMM}}$ ($-0.1$)} & \multirow{2}{*}{$\hat{\beta}_{2,\mathrm{LGA}}$ ($-0.1$)} \\ 
		& \multicolumn{1}{c}{$\tau=1.8$} & \multicolumn{1}{c}{$\tau=1.9$} & \multicolumn{1}{c|}{$\tau=2.0$} & \\ \hline
		
		\multirow{2}{*}{$O         $} & $-0.093152$ & $-0.097752$ & $-0.100402$ & $-0.109554$ \\ 
		& ( $0.1947$ ) & ( $0.1964$ ) & ( $0.1954$ ) & ( $0.1995$ ) \\ \hline
		
		\multirow{2}{*}{$10^{-8}I_2$} & $-0.093152$ & $-0.097752$ & $-0.100402$ & $-0.109502$ \\
		& ( $0.1947$ ) & ( $0.1964$ ) & ( $0.1954$ ) & ( $0.1995$ )\\ \hline
		
		\multirow{2}{*}{$10^{-7}I_2$} & $-0.093152$ & $-0.097752$ & $-0.100403$ & $-0.109275$ \\ 
		& ( $0.1947$ ) & ( $0.1964$ ) & ( $0.1954$ ) & ( $0.1997$ )\\ \hline
		
		\multirow{2}{*}{$10^{-6}I_2$} & $-0.093153$ & $-0.097751$ & $-0.100404$ & $-0.106142$ \\
		& ( $0.1947$ ) & ( $0.1964$ ) & ( $0.1954$ ) & ( $0.2024$ )\\ \hline
		
		\multirow{2}{*}{$10^{-5}I_2$} & $-0.093152$ & $-0.097750$ & $-0.100404$ & $-0.074222$ \\
		& ( $0.1947$ ) & ( $0.1964$ ) & ( $0.1954$ ) & ( $0.2333$ )\\ \hline
		
		\multirow{2}{*}{$10^{-4}I_2$} & $-0.093151$ & $-0.097747$ & $-0.100403$ & $0.237936$ \\
		& ( $0.1948$ ) & ( $0.1964$ ) & ( $0.1954$ ) & ( $0.6836$ )
		
	\end{tabular}
\end{table}

\begin{table}[h]
	\caption{comparison of estimator for $\beta_{3}^{\star}=-0.1$ with small noise (section 4.1).}
	{Topside values in cells denote empirical means; downside ones denote RMSE.}
	\centering
	\begin{tabular}{c|c|c|c|c}
		\multirow{2}{*}{$\Lambda_{\star}$} & \multicolumn{3}{c|}{$\hat{\beta}_{3,\mathrm{LMM}}$ (-0.1)} & \multirow{2}{*}{$\hat{\beta}_{3,\mathrm{LGA}}$ (-0.1)} \\ 
		& \multicolumn{1}{c}{$\tau=1.8$} & \multicolumn{1}{c}{$\tau=1.9$} & \multicolumn{1}{c|}{$\tau=2.0$} & \\ \hline
		
		\multirow{2}{*}{$O         $} & $-0.095493$ & $-0.095852$ & $-0.097983$ & $-0.108282$ \\ 
		& ( $0.1913$ ) & ( $0.1931$ ) & ( $0.1931$ ) & ( $0.1948$ ) \\ \hline
		
		\multirow{2}{*}{$10^{-8}I_2$} & $-0.095493$ & $-0.095852$ & $-0.097983$ & $-0.108265$ \\
		& ( $0.1913$ ) & ( $0.1931$ ) & ( $0.1931$ ) & ( $0.1949$  ) \\ \hline
		
		\multirow{2}{*}{$10^{-7}I_2$} & $-0.095492$ & $-0.095852$ & $-0.097982$ & $-0.107915$ \\ 
		& ( $0.1913$ ) & ( $0.1931$ ) & ( $0.1931$ ) & ( $0.1952$ ) \\ \hline
		
		\multirow{2}{*}{$10^{-6}I_2$} & $-0.095493$ & $-0.095853$ & $-0.097982$ & $-0.104802$ \\
		& ( $0.1913$ ) & ( $0.1931$ ) & ( $0.1931$ ) & ( $0.1978$ ) \\ \hline
		
		\multirow{2}{*}{$10^{-5}I_2$} & $-0.095491$ & $-0.095851$ & $-0.097979$ & $-0.073318$ \\
		& ( $0.1913$ ) & ( $0.1931$ ) & ( $0.1931$ ) & ( $0.2288$ ) \\ \hline
		
		\multirow{2}{*}{$10^{-4}I_2$} & $-0.095487$ & $-0.095846$ & $-0.097979$ & $0.238196$ \\
		& ( $0.1913$ ) & ( $0.1931$ ) & ( $0.1931$ ) & ( $0.6808$ ) 
		
	\end{tabular}
\end{table}

\begin{table}[h]
	\caption{comparison of estimator for $\beta_{4}^{\star}=-1$ with small noise (section 4.1).}
	{Topside values in cells denote empirical means; downside ones denote RMSE.}
	\centering
	\begin{tabular}{c|c|c|c|c}
		\multirow{2}{*}{$\Lambda_{\star}$} & \multicolumn{3}{c|}{$\hat{\beta}_{4,\mathrm{LMM}}$ ($-1$)} & \multirow{2}{*}{$\hat{\beta}_{4,\mathrm{LGA}}$ ($-1$)} \\ 
		& \multicolumn{1}{c}{$\tau=1.8$} & \multicolumn{1}{c}{$\tau=1.9$} & \multicolumn{1}{c|}{$\tau=2.0$} & \\ \hline
		
		\multirow{2}{*}{$O         $} & $-1.055341$ & $-1.064300$ & $-1.070131$ & $-1.092219$ \\ 
		& ( $0.1951$ ) & ( $0.2009$ ) & ( $0.2020$ ) & ( $0.2156$ ) \\ \hline
		
		\multirow{2}{*}{$10^{-8}I_2$} & $-1.055341$ & $-1.064301$ & $-1.070132$ & $-1.092547$ \\
		& ( $0.1951$ ) & ( $0.2009$ ) & ( $0.2020$ ) & ( $0.2158$ ) \\ \hline
		
		\multirow{2}{*}{$10^{-7}I_2$} & $-1.055341$ & $-1.064301$ & $-1.070132$ & $-1.095701$ \\ 
		& ( $0.1951$ ) & ( $0.2009$ ) & ( $0.2020$ ) & ( $0.2177$ ) \\ \hline
		
		\multirow{2}{*}{$10^{-6}I_2$} & $-1.055342$ & $-1.064302$ & $-1.070132$ & $-1.126843$ \\
		& ( $0.1951$ ) & ( $0.2009$ ) & ( $0.2020$ ) & ( $0.2376$ ) \\ \hline
		
		\multirow{2}{*}{$10^{-5}I_2$} & $-1.055342$ & $-1.064303$ & $-1.070134$ & $-1.438228$ \\
		& ( $0.1951$ ) & ( $0.2009$ ) & ( $0.2020$ ) & ( $0.5073$ ) \\ \hline
		
		\multirow{2}{*}{$10^{-4}I_2$} & $-1.055344$ & $-1.064302$ & $-1.070141$ & $-4.559194$ \\
		& ( $0.1951$ ) & ( $0.2009$ ) & ( $0.2020$ ) & ( $3.6493$ ) 
		
	\end{tabular}
\end{table}

\begin{table}[h]
	\caption{comparison of estimator for $\beta_{5}^{\star}=1$ with small noise (section 4.1).}
	{Topside values in cells denote empirical means; downside ones denote RMSE.}
	\centering
	\begin{tabular}{c|c|c|c|c}
		\multirow{2}{*}{$\Lambda_{\star}$} & \multicolumn{3}{c|}{$\hat{\beta}_{5,\mathrm{LMM}}$ ($-1$)} & \multirow{2}{*}{$\hat{\beta}_{5,\mathrm{LGA}}$ ($-1$)} \\ 
		& \multicolumn{1}{c}{$\tau=1.8$} & \multicolumn{1}{c}{$\tau=1.9$} & \multicolumn{1}{c|}{$\tau=2.0$} & \\ \hline
		
		\multirow{2}{*}{$O         $} & $1.057539$ & $1.060114$ & $1.062063$ & $1.093013$\\ 
		& ( $0.2713$ ) & ( $0.2802$ ) & ( $0.2751$ ) & ( $0.2863$ ) \\ \hline
		
		\multirow{2}{*}{$10^{-8}I_2$} & $1.057539$ & $1.060114$ & $1.062065$ & $1.093291$ \\
		& ( $0.2713$ ) & ( $0.2802$ ) & ( $0.2751$ ) & ( $0.2865$ ) \\ \hline
		
		\multirow{2}{*}{$10^{-7}I_2$} & $1.057539$ & $1.060114$ & $1.062065$ & $1.095840$ \\ 
		& ( $0.2713$ ) & ( $0.2802$ ) & ( $0.2751$ ) & ( $0.2880$ ) \\ \hline
		
		\multirow{2}{*}{$10^{-6}I_2$} & $1.057539$ & $1.060115$ & $1.062066$ & $1.121342$ \\
		& ( $0.2713$ ) & ( $0.2802$ ) & ( $0.2751$ ) & ( $0.3029$ ) \\ \hline
		
		\multirow{2}{*}{$10^{-5}I_2$} & $1.057537$ & $1.060116$ & $1.062063$ & $1.376932$ \\
		& ( $0.2713$ ) & ( $0.2802$ ) & ( $0.2751$ ) & ( $0.5083$ ) \\ \hline
		
		\multirow{2}{*}{$10^{-4}I_2$} & $1.057537$ & $1.060123$ & $1.062064$ & $3.936379$ \\
		& ( $0.2713$ ) & ( $0.2802$ ) & ( $0.2751$ ) & ( $3.1035$ ) \\
		
	\end{tabular}
\end{table}

\begin{table}[h]
	\caption{comparison of estimator for $\beta_{6}^{\star}=1$ with small noise (section 4.1).}
	{Topside values in cells denote empirical means; downside ones denote RMSE.}
	\centering
	\begin{tabular}{c|c|c|c|c}
		\multirow{2}{*}{$\Lambda_{\star}$} & \multicolumn{3}{c|}{$\hat{\beta}_{6,\mathrm{LMM}}$ ($1$)} & \multirow{2}{*}{$\hat{\beta}_{6,\mathrm{LGA}}$ ($1$)} \\ 
		& \multicolumn{1}{c}{$\tau=1.8$} & \multicolumn{1}{c}{$\tau=1.9$} & \multicolumn{1}{c|}{$\tau=2.0$} & \\ \hline
		
		\multirow{2}{*}{$O         $} & $1.046920$ & $1.055245$ & $1.063341$ & $1.076187$ \\ 
		& ( $0.2749$ ) & ( $0.2784$ ) & ( $0.2816$ ) & ( $0.2889$ ) \\ \hline
		
		\multirow{2}{*}{$10^{-8}I_2$} & $1.046920$ & $1.055245$ & $1.063341$ & $1.076418$ \\
		& ( $0.2749$ ) & ( $0.2784$ ) & ( $0.2816$ ) & ( $0.2891$ ) \\ \hline
		
		\multirow{2}{*}{$10^{-7}I_2$} & $1.046920$ & $1.055246$ & $1.063342$ & $1.079131$ \\ 
		& ( $0.2749$ ) & ( $0.2784$ ) & ( $0.2816$ ) & ( $0.2904$ ) \\ \hline
		
		\multirow{2}{*}{$10^{-6}I_2$} & $1.046920$ & $1.055246$ & $1.063341$ & $1.104471$ \\
		& ( $0.2749$ ) & ( $0.2784$ ) & ( $0.2816$ ) & ( $0.3043$ ) \\ \hline
		
		\multirow{2}{*}{$10^{-5}I_2$} & $1.046920$ & $1.055248$ & $1.063345$ & $1.358819$ \\
		& ( $0.2749$ ) & ( $0.2784$ ) & ( $0.2816$ ) & ( $0.5021$ ) \\ \hline
		
		\multirow{2}{*}{$10^{-4}I_2$} & $1.046918$ & $1.055244$ & $1.063347$ & $3.911360$ \\
		& ( $0.2749$ ) & ( $0.2784$ ) & ( $0.2816$ ) & ( $3.0893$ ) 
		
	\end{tabular}
\end{table}

\begin{table}[!h]
	\caption{Estimators with large noise (section 4.2). Topside values in cells denote empirical means; 
		downside ones denote RMSE.}
	\centering
	\begin{tabular}{c|c|c|c|c|c}
		& & \multicolumn{3}{c|}{LMM} & \multirow{2}{*}{LGA}\\
		
		& true value & $\tau=1.8$ & $\tau=1.9$ & $\tau=2.0$ & \\\hline
		
		\multirow{2}{*}{$\mathrm{\hat{\Lambda}^{\left(1,1\right)}}$} & \multirow{2}{*}{$\mathrm{1}$} & \multicolumn{4}{c}{ $1.000106$ }  \\ 
		& & \multicolumn{4}{c}{( $0.001678$ )} \\\hline
		
		\multirow{2}{*}{$\mathrm{\hat{\Lambda}^{\left(1,2\right)}}$} & \multirow{2}{*}{$\mathrm{0}$} & \multicolumn{4}{c}{ $1.796561\times10^{-5}$ }  \\ 
		& & \multicolumn{4}{c}{( $0.001226$ )} \\\hline
		
		\multirow{2}{*}{$\mathrm{\hat{\Lambda}^{\left(2,2\right)}}$} & \multirow{2}{*}{$\mathrm{1}$} & \multicolumn{4}{c}{ $1.000030$ }  \\ 
		& & \multicolumn{4}{c}{( $0.001826$ )} \\\hline
		
		\multirow{2}{*}{$\hat{\alpha}_1$} & \multirow{2}{*}{$1$}
		& $0.996805$ & $1.000907$ & $1.017547$ & $178.068993$ \\
		& & ( $0.0217$ ) & ( $0.0267$ ) & ( $0.0396$ )  & ( $177.0733$ ) \\\hline
		
		\multirow{2}{*}{$\hat{\alpha}_2$} & \multirow{2}{*}{$0.1$}
		& $0.098530$ & $0.098489$ & $0.097489$ & $0.313443$ \\
		& & ( $0.0155$ ) & ( $0.0196$ ) & ( $0.0250$ )  & ( $9.9737$ ) \\\hline
		
		\multirow{2}{*}{$\hat{\alpha}_3$} & \multirow{2}{*}{$1$}
		& $0.996391$ & $1.000373$ & $1.018511$ & $177.962836$ \\
		& & ( $0.0211$ ) & ( $0.0271$ ) & ( $0.0383$ ) & ( $176.9738$ ) \\\hline

		\multirow{2}{*}{$\hat{\beta}_1$} & \multirow{2}{*}{$-1$}
		& $-1.050453$ & $-1.05002$ & $-1.048604$ & $3.51\times10^7$ \\
		& & ( $0.1919$ ) & ( $0.1927$ ) & ( $0.1908$ ) & ( $1.11\times 10^9$ )\\\hline
		
		\multirow{2}{*}{$\hat{\beta}_2$} & \multirow{2}{*}{$-0.1$}
		& $-0.103636$ & $-0.105248$ & $-0.106449$ & $1.37\times10^8$ \\
		& & ( $0.1931$ ) & ( $0.1955$ ) & ( $0.1987$ ) & ( $4.34\times 10^9$ ) \\\hline
		
		\multirow{2}{*}{$\hat{\beta}_3$} & \multirow{2}{*}{$-0.1$}
		& $-0.084856$ & $-0.086533$ & $-0.087520$ & $1.27\times10^8$ \\
		& & ( $0.1908$ ) & ( $0.1913$ ) & ( $0.1920$ ) & ( $4.03\times 10^9$ ) \\\hline
		
		\multirow{2}{*}{$\hat{\beta}_4$} & \multirow{2}{*}{$-1$}
		& $-1.046981$ & $-1.04722$ & $-1.045288$ & $-4.57\times10^7$  \\
		& & ( $0.1891$ ) & ( $0.1916$ ) & ( $0.1923$ ) & ( $1.44\times 10^9$ )\\\hline
		
		\multirow{2}{*}{$\hat{\beta}_5$} & \multirow{2}{*}{$1$}
		& $1.032952$ & $1.034185$ & $1.033792$ & $3.89\times10^6$  \\
		& & ( $0.2719$ ) & ( $0.2725$ ) & ( $0.2715$ ) & ( $1.23\times 10^8$ ) \\\hline
		
		\multirow{2}{*}{$\hat{\beta}_6$} & \multirow{2}{*}{$1$}
		& $1.041460$ & $1.043000$ & $1.042522$ & $1.57\times10^7$  \\
		& & ( $0.2716$ ) & ( $0.2744$ ) & ( $0.2753$ ) & ( $4.96\times 10^8$ ) \\
		
	\end{tabular}
\end{table}

\clearpage

\section{Real data analysis: Met Data of NWTC}

We analyse the wind data called Met Data provided by National Wind Technology Center in United States. Met Data is the dataset recording several quantities related to wind such as velocity, speed, and temperature at the towers named M2, M4 and M5 with recording facilities in some altitudes. 
{
The statistical modelling for wind data with stochastic differential equations have gathered interest: \citep{BB16} fits Cox-Ingersoll-Ross model to wind speed data and reports that the CIR model overwhelms other methods such as {static models} in terms of prediction precision; \citep{VASKV16} models both power generation by windmills and power demand with Ornstein-Uhlenbeck processes {with} some preprocessing and examines their performances for practical purposes.}
We especially focus {on} the 2-dimensional data with 0.05-second resolution representing wind velocity labelled Sonic x and Sonic y (119M) at the M5 tower, from 00:00:00 on 1st July, 2017 to 20:00:00 on 5th July, 2017. For {details}, see \citep{NWTC}. 
We fit the 2-dimensional Ornstein-Uhlenbeck process such that
\begin{align}
\dop \crotchet{\begin{matrix}
	X_t\\
	Y_t
	\end{matrix}}=\parens{\crotchet{\begin{matrix}
		\beta_1 & \beta_3\\
		\beta_2 & \beta_4
		\end{matrix}}\crotchet{\begin{matrix}
		X_t\\
		Y_t
		\end{matrix}}+\crotchet{\begin{matrix}
		\beta_5\\
		\beta_6
		\end{matrix}}}\dop t
+ \crotchet{\begin{matrix}
	\alpha_1 & \alpha_2\\
	\alpha_2 & \alpha_3
	\end{matrix}}\dop w_t,\ \crotchet{\begin{matrix}
	X_0\\
	Y_0
	\end{matrix}}=\crotchet{\begin{matrix}
	x_0\\
	y_0
	\end{matrix}},
\end{align}
{where $\left(x_0, y_0\right)$ is the initial value.}
We summarise some relevant quantities as follows.
\begin{table}[h]
	\begin{center}
		\caption{Relevant Quantities in Section 5}
		\begin{tabular}{c|c}
			quantity & approximation \\\hline
			$n$ & 8352000\\
			$h_{n}$ & $6.944444\times 10^{-6}$\\
			${nh_{n}}$ &  58\\
			$nh_{n}^{2}$ & $4.027778\times 10^{-4}$\\
			$\tau$ & $1.9$\\
			$p_{n}$ & $518$\\
			$k_{n}$ & $16123$\\
			$\Delta_{n}$ & $0.003597222$\\
			$k_{n}\Delta_{n}^{2}$ & $0.2086317$
		\end{tabular}
	\end{center}
\end{table}
We have taken 2 hours as the time unit and fixed {$\tau=1.9$}.

Our test for noise detection results in $Z=514.0674$ and $p<10^{-16}$; therefore, for any $\alpha\ge 10^{-16}$, the alternative hypothesis $\Lambda\neq O$ is adopted. Our estimator gives the fitting such that
\begin{align}
\dop \crotchet{\begin{matrix}
	X_t\\
	Y_t
	\end{matrix}}
&=\parens{\crotchet{\begin{matrix}
		-3.04 & -0.18\\
		-0.25 & -4.19
		\end{matrix}}
	\crotchet{\begin{matrix}
		X_t\\
		Y_t
		\end{matrix}}
	+ \crotchet{\begin{matrix}
		2.95\\
		-2.24
		\end{matrix}}}\dop t+\crotchet{\begin{matrix}
	12.18 & -0.25\\
	-0.25 & 11.48
	\end{matrix}}\dop w_t,
\end{align}
and the estimation of the noise variance
\begin{align}
\hat{\Lambda}_{n}=\crotchet{\begin{matrix}
	6.67\times 10^{-3} & 3.75\times 10^{-5}\\
	3.75\times 10^{-5} & 6.79\times 10^{-3}\\
	\end{matrix}};
\end{align} and the diffusion fitting with LGA method which is {asymptotically} efficient if $\Lambda=O$ gives
\begin{align}
\dop \crotchet{\begin{matrix}
	X_t\\
	Y_t
	\end{matrix}}
&=\parens{\crotchet{\begin{matrix}
		-67.53 &  -9.29\\
		-10.37 & -104.45
		\end{matrix}}
	\crotchet{\begin{matrix}
		X_t\\
		Y_t
		\end{matrix}}
	+ \crotchet{\begin{matrix}
		63.27\\
		-50.24
		\end{matrix}}}\dop t+\crotchet{\begin{matrix}
	43.82 & 0.13\\
	0.13 & 44.22
	\end{matrix}}\dop w_t.
\end{align}
What we see here is that these estimators give obviously different values with the data. If $\Lambda=O$, then we should have the reasonably similar values to each other. Since we have already obtained the result $\Lambda\neq O$, there is no reason to regard the latter estimate should be adopted.

\section{Conclusion}

Our contribution is composed of three parts:
proofs of the asymptotic properties
for adaptive estimation of diffusion-plus-noise models and noise detection test, the simulation study of the 
asymptotic results
developed above, 
and the real data analysis showing that there exists situation where 
the proposed method
should be adopted. 
The adaptive ML-type estimators 
introduced in Section 3.1 are so simple that it is only necessary for us to optimise the quasi-likelihood functions quite similar to the Gaussian likelihood after we compute the much simpler estimator for the variance of noise. The test for noise detection is nonparametric; therefore, there is no need to set any model structure or quantities other than $\tau$ and time unit. We could check our methodology works well in simulation section regardless of the size of variance of noise : the estimators could perform better than or at least as well as LGA method. The real data analysis shows that our methodology is certainly {helpful} to analyse some high-frequency data.

As mentioned in the introduction, high-frequency setting of observation can relax some complexness and difficulty of state-space modelling. It results in a simple and unified methodology for both linear and nonlinear models since we can write the quasi-likelihood functions whether the model is linear or not. The innovation in state-space model can be dependent on the latent process itself; therefore, we can let the processes be with fat-tail which has been regarded as a stylised fact in financial econometrics these decades. The increase in amount of real-time data seen today will continue 
{at so brisk pace} that diffusion-plus-noise modelling with these desirable properties will gain more usefulness in wide range of situations.

\begin{figure}[h]
	\centering
	\includegraphics[bb=0 0 720 480,width=12cm]{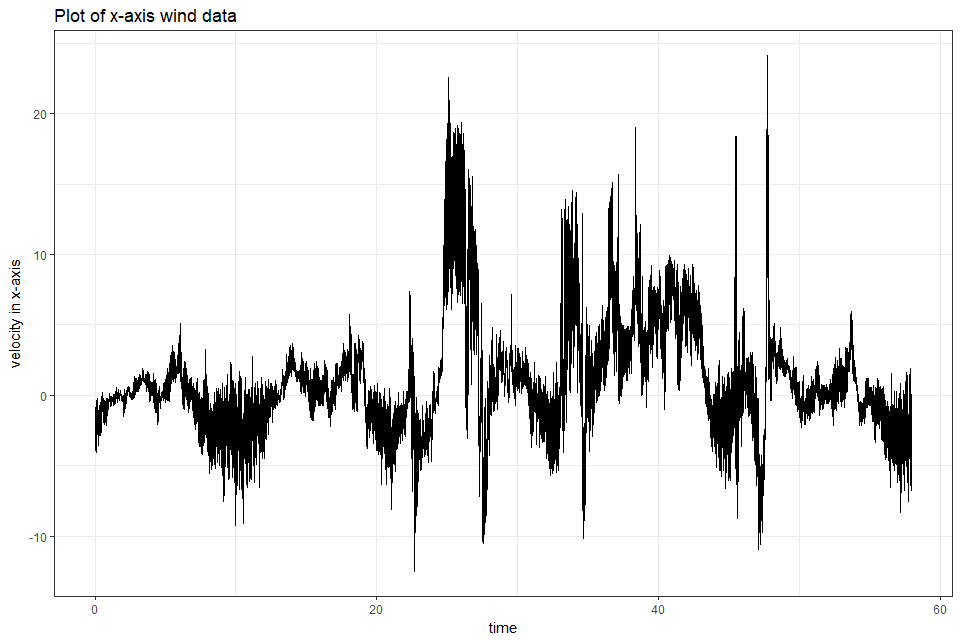}
	\caption{plot of x-axis, Met Data}
	\centering
	\includegraphics[bb=0 0 720 480,width=12cm]{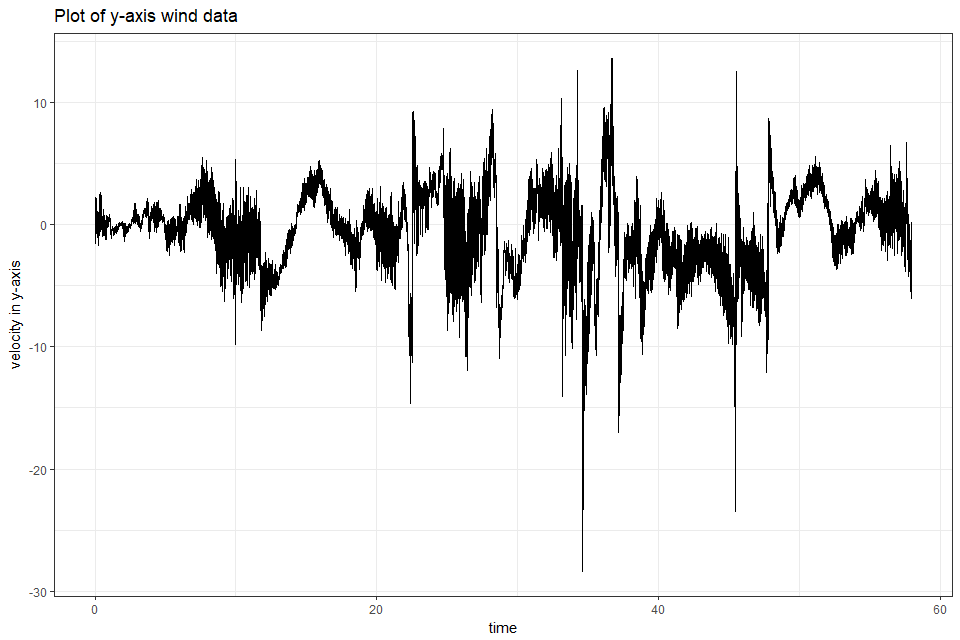}
	\caption{plot of y-axis, Met Data}
\end{figure}

\clearpage

%% file: proof_journal_revised_black.tex
\section{Proofs}
We give the proofs of the main theorems discussed above and some preliminary ones.
Some of them are also discussed in \citep{NU17} with details.

We set some notations which only appear in the proof section.
\begin{enumerate}
	\item Let us denote some $\sigma$-fields such that $\mathcal{G}_t:=\sigma(w_s;s\le t,x_0)$, $\mathcal{G}_j^n:=\mathcal{G}_{j\Delta_{n}}$, $\mathcal{A}_j^n:=\sigma(\epsilon_{ih_{n}};i\le jp_{n}-1)$, $\mathcal{H}_j^n:=\mathcal{G}_j^n\vee \mathcal{A}_j^n$.
	\item We define the following $\Re^r$-valued random variables which appear in the expansion:
	\begin{align*}
	\zeta_{j+1,n}&=\frac{1}{p_{n}}\sum_{i=0}^{p_{n}-1}\int_{j\Delta_{n}+ih_{n}}^{(j+1)\Delta_{n}}\dop w_s,\ 
	\zeta_{j+2,n}'=\frac{1}{p_{n}}\sum_{i=0}^{p_{n}-1}\int_{(j+1)\Delta_{n}}^{(j+1)\Delta_{n}+ih_{n}}\dop w_s.
	\end{align*}
	\item $I_{j,k,n}:=I_{j,k}=[j\Delta_{n}+kh_{n},j\Delta_{n}+(k+1)h_{n}),\ j=0,\ldots, k_{n}-1,\ k=0,\ldots,p_{n}-1$.
	\item We set the following empirical functionals:
	\begin{align*}
	\bar{M}_{n}(f(\cdot,\vartheta))&:=\frac{1}{k_{n}}\sum_{j=0}^{k_{n}-1}f(\lm{Y}{j},\vartheta),\\
	\bar{D}_{n}(f(\cdot,\vartheta))&:=\frac{1}{k_{n}\Delta_{n}}\sum_{j=1}^{k_{n}-2}f(\lm{Y}{j-1},\vartheta)
	\parens{\lm{Y}{j+1}-\lm{Y}{j}-\Delta_{n}b(\lm{Y}{j-1})},\\
	\bar{Q}_{n}(B(\cdot,\vartheta))&=\frac{1}{k_{n}\Delta_{n}}\sum_{j=1}^{k_{n}-2}\ip{B(\lm{Y}{j-1},\vartheta)}{\parens{\lm{Y}{j+1}-\lm{Y}{j}}^{\otimes2}}.
	\end{align*}
	\item 
	Let us define $D_{j,n}:=\frac{1}{2p_{n}}\sum_{i=0}^{p_{n}-1}\parens{\parens{Y_{j\Delta_{n}+(i+1)h_{n}}-Y_{j\Delta_{n}+ih_{n}}}^{\otimes2}-\Lambda_{\star}}$, 
	and for $1\le l_2\le l_1\le d$,$D_{n}^{(l_1,l_2)}:= \frac{1}{k_{n}}\sum_{j=0}^{k_{n}-1}D_{j,n}^{(l_1,l_2)}$ and $D_{n}:=\left[D_{n}^{(1,1)},D_{n}^{(2,1)},\ldots, D_{n}^{(d,d-1)},D_{n}^{(d,d)}\right]=\vech\parens{\hat{\Lambda}_{n}-\Lambda_{\star}}$.
	\item We denote
	\begin{align*}
	&\left\{B_{\kappa}\left|\kappa=1,\ldots,m_1,\ B_{\kappa}=(B_{\kappa}^{(j_1,j_2)})_{j_1,j_2}\right.\right\},\\
	&\left\{f_{\lambda}\left|\lambda=1,\ldots,m_2,\ f_{\lambda}=(f^{(1)}_{\lambda},\ldots,f^{(d)}_{\lambda})\right.\right\},
	\end{align*}
	which are sequences of $\Re^d\otimes \Re^d$-valued functions and $\Re^d$-valued ones such that the components of themselves and their derivatives with respect to $x$ are polynomial growth functions for all $\kappa$ and $\lambda$.
	\item 
	Let us define
	\begin{align*}
	&\left\{B_{\kappa,n}(x)\left|\kappa=1,\ldots,m_1,\ B_{\kappa,n}=(B_{\kappa,n}^{(j_1,j_2)})_{j_1,j_2}\right.\right\},
	\end{align*}
	which is a family of sequences of the functions such that the components of the functions and their derivatives with respect to $x$ are polynomial growth functions and there exist a $\Re$-valued sequence $\tuborg{v_{n}}_{n}$ s.t. $v_{n}\to0$ and $C>0$ such that for all $x\in\Re^d$ and for the sequence $\tuborg{B_{\kappa}}$ discussed above,
	\begin{align*}
	\sum_{\kappa=1}^{m_1}\norm{B_{\kappa,n}(x)-B_{\kappa}(x)}
	\le v_{n}\parens{1+\norm{x}^C}.
	\end{align*}
	\item Denote
	\begin{align*}
	W^{(\tau)}\parens{\tuborg{B_{\kappa}}_{\kappa},\tuborg{f_{\lambda}}_{\lambda}}:=\crotchet{\begin{matrix}
		W_1 & O & O\\
		O & W_2^{\tau}\parens{\tuborg{B_{\kappa}}_{\kappa}} & O \\
		O & O & W_3\parens{\tuborg{f_{\lambda}}_{\lambda}}
		\end{matrix}}.
	\end{align*}
\end{enumerate}

\subsection{Conditional expectation of supremum}
The following two propositions are multidimensional extensions of Proposition 5.1 and Proposition A in \citep{Gl00} respectively.
\begin{proposition}\label{pro711}
	 Under (A1), for all $k\ge 1$, there exists a constant $C(k)$ such that for all $t\ge 0$,
	\begin{align*}
	\CE{\sup_{s\in[t,t+1]}\norm{X_s}^k}{\mathcal{G}_t}\le C(k)(1+\norm{X_t}^k).
	\end{align*}
\end{proposition}

\begin{proposition}\label{pro712} Under (A1) and for a function $f$ whose components are in $\mathcal{C}^1(\Re^d)$, assume that there exists $C>0$ such that
\begin{align*}
	\norm{f'(x)}&\le C(1+\norm{x})^C.
\end{align*}
Then for any $k\in\mathbf{N}$,
\begin{align*}
	\CE{\sup_{s\in[j\Delta_{n},(j+1)\Delta_{n}]}\norm{f(X_s)-f(X_{j\Delta_{n}})}^k}{\mathcal{G}_j^n}\le C(k)\Delta_{n}^{k/2}\parens{1+\norm{X_{j\Delta_{n}}}^{C(k)}}.
\end{align*}
Especially for $f(x)=x$,
\begin{align*}
	\CE{\sup_{s\in[j\Delta_{n},(j+1)\Delta_{n}]}\norm{X_s-X_{j\Delta_{n}}}^k}{\mathcal{G}_j^n}\le C(k)\Delta_{n}^{k/2}\parens{1+\norm{X_{j\Delta_{n}}}^{k}}.
\end{align*}
\end{proposition}
The next proposition summarises some results useful for computation.
\begin{proposition}\label{pro713}
	Under (A1), for all $t_3\ge t_2\ge t_1\ge 0$ where there exists a constant $C$ such that $t_3-t_1\le C$ and $l\ge 2$, we have
	\begin{align*}
	\mathrm{(i)}\ &\sup_{s_1,s_2\in[t_1,t_2]}\norm{\CE{b(X_{s_1})-b(X_{s_2})}{\mathcal{G}_{t_1}}}\le C(t_2-t_1)
	\parens{1+\norm{X_{t_1}}^3},\\
	\mathrm{(ii)}\ &\sup_{s_1,s_2\in[t_1,t_2]}\norm{\CE{a(X_{s_1})-a(X_{s_2})}{\mathcal{G}_{t_1}}}\le C(t_2-t_1)
	\parens{1+\norm{X_{t_1}}^3},\\
	\mathrm{(iii)}\ &\norm{\CE{\int_{t_2}^{t_3}\parens{b(X_s)-b(X_{t_2})}\dop s}{\mathcal{G}_{t_1}}}\le C(t_3-t_2)^2\parens{1+\CE{\norm{X_{t_2}}^6}{\mathcal{G}_{t_1}}}^{1/2},\\
	\mathrm{(iv)}\ &\CE{\norm{\int_{t_2}^{t_3}\parens{b(X_s)-b(X_{t_2})}\dop s}^l}{\mathcal{G}_{t_1}}
	\le C(l)(t_3-t_2)^{3l/2}\parens{1+\CE{\norm{X_{t_2}}^{2l}}{\mathcal{G}_{t_1}}},\\
	\mathrm{(v)}\ &\CE{\norm{\int_{t_1}^{t_2}\parens{\int_{t_1}^{s}\parens{a(X_{u})-a(X_{t_1})}\dop w_u}\dop s}^l
	}{\mathcal{G}_{t_1}}\le C(l)\parens{t_2-t_1}^{2l}\parens{1+\norm{X_{t_1}}^{2l}}.
	\end{align*} 
\end{proposition}

\begin{proof} (i), (ii): Let $L$ be the infinitesimal generator of the diffusion process. Since Ito-Taylor expansion, for all $s\in [t_1,t_2]$,
	\begin{align*}
	\CE{b(X_s)}{\mathcal{G}_{t_1}}=b(X_{t_1})+\int_{t_1}^{s}\CE{Lb(X_u)}{\mathcal{G}_{t_1}}\dop u,
	\end{align*}
	and the second term has the evaluation
	\begin{align*}
	\sup_{s\in[t_1,t_2]}\norm{\int_{t_1}^{s}\CE{Lb(X_u)}{\mathcal{G}_{t_1}}\dop u}\le C\left(t_{2}-t_{1}\right)\parens{1+\norm{X_{t_1}}^3}.
	\end{align*}
	Therefore, we have (ii)
	and identical revaluation holds for (ii).\\
	(iii): Using (i) and H\"{o}lder's inequality, we have the result.\\
	(iv): Because of Proposition \ref*{pro712} and H\"{o}lder's inequality, We obtain the proof.\\
	(v): For convexity, we have
	\begin{align*}
		&\CE{\norm{\int_{t_1}^{t_2}\parens{\int_{t_1}^{s}\parens{a(X_{u})-a(X_{t_1})}\dop w_u}\dop s}^l
		}{\mathcal{G}_{t_1}}\\
		&\le C(l)\sum_{i=1}^{d}\sum_{j=1}^{r}\CE{\abs{\int_{t_1}^{t_2}\parens{\int_{t_1}^{s}
					\parens{a^{(i,j)}(X_{u})-a^{(i,j)}(X_{t_1})}
					\dop w_u^{(j)}\dop s}^{l}}}{\mathcal{G}_{t_1}}.
	\end{align*}
	H\"{o}lder's inequality, Fubini's theorem, BDG theorem and Proposition \ref*{pro712} give the result.
\end{proof}

\subsection{Propositions for ergodicity and evaluations of expectation}

The next result is a multivariate version of \citep{K97} or \citep{Gl06} using Proposition \ref{pro711}.
\begin{lemma}\label{lem721}
	Assume (A1)-(A3) hold. Let $f$ be a function in $\mathcal{C}^1(\Re^d\times\Xi)$ and assume that $f$, the components of $\partial_{x} f$ and $\nabla_{\vartheta} f$ are polynomial growth functions uniformly in $\vartheta\in\Xi$. 
	Then the following convergence holds:
	\begin{align*}
		\frac{1}{k_{n}}\sum_{j=0}^{k_{n}-1}f(X_{j\Delta_{n}},\vartheta)\cp \nu_0\parens{f(\cdot,\vartheta)}\ \text{uniformly in }\vartheta.
	\end{align*}
\end{lemma}

\subsection{Characteristics of local means}

The following propositions, lemmas and corollary are multidimensional extensions of those in \citep{Gl00} and \citep{Fa14}.

\begin{lemma}\label{lem732}
	$\zeta_{j+1,n}$ and $\zeta_{j+1,n}'$ are $\mathcal{G}_{j+1}^n$-measurable, independent of $\mathcal{G}_{j}^n$ and Gaussian. These random variables have the following decomposition:
	\begin{align*}
	\zeta_{j+1,n}&=\frac{1}{p_{n}}\sum_{k=0}^{p_{n}-1}(k_{n}+1)\int_{I_{j,k}}\dop w_t,\\
	\zeta_{j+1,n}'&=\frac{1}{p_{n}}\sum_{k=0}^{p_{n}-1}(p_{n}-1-k)\int_{I_{j,k}}\dop w_t.
	\end{align*}
	In addition, the evaluation of the following conditional expectations holds:
	\begin{align*}
	\CE{\zeta_{j,n}}{\mathcal{G}_j^n}=\CE{\zeta_{j+1,n}'}{\mathcal{G}_j^n}&=\mathbf{0},\\
	\CE{\zeta_{j+1,n}\parens{\zeta_{j+1,n}}^T}{\mathcal{G}_j^n}&=m_{n}\Delta_{n}I_r,\\
	\CE{\zeta_{j+1,n}'\parens{\zeta_{j+1,n}'}^T}{\mathcal{G}_j^n}
	&=m_{n}'\Delta_{n}I_r,\\
	\CE{\zeta_{j+1,n}\parens{\zeta_{j+1,n}'}^T}{\mathcal{G}_j^n}&=\chi_{n}\Delta_{n}I_r,
	\end{align*}
	where $m_{n}=\parens{\frac{1}{3}+\frac{1}{2p_{n}}+\frac{1}{6p_{n}^2}}$, $m_{n}'=\parens{\frac{1}{3}-\frac{1}{2p_{n}}+\frac{1}{6p_{n}^2}}$ and $\chi_{n}=\frac{1}{6}\parens{1-\frac{1}{p_{n}^2}}$.
\end{lemma}
\noindent For the proof, see Lemma 8.2 in \citep{Fa14} and extend it to multidimensional discussion.

\begin{proposition}\label{cor736}
	Under (A1), (AH), assume the component of the function $f$ on $\Re^d\times\Xi$, $\partial_x f$ and $\partial_x^2 f$ are polynomial growth functions uniformly in $\vartheta\in\Xi$. Then there exists $C>0$ such that for all $j\le k_{n}-1$ and $\vartheta\in\Xi$,
	 \begin{align*}
	 	&\norm{\CE{f(\lm{Y}{j}{},\vartheta)-f(X_{j\Delta_{n}},\vartheta)}{\mathcal{H}_j^n}}\le C\Delta_{n}\parens{1+\norm{X_{j\Delta_{n}}}^{C}}.
	 \end{align*}
	 Moreover, for all $l\ge 2$,
	 \begin{align*}
	 	&\CE{\norm{f(\lm{Y}{j}{},\vartheta)-f(X_{j\Delta_{n}},\vartheta)}^l}{\mathcal{H}_j^n}\le C(l)\Delta_{n}^{l/2}\parens{1+\norm{X_{j\Delta_{n}}}^{C(l)}}.
	 \end{align*}
\end{proposition}
\noindent The proof is almost identical to that of Corollary 3.3 in \citep{Fa14} except for dimension, but it does not influence the evaluation.

\begin{proposition}\label{pro737}
	Under (A1) and (AH),
	\begin{align*}
		\lm{Y}{j+1}-\lm{Y}{j}-\Delta_{n}b(\lm{Y}{j})=a(X_{j\Delta_{n}})\parens{\zeta_{j+1,n}+\zeta_{j+2,n}'}+e_{j,n}
		+\Lambda_{\star}^{1/2}\parens{\lm{\epsilon}{j+1}-\lm{\epsilon}{j}},
	\end{align*}
	where $e_{j,n}$ is a $\mathcal{H}_{j+2}^n$-measurable random variable such that there exists $C>0$ and $C(l)>0$ for all
	 $l\ge2$ satisfying the inequalities
	\begin{align*}
		&\norm{\CE{e_{j,n}}{\mathcal{H}_j^n}}\le 
		C\Delta_{n}^2\parens{1+\norm{X_{j\Delta_{n}}}^{5}},\\
		&\CE{\norm{e_{j,n}}^l}{\mathcal{H}_j^n}\le C(l)\Delta_{n}^l\parens{1+\norm{X_{j\Delta_{n}}}^{3l}},\\
		&\norm{\CE{e_{j,n}\parens{\zeta_{j+1,n}}^T}{\mathcal{H}_j^n}}+\norm{\CE{e_{j,n}\parens{\zeta_{j+2,n}'}^T}{\mathcal{H}_j^n}}\le C\Delta_{n}^2\parens{1+\norm{X_{j\Delta_{n}}}^{3}}.
	\end{align*}
\end{proposition}

\noindent For the proof, see that of Proposition 3.4 in \citep{Fa14} and extend the discussion to multidimensional one.

\begin{corollary}\label{cor738}
	Under (A1) and (AH),
	\begin{align*}
	\lm{Y}{j+1}-\lm{Y}{j}-\Delta_{n}b(X_{j\Delta_{n}})=a(X_{j\Delta_{n}})\parens{\zeta_{j+1,n}+\zeta_{j+2,n}'}+e_{j,n}
	+\Lambda_{\star}^{1/2}\parens{\lm{\epsilon}{j+1}-\lm{\epsilon}{j}},
	\end{align*}
	where $e_{j,n}$ is a $\mathcal{H}_{j+2}^n$-measurable random variable such that there exists $C>0$ and $C(l)>0$ for all
	$l\ge2$ satisfying the inequalities
	\begin{align*}
	&\norm{\CE{e_{j,n}}{\mathcal{H}_j^n}}\le 
	C\Delta_{n}^2\parens{1+\norm{X_{j\Delta_{n}}}^{5}},\\
	&\CE{\norm{e_{j,n}}^l}{\mathcal{H}_j^n}\le C(l)\Delta_{n}^l\parens{1+\norm{X_{j\Delta_{n}}}^{3l}},\\
	&\norm{\CE{e_{j,n}\parens{\zeta_{j+1,n}}^T}{\mathcal{H}_j^n}}+\norm{\CE{e_{j,n}\parens{\zeta_{j+2,n}'}^T}{\mathcal{H}_j^n}}\le C\Delta_{n}^2\parens{1+\norm{X_{j\Delta_{n}}}^{3}}.
	\end{align*}
\end{corollary}

\begin{proof}
	It is enough to see $\Delta_{n}b(\lm{Y}{j})-\Delta_{n}b(X_{j\Delta_{n}})$ satisfies the evaluation for $e_{j,n}$. Corollary \ref*{cor736} and
	Proposition \ref*{pro737} give
	\begin{align*}
		&\norm{\CE{\Delta_{n}b(\lm{Y}{j})-\Delta_{n}b(X_{j\Delta_{n}})}{\mathcal{H}_j^n}}\le 
		C\Delta_{n}^2\parens{1+\norm{X_{j\Delta_{n}}}^{5}},\\
		&\CE{\norm{\Delta_{n}b(\lm{Y}{j})-\Delta_{n}b(X_{j\Delta_{n}})}^l}{\mathcal{H}_j^n}
		\le C(l)\Delta_{n}^l\parens{1+\norm{X_{j\Delta_{n}}}^{3l}}.
	\end{align*}
	With respect to the third evaluation, H\"{o}lder's inequality verifies the result.
\end{proof}

The following lemma summarises some useful evaluations for computation.
\begin{lemma}\label{lem739}
	Assume $f$ is a function whose components are in $\mathcal{C}^2(\Re^d\times\Xi)$ and the components of $f$ and $\partial_xf$ are polynomial growth functions in $\vartheta\in\Xi$. In addition, $g$ denotes a function whose components are in $\mathcal{C}^2(\Re^d)$ and that the components of $g$ and $\partial_xg$ are polynomial growth functions.  Under (A1), (A3), (A4) and (AH), the following uniform evaluation holds:
	\begin{align*}
	\mathrm{(i) }&^{\forall} l_1,l_2\in\mathbf{N}_0,\ 
	\sup_{j,n}\E{\sup_{\vartheta\in\Xi}\norm{f(\lm{Y}{j-1},\vartheta)}^{l_1}\parens{1+\norm{X_{j\Delta_{n}}}}^{l_2}}\le C(l_1,l_2),\\
	\mathrm{(ii) }&^{\forall} l\in \mathbf{N},\ ^{\forall} j\le k_{n}-2,\ \E{\norm{\zeta_{j+1,n}+\zeta_{j+2,n}'}^{l}}\le C(l)\Delta_{n}^{l/2},\\
	\mathrm{(iii) }&^{\forall} l\in\mathbf{N},\ ^{\forall} j\le k_{n}-1,\ \E{\norm{g(X_{(j+1)\Delta_{n}})-g(X_{j\Delta_{n}})}^{l}}\le C(l)\Delta_{n}^{l/2},\\
	\mathrm{(iv) }&^{\forall} l\in\mathbf{N},\ ^{\forall} j\le k_{n}-1,\ \E{\norm{g(\lm{Y}{j})-g(X_{j\Delta_{n}})}^{l}}\le C(l)\Delta_{n}^{l/2},\\
	\mathrm{(v) }&^{\forall} l\in\mathbf{N},\ ^{\forall} j\le k_{n}-2,\  \E{\norm{g(\lm{Y}{j+1})-g(\lm{Y}{j})}^{l}}\le C(l)\Delta_{n}^{l/2},\\
	\mathrm{(vi) }&^{\forall} l\in\mathbf{N},\ ^{\forall} j\le k_{n}-2,\ \E{\norm{e_{j,n}}^l}\le C(l)\Delta_{n}^l,\\
	\mathrm{(vii) }&^{\forall} l\in\mathbf{N},\ 
	\sup_{j,n}\parens{\frac{\mathbf{E}\left[\norm{\lm{\epsilon}{j}}^l\right]}{\Delta_{n}^{l/2}}}\le C(l).
	\end{align*}
\end{lemma}

\begin{proof}Simple computations and the results above lead to the proof.
\end{proof}

\subsection{Uniform law of large numbers}

The following propositions and theorems are multidimensional version of \citep{Fa14}.

\begin{proposition}\label{pro741}
	Assume $f$ is a function in $\mathcal{C}^2(\Re^d\times\Xi)$ and $f$, the components of $\partial_xf$, $\partial_x^2f$ and $\partial_{\vartheta} f$ are polynomial growth functions uniformly in $\vartheta\in\Xi$. Under (A1)-(A4), (AH),
	\begin{align*}
		\bar{M}_{n}(f(\cdot,\vartheta)) \cp \nu_0(f(\cdot,\vartheta))
		\text{ uniformly in }\vartheta.
	\end{align*}
\end{proposition}

\noindent The proof is almost same as Proposition 4.1 in \citep{Fa14}.

\begin{theorem}\label{thm742}
	Assume $f=\parens{f^1,\ldots,f^d}$ is a function in $\mathcal{C}^2(\Re^d\times\Xi; \Re^d)$ and the components of $f$, $\partial_xf$, $\partial_x^2f$ and $\partial_{\vartheta} f$ are polynomial growth functions uniformly in $\vartheta\in\Xi$. Under (A1)-(A4), (AH),
	\begin{align*}
		\bar{D}_{n}(f(\cdot,\vartheta)) \cp 0\text{ uniformly in }\vartheta.
	\end{align*}
\end{theorem}

\begin{proof}
	We define the following random variables:
	\begin{align*}
		V_j^n(\vartheta)&:=f(\lm{Y}{j-1},\vartheta)\parens{\lm{Y}{j+1}-\lm{Y}{j}-\Delta_{n}b(\lm{Y}{j})},\\
		\tilde{D}_{n}(f(\cdot,\vartheta))&:=\frac{1}{k_{n}\Delta_{n}}\sum_{j=1}^{k_{n}-2}V_j^n(\vartheta)
	\end{align*}
	and then
	\begin{align*}
		\bar{D}_{n}(f(\cdot,\vartheta))
		&=\tilde{D}_{n}(\cdot,\vartheta)+\frac{1}{k_{n}}\sum_{j=1}^{k_{n}-2}f(\lm{Y}{j-1},\vartheta)
		\parens{b(\lm{Y}{j})-b(\lm{Y}{j-1})}.
	\end{align*}
	Hence it is enough to see the uniform convergences in probability of the first term and the second one in the right hand side. \\
	
	In the first place, we consider the first term of the right hand side above. We can decompose the sum of $V_j^n(\vartheta)$ as follows:
	\begin{align*}
		\sum_{j=1}^{k_{n}-2}V_j^n(\vartheta)=\sum_{1\le 3j\le k_{n}-2}V_{3j}^n(\vartheta)+\sum_{1\le 3j+1\le k_{n}-2}V_{3j+1}^n(\vartheta)
		+\sum_{1\le 3j+2\le k_{n}-2}V_{3j+2}^n(\vartheta).
	\end{align*}
	To simplify notations, we only consider the first term of the right hand side and the other terms have the identical evaluation.
	Let us define the following random variables:
	\begin{align*}
		v_{3j,n}^{(1)}(\vartheta)&:=f(\lm{Y}{3j-1},\vartheta)a(X_{3j\Delta_{n}})\parens{\zeta_{3j+1,n}+\zeta_{3j+2,n}'},\\
		v_{3j,n}^{(2)}(\vartheta)&:=f(\lm{Y}{3j-1},\vartheta)\Lambda_{\star}^{1/2}\parens{\lm{\epsilon}{3j+1}-\lm{\epsilon}{3j}},\\
		v_{3j,n}^{(3)}(\vartheta)&:=f(\lm{Y}{3j-1},\vartheta)e_{3j,n},
	\end{align*}
	and recall Proposition \ref*{pro737} which states
	\begin{align*}
		\lm{Y}{j+1}-\lm{Y}{j}-\Delta_{n}b(\lm{Y}{j})=a(X_{j\Delta_{n}})\parens{\zeta_{j+1,n}+\zeta_{j+2,n}'}+e_{j,n}
		+\Lambda_{\star}^{1/2}\parens{\lm{\epsilon}{j+1}-\lm{\epsilon}{j}}.
	\end{align*}
	Therefore we have
	\begin{align*}
		V_{3j}^{n}(\vartheta)=v_{3j,n}^{(1)}(\vartheta)+v_{3j,n}^{(2)}(\vartheta)+v_{3j,n}^{(3)}(\vartheta).
	\end{align*}
	First of all, the pointwise convergence to 0 for all $\vartheta$ and we abbreviate $f(\cdot,\vartheta)$ as $f(\cdot)$. Since $V_{3j}^n$ is $\mathcal{H}_{3j+2}^{n}$-measurable and 
	hence $\mathcal{H}_{3j+3}^n$-measurable, the sequence of random variables $\tuborg{V_{3j}}_{1\le 3j\le k_{n}-2}$ are
	 $\tuborg{\mathcal{H}_{3j+3}^{n}}_{1\le 3j\le k_{n}-2}$-adopted, and hence it is enough to see
	\begin{align*}
		\frac{1}{k_{n}\Delta_{n}}\sum_{1\le 3j\le k_{n}-2}\CE{V_{3j}^{n}}{\mathcal{H}_{3(j-1)+3}^n}=\frac{1}{k_{n}\Delta_{n}}\sum_{1\le 3j\le k_{n}-2}\CE{V_{3j}^{n}}{\mathcal{H}_{3j}^n}&\cp 0,\\
		\frac{1}{k_{n}^2\Delta_{n}^2}\sum_{1\le 3j\le k_{n}-2}\CE{\parens{V_{3j}^{n}}^2}{\mathcal{H}_{3j}^n}&\cp 0
	\end{align*}
	because of Lemma 9 in \citep{GeJ93}. It is quite routine to show it because of Proposition \ref{pro737}.

	Next, we consider the uniform convergence in probability of $\tilde{D}_{n}(f(\cdot,\vartheta))$. Let us define
	\begin{align*}
		S_{n}^{(l)}(\vartheta):=\frac{1}{k_{n}\Delta_{n}}\sum_{1\le 3j\le k_{n}-2}v_{3j,n}^{(l)}(\vartheta),\ l=1,2,3.
	\end{align*}
	We will see for all $l$, $S_{n}^{(l)}(\vartheta)$ uniformly converges to 0 in probability. Firstly we examine $S_{n}^{(3)}$: Lemma \ref*{lem739} gives
	\begin{align*}
		\E{\sup_{\vartheta\in\Xi}\abs{\nabla_{\vartheta} v_{3j,n}^{(3)}(\vartheta)}}
		\le C\Delta_{n}.
	\end{align*}
	Hence we obtain
	\begin{align*}
		\sup_{n\in\mathbf{N}}\E{\sup_{\vartheta\in\Xi}\abs{\nabla_{\vartheta} S_{n}^{(3)}(\vartheta)}}
		\le C
		<\infty.
	\end{align*}
	Therefore it holds
	\begin{align*}
		S_{n}^{(3)}(\vartheta)\cp0\text{ uniformly in }\vartheta
	\end{align*}
	as the discussion in \citep{K97} or Proposition A1 in \citep{Gl06}.
	
	For $S_{n}^{(l)},\ l=1,2$ we see the following inequalities hold (see \citep{IH81}): there exist $C>0$ and $\kappa>\dim \Xi$ such that
	\begin{align*}
		\mathrm{(1) }&^{\forall} \vartheta\in\Xi,\ ^{\forall} n\in\mathbf{N},\ \E{\abs{S_{n}^{(l)}(\vartheta)}^{\kappa}}\le C,\\
		\mathrm{(2) }&^{\forall} \vartheta,\vartheta'\in\Xi,\ ^{\forall} n\in\mathbf{N},\ 
		\E{\abs{S_{n}^{(l)}(\vartheta)-S_{n}^{(l)}(\vartheta')}^{\kappa}}\le C\norm{\vartheta-\vartheta'}^{\kappa}.
	\end{align*}
	Assume $\kappa=2k,\ k\in\mathbf{N}$ in the following discussion.
	
	For $l=1$, Burkholder's inequality and Lemma \ref*{lem739}gives that for all $\kappa$, there exists $C=C(\kappa)$ such that
	\begin{align*}
		\E{\abs{S_{n}^{(1)}(\vartheta)}^{\kappa}}
		&\le \frac{C(\kappa)}{\parens{k_{n}\Delta_{n}}^{\kappa}}k_{n}^{\kappa/2-1}
		\sum_{1\le 3j\le k_{n}-2}\E{\abs{v_{3j,n}^{(1)}(\vartheta)}^{\kappa}}\\
		&\le \frac{C(\kappa)}{\parens{k_{n}\Delta_{n}}^{\kappa}}k_{n}^{\kappa/2-1}
		\sum_{1\le 3j\le k_{n}-2}C(\kappa)\Delta_{n}^{\kappa/2}\\
		&\le \frac{C(\kappa)}{\parens{k_{n}\Delta_{n}}^{\kappa/2}}.
	\end{align*}
	With respect to $\E{\abs{S_{n}^{(1)}(\vartheta)-S_{n}^{(1)}(\vartheta')}^{\kappa}}$, identically
	\begin{align*}
		\E{\abs{S_{n}^{(1)}(\vartheta)-S_{n}^{(1)}(\vartheta')}^{\kappa}}
		&\le \frac{C(\kappa)}{\parens{k_{n}\Delta_{n}}^{\kappa}}k_{n}^{\kappa/2-1}
		\sum_{1\le 3j\le k_{n}-2}\E{\abs{v_{3j,n}^{(1)}(\vartheta)-v_{3j,n}^{(1)}(\vartheta')}^{\kappa}}\\
		&\le \frac{C(\kappa)}{\parens{k_{n}\Delta_{n}}^{\kappa}}k_{n}^{\kappa/2-1}
		\sum_{1\le 3j\le k_{n}-2}C(\kappa)\Delta_{n}^{\kappa/2}\norm{\vartheta-\vartheta'}^{\kappa}\\
		&\le \frac{C(\kappa)}{\parens{k_{n}\Delta_{n}}^{\kappa/2}}\norm{\vartheta-\vartheta'}^{\kappa}.
	\end{align*}
	This result gives uniform convergence in probability of $S_{n}^{(1)}$.
	The identical evaluation holds for $S_{n}^{(2)}$ and then these lead to uniform convergence of $\tilde{D}_{n}\parens{\vartheta}$. 
	
	Finally we check the uniform convergence in probability of the following random variable,
	\begin{align*}
		\frac{1}{k_{n}}\sum_{1\le 3j\le k_{n}-2}f(\lm{Y}{j-1},\vartheta)\parens{b(\lm{Y}{j})-b(\lm{Y}{j-1})}.
	\end{align*}
	By Lemma \ref*{lem739}, it is easily shown that
	\begin{align*}
		\E{\sup_{\vartheta\in\Xi}\abs{\frac{1}{k_{n}}\sum_{1\le 3j\le k_{n}-2}
				f(\lm{Y}{j-1},\vartheta)\parens{b(\lm{Y}{j})-b(\lm{Y}{j-1})}}}\to0.
	\end{align*}
	It completes the proof.
\end{proof}

\begin{theorem}\label{thm743}
	Assume $B=\parens{B^{\left(i,j\right)}}_{i,j}$ is a function in $\mathcal{C}^2(\Re^d\times\Xi;\Re^d\otimes\Re^d)$ and the components of $B$, $\partial_xB$, $\partial_x^2B$ and $\partial_{\vartheta} B$ are polynomial growth functions uniformly in $\vartheta\in\Xi$. Under (A1)-(A4), (AH), if $\tau\in(1,2]$,
	\begin{align*}
	\bar{Q}_{n}(B(\cdot,\vartheta))
	\cp \frac{2}{3}\nu_0\parens{\ip{B(\cdot,\vartheta)}{A^{\tau}(\cdot,\alpha^{\star},\Lambda_{\star})}}\text{ uniformly in }\vartheta.
	\end{align*}
\end{theorem}

\begin{proof}
	We define $q_{j,n}(\vartheta):=\sum_{i=1}^{6}q_{j,n}^{(i)}(\vartheta)$ where
	\begin{align*}
		q_{j,n}^{(1)}(\vartheta)&:=\ip{B(\lm{Y}{j-1},\vartheta)}{\parens{a(X_{j\Delta_{n}})\parens{\zeta_{j+1,n}+\zeta_{j+2,n}'}}^{\otimes 2}},\\
		q_{j,n}^{(2)}(\vartheta)&:=\ip{B(\lm{Y}{j-1},\vartheta)}{\parens{\Lambda_{\star}^{1/2}\parens{\lm{\epsilon}{j+1}-\lm{\epsilon}{j}}}^{\otimes2}},\\
		q_{j,n}^{(3)}(\vartheta)&:=\ip{B(\lm{Y}{j-1},\vartheta)}{\parens{\Delta_{n}b(\lm{Y}{j})+e_{j,n}}^{\otimes2}},\\
		q_{j,n}^{(4)}(\vartheta)&:=\ip{B(\lm{Y}{j-1},\vartheta)}{a(X_{j\Delta_{n}})\parens{\zeta_{j+1,n}+\zeta_{j+2,n}'}\parens{\Lambda_{\star}^{1/2}\parens{\lm{\epsilon}{j+1}-\lm{\epsilon}{j}}}^T},\\
		&\qquad+\ip{B(\lm{Y}{j-1},\vartheta)}{\parens{\Lambda_{\star}^{1/2}\parens{\lm{\epsilon}{j+1}-\lm{\epsilon}{j}}}\parens{\zeta_{j+1,n}+\zeta_{j+2,n}'}^Ta(X_{j\Delta_{n}})^T},\\
		q_{j,n}^{(5)}(\vartheta)&:=\ip{B(\lm{Y}{j-1},\vartheta)}{a(X_{j\Delta_{n}})\parens{\zeta_{j+1,n}+\zeta_{j+2,n}'}\parens{\Delta_{n}b(\lm{Y}{j})+e_{j,n}}^T}\\
		&\qquad+\ip{B(\lm{Y}{j-1},\vartheta)}{\parens{\Delta_{n}b(\lm{Y}{j})+e_{j,n}}
			\parens{\zeta_{j+1,n}+\zeta_{j+2,n}'}^Ta(X_{j\Delta_{n}})^T},\\
		q_{j,n}^{(6)}(\vartheta)&:=\ip{B(\lm{Y}{j-1},\vartheta)}{\parens{\Delta_{n}b(\lm{Y}{j})+e_{j,n}}\parens{\Lambda_{\star}^{1/2}\parens{\lm{\epsilon}{j+1}-\lm{\epsilon}{j}}}^T},\\
		&\qquad+\ip{B(\lm{Y}{j-1},\vartheta)}{\parens{\Lambda_{\star}^{1/2}\parens{\lm{\epsilon}{j+1}-\lm{\epsilon}{j}}}\parens{\Delta_{n}b(\lm{Y}{j})+e_{j,n}}^T},
	\end{align*}
	and	then for Proposition \ref*{pro737} we obtain
	\begin{align*}
		\bar{Q}_{n}(B(\cdot,\vartheta))=\frac{1}{k_{n}\Delta_{n}}\sum_{j=1}^{k_{n}-2}q_{j,n}(\vartheta).
	\end{align*}
	We examine the following quantities which divide the summation into three parts: for $l=0,1,2$,
	\begin{align*}
		T_{l,n}^{(i)}(\vartheta):=\frac{1}{k_{n}\Delta_{n}}\sum_{1\le 3j+l\le k_{n}-2}q_{3j+l,n}^{(i)}(\vartheta)\text{ for }i=1,\ldots,6.
	\end{align*}
	Firstly we see the pointwise-convergence in probability with respect to $\vartheta$.\\
	
	We examine $T_{0,n}^{(1)}(\vartheta)$ and consider to show convergence in probability with Lemma 9 in \citep{GeJ93}. Lemma \ref*{lem732} gives
	\begin{align*}
		\CE{q_{3j,n}^{(1)}(\vartheta)}{\mathcal{H}_{3j}^{n}}
		=\parens{\frac{2}{3}+\frac{1}{3p_{n}^2}}\Delta_{n}\ip{B(\lm{Y}{3j-1},\vartheta)}{A(X_{3j\Delta_{n}})}.
	\end{align*}
	Note that Lemma \ref*{lem721} gives
	\begin{align*}
		&\frac{1}{k_{n}\Delta_{n}}\sum_{1\le 3j\le k_{n}-2}\parens{\frac{2}{3}+\frac{1}{3p_{n}^2}}
		\Delta_{n}\ip{B(X_{3j\Delta_{n}},\vartheta)}{A(X_{3j\Delta_{n}})}
		\cp\frac{1}{3}\times\frac{2}{3}\times\nu_0\parens{\ip{B(\cdot,\vartheta)}{A(\cdot)}},
	\end{align*}
	and we can obtain
	\begin{align*}
		\frac{1}{k_{n}}\sum_{1\le 3j\le k_{n}-2}\parens{\frac{2}{3}+\frac{1}{3p_{n}^2}}
		\ip{\parens{B(X_{3j\Delta_{n}},\vartheta)-B(X_{(3j-1)\Delta_{n}},\vartheta)}}{A(X_{3j\Delta_{n}})}&\cp0\\
		\frac{1}{k_{n}}\sum_{1\le 3j\le k_{n}-2}\parens{\frac{2}{3}+\frac{1}{3p_{n}^2}}
		\ip{\parens{B(\lm{Y}{3j-1},\vartheta)-B(X_{(3j-1)\Delta_{n}},\vartheta)}}{A(X_{3j\Delta_{n}})}&\cp0
	\end{align*}
	because of Lemma \ref*{lem739}; hence we have
	\begin{align*}
		\frac{1}{k_{n}\Delta_{n}}\sum_{1\le 3j\le k_{n}-2}\CE{q_{3j,n}^{(1)}(\vartheta)}{\mathcal{H}_{3j}^{n}}\cp \frac{2}{9}\nu_0\parens{\ip{B(\cdot,\vartheta)}{A(\cdot)}}.
	\end{align*}
	Next we have
	\begin{align*}
		\CE{\abs{q_{3j,n}^{(1)}(\vartheta)}^2}{\mathcal{H}_{3j}^{n}}
		\le C\Delta_{n}^2\norm{B(\lm{Y}{3j-1},\vartheta)}^2\norm{a(X_{3j\Delta_{n}})}^4
	\end{align*}
	because of Lemma \ref*{lem739}, and then we obtain
	\begin{align*}
		\E{\abs{\frac{1}{k_{n}^2\Delta_{n}^2}\sum_{1\le 3j\le k_{n}-2}\CE{\abs{q_{3j,n}^{(1)}(\vartheta)}^2}{\mathcal{H}_{3j}^{n}}}}
		\to0
	\end{align*}
	also because of Lemma \ref*{lem739}. Therefore, Lemma 9 in \citep{GeJ93} gives
	\begin{align*}
		T_{0,n}^{(1)}(\vartheta)\cp\frac{2}{9}\nu_0\parens{\ip{B(\cdot,\vartheta)}{A(\cdot)}}
	\end{align*}
	and identical convergences for $T_{1,n}^{(1)}(\vartheta)$ and $T_{2,n}^{(1)}(\vartheta)$ can be given. Hence
	\begin{align*}
		T_{0,n}^{(1)}(\vartheta)+T_{1,n}^{(1)}(\vartheta)+T_{2,n}^{(1)}(\vartheta)
		\cp\frac{2}{3}\nu_0\parens{\ip{B(\cdot,\vartheta)}{A(\cdot)}}.
	\end{align*}
	
	For $T_{0,n}^{(2)}(\vartheta)$, we also see the pointwise convergence in probability with \citep{GeJ93}. Firstly,
	\begin{align*}
		\CE{q_{3j,n}^{(2)}(\vartheta)}{\mathcal{H}_{3j}^n}
		=\frac{2}{p_{n}}\ip{B(\lm{Y}{j-1},\vartheta)}{\Lambda_{\star}}
	\end{align*}
	and then Proposition \ref*{pro741} leads to
	\begin{align*}
		\frac{1}{k_{n}\Delta_{n}}\sum_{1\le 3j\le k_{n}-2}\CE{q_{3j,n}^{(2)}(\vartheta)}{\mathcal{H}_{3j}^n}
		\cp\begin{cases}
		0 & \text{ if }\tau\in(1,2)\\
		\frac{2}{3}\nu_0\parens{\ip{B(\cdot,\vartheta)}{\Lambda_{\star}}} & \text{ if }\tau=2
		\end{cases}
	\end{align*}
	because
	\begin{align*}
		\frac{1}{p_{n}\Delta_{n}}=\frac{1}{p_{n}^2h_{n}}=\frac{1}{p_{n}^{2-\tau}}\to\begin{cases}
			0 & \text{ if }\tau\in(1,2)\\
			1 & \text{ if }\tau=2.
		\end{cases}
	\end{align*}
	Because of Lemma \ref*{lem739}, we also easily have the conditional second moment evaluation such that
	\begin{align*}
		\E{\abs{\frac{1}{k_{n}^2\Delta_{n}^2}\sum_{1\le 3j\le k_{n}-2}\CE{\abs{q_{3j,n}^{(2)}(\vartheta)}^2}{\mathcal{H}_{3j}^n}}}
		\to0,
	\end{align*}
	therefore this $L^1$ convergence verifies convergence in probability and by Lemma 9 in \citep{GeJ93},
	\begin{align*}
		T_{0,n}^{(2)}(\vartheta)\cp\begin{cases}
		0 & \text{ if }\tau\in(1,2)\\
		\frac{2}{3}\nu_0\parens{\ip{B(\cdot,\vartheta)}{\Lambda_{\star}}} & \text{ if }\tau=2
		\end{cases}
	\end{align*}
	and then
	\begin{align*}
		T_{0,n}^{(2)}(\vartheta)+T_{1,n}^{(2)}(\vartheta)+T_{2,n}^{(2)}(\vartheta)\cp\begin{cases}
		0 & \text{ if }\tau\in(1,2)\\
		2\nu_0\parens{\ip{B(\cdot,\vartheta)}{\Lambda_{\star}}} & \text{ if }\tau=2.
		\end{cases}
	\end{align*}

	It is easy to show $T_{0,n}^{(l)}(\vartheta)\cp 0$ for $l=3,\ldots,6$. Here we have pointwise convergence in probability of $\bar{Q}_{n}(B(\cdot,\vartheta))$ for all $\vartheta$.\\
	
	Finally we see the uniform convergence. It can be obtained as
	\begin{align*}
		\sup_{n\in\mathbf{N}}\E{\sup_{\vartheta\in\Xi}\abs{\nabla\bar{Q}_{n}(B(\cdot,\vartheta))}}\le C
	\end{align*}
	whose computation is verified by Lemma \ref*{lem739}. Therefore uniform convergence in probability is obtained.
\end{proof}

\subsection{Asymptotic normality}

\begin{theorem}\label{thm751}
	Under (A1)-(A5), (AH) and $k_{n}\Delta_{n}^2\to0$,
	\begin{align*}
		\crotchet{\begin{matrix}
			\sqrt{n}D_{n}\\
			\sqrt{k_{n}}\crotchet{\bar{Q}_{n}\parens{B_{\kappa}(\cdot)}
				-\frac{2}{3}\bar{M}_{n}\parens{\ip{B_{\kappa}(\cdot)}{A_{n}^{\tau}\parens{\cdot,\alpha^{\star},\Lambda_{\star}}}}}_{\kappa}\\
			\sqrt{k_{n}\Delta_{n}}\crotchet{\bar{D}_{n}\parens{f_{\lambda}(\cdot)}}^{\lambda}
			\end{matrix}}\cl N(\mathbf{0},W^{(\tau)}(\tuborg{B_{\kappa}},\tuborg{f_{\lambda}})),
	\end{align*}
	where
	\begin{align*}
		A_{n}^{\tau}(x,\alpha,\Lambda)&:=A(x,\alpha)+3\Delta_{n}^{\frac{2-\tau}{\tau-1}}\Lambda.
	\end{align*}
\end{theorem}

\begin{proof}\textbf{(Step 1): } We can decompose
	\begin{align*}
		\sqrt{n}\parens{\hat{\Lambda}_{n}-\Lambda_{\star}}&=\frac{1}{2\sqrt{n}}\sum_{i=0}^{n-1}\parens{X_{(i+1)h_{n}}-X_{ih_{n}}}^{\otimes 2}
		\\
		&\qquad+\frac{1}{2\sqrt{n}}\sum_{i=0}^{n-1}\Lambda_{\star}^{1/2}\parens{\parens{\epsilon_{(i+1)h_{n}}-\epsilon_{ih_{n}}}^{\otimes 2}-I_d}\Lambda_{\star}^{1/2}\\
		&\qquad+\frac{1}{2\sqrt{n}}\sum_{i=0}^{n-1}\parens{X_{(i+1)h_{n}}-X_{ih_{n}}}\parens{\epsilon_{(i+1)h_{n}}-\epsilon_{ih_{n}}}^T\Lambda_{\star}^{1/2}\\
		&\qquad+\frac{1}{2\sqrt{n}}\sum_{i=0}^{n-1}\Lambda_{\star}^{1/2}\parens{\epsilon_{(i+1)h_{n}}-\epsilon_{ih_{n}}}\parens{X_{(i+1)h_{n}}-X_{ih_{n}}}^T.
	\end{align*}
	The first, third and fourth terms in right hand side are $o_P(1)$ which can be shown by $L^{1}$-evaluation and Lemma 9 in \citep{GeJ93}.
	Then we obtain
	\begin{align*}
		\sqrt{n}\parens{\hat{\Lambda}_{n}-\Lambda_{\star}}&=\frac{1}{2\sqrt{n}}\sum_{i=0}^{n-1}\Lambda_{\star}^{1/2}\parens{\parens{\epsilon_{(i+1)h_{n}}-\epsilon_{ih_{n}}}^{\otimes 2}-2I_d}\Lambda_{\star}^{1/2}+o_P(1)
	\end{align*}
	and 
	\begin{align*}
		\sqrt{n}D_{n}&=\frac{1}{2\sqrt{n}}\sum_{i=0}^{n-1}\vech\parens{\Lambda_{\star}^{1/2}\parens{\parens{\epsilon_{(i+1)h_{n}}-\epsilon_{ih_{n}}}^{\otimes 2}-2I_d}\Lambda_{\star}^{1/2}}+o_P(1).
	\end{align*}
	We can rewrite the summation as
	\begin{align*}
		&\frac{1}{2\sqrt{n}}\sum_{i=0}^{n-1}\Lambda_{\star}^{1/2}\parens{\parens{\epsilon_{(i+1)h_{n}}-\epsilon_{ih_{n}}}^{\otimes 2}-2I_d}\Lambda_{\star}^{1/2}\\
		&=\frac{1}{2\sqrt{n}}\sum_{i=1}^{n-1}\Lambda_{\star}^{1/2}\parens{2\parens{\epsilon_{ih_{n}}}^{\otimes 2}+\parens{\epsilon_{ih_{n}}}\parens{\epsilon_{(i-1)h_{n}}}^T+\parens{\epsilon_{(i-1)h_{n}}}\parens{\epsilon_{ih_{n}}}^T-2I_d}\Lambda_{\star}^{1/2}+o_P(1)\\
		&=\frac{1}{2\sqrt{n}}\sum_{i=p_{n}}^{n-p_{n}-1}\Lambda_{\star}^{1/2}\parens{2\parens{\epsilon_{ih_{n}}}^{\otimes 2}+\parens{\epsilon_{ih_{n}}}\parens{\epsilon_{(i-1)h_{n}}}^T+\parens{\epsilon_{(i-1)h_{n}}}\parens{\epsilon_{ih_{n}}}^T-2I_d}\Lambda_{\star}^{1/2}\\
		&\qquad+o_P(1)
	\end{align*}
	where the last evaluation holds because of Lemma 9 in \citep{GeJ93},
	and then
	\begin{align*}
		\sqrt{n}D_{n}&=\frac{\sqrt{n}}{k_{n}}\sum_{j=1}^{k_{n}-2}D_{j,n}'+o_P(1),
	\end{align*}
	where
	\begin{align*}
		D_{j,n}'&=\frac{1}{2p_{n}}\sum_{i=0}^{p_{n}-1}\vech\lparens{\Lambda_{\star}^{1/2}\lparens{2\parens{\epsilon_{j\Delta_{n}+ih_{n}}}^{\otimes 2}+\parens{\epsilon_{j\Delta_{n}+ih_{n}}}\parens{\epsilon_{j\Delta_{n}+(i-1)h_{n}}}^T}}\\
			&\qquad\qquad\qquad\qquad\qquad\rparens{\rparens{+\parens{\epsilon_{j\Delta_{n}+(i-1)h_{n}}}\parens{\epsilon_{j\Delta_{n}+ih_{n}}}^T-2I_d}\Lambda_{\star}^{1/2}},\\
		\parens{D_{j,n}'}^{(l_1,l_2)}&=\frac{1}{2p_{n}}\sum_{i=0}^{p_{n}-1}\lparens{\Lambda_{\star}^{1/2}\lparens{2\parens{\epsilon_{j\Delta_{n}+ih_{n}}}^{\otimes 2}+\parens{\epsilon_{j\Delta_{n}+ih_{n}}}\parens{\epsilon_{j\Delta_{n}+(i-1)h_{n}}}^T}}\\
		&\qquad\qquad\qquad\qquad\qquad\rparens{\rparens{+\parens{\epsilon_{j\Delta_{n}+(i-1)h_{n}}}\parens{\epsilon_{j\Delta_{n}+ih_{n}}}^T-2I_d}\Lambda_{\star}^{1/2}}^{(l_1,l_2)}.
	\end{align*}
	The conditional moment of $u_{i}$ is given as
	\begin{align*}
	\CE{D_{j,n}'}{\mathcal{H}_{j}^n}=\mathbf{0}.
	\end{align*}
	Note that
	\begin{align*}
		&\parens{\Lambda_{\star}^{1/2}\parens{2\parens{\epsilon_{ih_{n}}}^{\otimes 2}+\parens{\epsilon_{ih_{n}}}\parens{\epsilon_{(i-1)h_{n}}}^T+\parens{\epsilon_{(i-1)h_{n}}}\parens{\epsilon_{ih_{n}}}^T-2I_d}\Lambda_{\star}^{1/2}}^{(l_1,l_2)}\\
		&=2\parens{\sum_{k_1=1}^{d}\parens{\Lambda_{\star}^{1/2}}^{(l_1,k_1)}\parens{\epsilon_{ih_{n}}^{(k_1)}}}
		\parens{\sum_{k_2=1}^{d}\parens{\epsilon_{ih_{n}}^{(k_2)}}\parens{\Lambda_{\star}^{1/2}}^{(k_2,l_2)}}\\
		&\qquad+\parens{\sum_{k_1=1}^{d}\parens{\Lambda_{\star}^{1/2}}^{(l_1,k_1)}\parens{\epsilon_{ih_{n}}^{(k_1)}}}
		\parens{\sum_{k_2=1}^{d}\parens{\epsilon_{(i-1)h_{n}}^{(k_2)}}\parens{\Lambda_{\star}^{1/2}}^{(k_2,l_2)}}\\
		&\qquad+\parens{\sum_{k_1=1}^{d}\parens{\Lambda_{\star}^{1/2}}^{(l_1,k_1)}\parens{\epsilon_{(i-1)h_{n}}^{(k_1)}}}
		\parens{\sum_{k_2=1}^{d}\parens{\epsilon_{ih_{n}}^{(k_2)}}\parens{\Lambda_{\star}^{1/2}}^{(k_2,l_2)}}\\
		&\qquad-2\Lambda_{\star}^{(l_1,l_2)}
	\end{align*}
	and hence
	\begin{align*}
		&\tilde{D}_{ih_{n},n}\parens{(l_1,l_2),(l_3,l_4)}\\
		&:=\mathbf{E}\lcrotchet{\parens{\Lambda_{\star}^{1/2}\parens{2\parens{\epsilon_{ih_{n}}}^{\otimes 2}+\parens{\epsilon_{ih_{n}}}\parens{\epsilon_{(i-1)h_{n}}}^T+\parens{\epsilon_{(i-1)h_{n}}}\parens{\epsilon_{ih_{n}}}^T-2I_d}\Lambda_{\star}^{1/2}}^{(l_1,l_2)}}\\
		&\qquad\times\rcrotchet{\left.\parens{\Lambda_{\star}^{1/2}\parens{2\parens{\epsilon_{ih_{n}}}^{\otimes 2}+\parens{\epsilon_{ih_{n}}}\parens{\epsilon_{(i-1)h_{n}}}^T+\parens{\epsilon_{(i-1)h_{n}}}\parens{\epsilon_{ih_{n}}}^T-2I_d}\Lambda_{\star}^{1/2}}^{(l_3,l_4)}\right|\mathcal{H}_{(i-1)h_{n}}^n}\\
		&=4\sum_{k=1}^{d}\parens{\Lambda_{\star}^{1/2}}^{(l_1,k)}\parens{\Lambda_{\star}^{1/2}}^{(l_2,k)}\parens{\Lambda_{\star}^{1/2}}^{(l_3,k)}\parens{\Lambda_{\star}^{1/2}}^{(l_4,k)}
		\parens{\E{\abs{\epsilon_0^{(k)}}^4}-3}\\
		&\quad+4\parens{\Lambda_{\star}^{(l_1,l_3)}\Lambda_{\star}^{(l_2,l_4)}+\Lambda_{\star}^{(l_1,l_4)}
			\Lambda_{\star}^{(l_2,l_3)}}\\
		&\quad+\Lambda_{\star}^{(l_1,l_3)}\parens{\sum_{k_2=1}^{d}\parens{\epsilon_{(i-1)h_{n}}^{(k_2)}}\parens{\Lambda_{\star}^{1/2}}^{(k_2,l_2)}}\parens{\sum_{k_2=1}^{d}\parens{\epsilon_{(i-1)h_{n}}^{(k_2)}}\parens{\Lambda_{\star}^{1/2}}^{(k_2,l_4)}}\\
		&\quad+\Lambda_{\star}^{(l_1,l_4)}\parens{\sum_{k_2=1}^{d}\parens{\epsilon_{(i-1)h_{n}}^{(k_2)}}\parens{\Lambda_{\star}^{1/2}}^{(k_2,l_2)}}
		\parens{\sum_{k_1=1}^{d}\parens{\Lambda_{\star}^{1/2}}^{(l_3,k_1)}\parens{\epsilon_{(i-1)h_{n}}^{(k_1)}}}\\
		&\quad+\Lambda_{\star}^{(l_2,l_3)}\parens{\sum_{k_1=1}^{d}\parens{\Lambda_{\star}^{1/2}}^{(l_1,k_1)}\parens{\epsilon_{(i-1)h_{n}}^{(k_1)}}}\parens{\sum_{k_2=1}^{d}\parens{\epsilon_{(i-1)h_{n}}^{(k_2)}}\parens{\Lambda_{\star}^{1/2}}^{(k_2,l_4)}}\\
		&\quad+\Lambda_{\star}^{(l_2,l_4)}\parens{\sum_{k_1=1}^{d}\parens{\Lambda_{\star}^{1/2}}^{(l_1,k_1)}\parens{\epsilon_{(i-1)h_{n}}^{(k_1)}}}\parens{\sum_{k_1=1}^{d}\parens{\Lambda_{\star}^{1/2}}^{(l_3,k_1)}\parens{\epsilon_{(i-1)h_{n}}^{(k_1)}}}
	\end{align*}
	and
	\begin{align*}
		\CE{\parens{\sum_{k_2=1}^{d}\parens{\epsilon_{(i-1)h_{n}}^{(k_2)}}\parens{\Lambda_{\star}^{1/2}}^{(k_2,l_2)}}\parens{\sum_{k_2=1}^{d}\parens{\epsilon_{(i-1)h_{n}}^{(k_2)}}\parens{\Lambda_{\star}^{1/2}}^{(k_2,l_4)}}}{\mathcal{H}_{(i-2)h_{n}}^n}
		= \Lambda_{\star}^{(l_2,l_4)}.
	\end{align*}
	These lead to
	\begin{align*}
		\frac{n}{k_{n}^2}\sum_{j=1}^{k_{n}-2}\CE{\parens{D_{j,n}'}^{(l_1,l_2)}\parens{D_{j,n}'}^{(l_3,l_4)}}{\mathcal{H}_j^n}
		&=\frac{n}{4n^2}\sum_{j=1}^{k_{n}-2}\CE{\sum_{i=0}^{p_{n}-1}\tilde{D}_{j\Delta_{n}+ih_{n},n}\parens{(l_1,l_2),(l_3,l_4)}}{\mathcal{H}_j^n}\\
		&\cp W_1^{(l_1,l_2),(l_3,l_4)}.
	\end{align*}
	Then
	\begin{align*}
		\frac{n}{k_{n}^2}\sum_{j=1}^{k_{n}-2}\CE{\parens{D_{j,n}'}^{\otimes2}}{\mathcal{H}_{j}^n}\cp W_1.
	\end{align*}
	Finally we check
	\begin{align*}
		\E{\abs{\frac{n^2}{k_{n}^4}\sum_{j=1}^{k_{n}-2}\CE{\norm{D_{j,n}'}^4}{\mathcal{H}_{j}^n}}}\to 0.
	\end{align*}
	Note that $\tuborg{\epsilon_{ih_{n}}}$ are i.i.d. and when we denote
	\begin{align*}
		M_i=2\parens{\epsilon_{ih_{n}}}^{\otimes 2}+\parens{\epsilon_{ih_{n}}}\parens{\epsilon_{(i-1)h_{n}}}^T+\parens{\epsilon_{(i-1)h_{n}}}\parens{\epsilon_{ih_{n}}}^T-2I_d,
	\end{align*}
	then
	\begin{align*}
		\E{\norm{\sum_{i=1}^{p_{n}}M_i}^4}&=\E{\sum_{i_1}\sum_{i_2}\parens{\tr\parens{M_{i_1}M_{i_2}}}^2+\sum_{i_1}\sum_{i_2}\sum_{i_3}\sum_{i_4}\tr\parens{M_{i_1}M_{i_2}M_{i_3}M_{i_4}}}\\
		&\le Cp_{n}^2;
	\end{align*}
	They verify the result.\\
	
	\noindent\textbf{(Step 2): }
	Let us $u_{j,n}^{(l)}(\kappa),\ l=1,\ldots,7$ such that
	\begin{align*}
		u_{j,n}^{(1)}(\kappa)
		&:=\ip{B_{\kappa}(\lm{Y}{j-1})}{a(X_{j\Delta_{n}})\parens{\frac{1}{\Delta_{n}}\parens{\zeta_{j+1,n}+\zeta_{j+2,n}'}^{\otimes 2}
				-\frac{2}{3}I_r}a(X_{j\Delta_{n}})^{T}},\\
		u_{j,n}^{(2)}(\kappa)
		&:=\frac{2}{3}\ip{B_{\kappa}(\lm{Y}{j-1})}{A(X_{j\Delta_{n}})-A(\lm{Y}{j-1})},\\
		u_{j,n}^{(3)}(\kappa)
		&:=\ip{B_{\kappa}(\lm{Y}{j-1})}{\Lambda_{\star}^{1/2}
			\parens{\frac{1}{\Delta_{n}}\parens{\lm{\epsilon}{j+1}-\lm{\epsilon}{j}}^{\otimes 2}-2\Delta_{n}^{\frac{2-\tau}{\tau-1}}I_d}\Lambda_{\star}^{1/2}}\\
		u_{j,n}^{(4)}(\kappa)
		&:=\frac{2}{\Delta_{n}}\ip{\bar{B}_{\kappa}(\lm{Y}{j-1})}{
			\parens{a(X_{j\Delta_{n}})\parens{\zeta_{j+1,n}+\zeta_{j+2,n}'}}
			\parens{\Lambda_{\star}^{1/2}\parens{\lm{\epsilon}{j+1}-\lm{\epsilon}{j}}}^T},\\
		u_{j,n}^{(5)}(\kappa)
		&:=\frac{1}{\Delta_{n}}\ip{B_{\kappa}(\lm{Y}{j-1})}{\parens{\Delta_{n}b(X_{j\Delta_{n}})+e_{j,n}}^{\otimes 2}},\\
		u_{j,n}^{(6)}(\kappa)
		&:=\frac{2}{\Delta_{n}}\ip{\bar{B}_{\kappa}(\lm{Y}{j-1})}{
			\parens{\Delta_{n}b(X_{j\Delta_{n}})+e_{j,n}}\parens{a(X_{j\Delta_{n}})\parens{\zeta_{j+1,n}+\zeta_{j+2,n}'}}^T},\\
		u_{j,n}^{(7)}(\kappa)
		&:=\frac{2}{\Delta_{n}}
		\ip{\bar{B}_{\kappa}(\lm{Y}{j-1})}{
			\parens{\Delta_{n}b(X_{j\Delta_{n}})+e_{j,n}}
			\parens{\Lambda_{\star}^{1/2}\parens{\lm{\epsilon}{j+1}-\lm{\epsilon}{j}}}^T}.
	\end{align*}
	where $\bar{B}_{\kappa}:=\frac{1}{2}\parens{B_{\kappa}+B_{\kappa}^T}$. Then because of Corollary \ref{cor738}, we obtain
	\begin{align*}
		U_{n}(\kappa)=\sum_{l=1}^{7}U_{n}^{(l)}(\kappa),\ 
		U_{n}^{(l)}(\kappa)=\frac{1}{\sqrt{k_{n}}}\sum_{j=1}^{k_{n}-2}u_{j,n}^{(l)}(\kappa).
	\end{align*}
	With respect to $U_{n}^{(1)}(\kappa)$, we have
	\begin{align*}
		U_{n}^{(1)}(\kappa)=\frac{1}{\sqrt{k_{n}}}\sum_{j=1}^{k_{n}-2}u_{j,n}^{(1)}(\kappa)
		=\frac{1}{\sqrt{k_{n}}}\sum_{j=2}^{k_{n}-2}s_{j,n}^{(1)}(\kappa)+
		\frac{1}{\sqrt{k_{n}}}\sum_{j=2}^{k_{n}-2}\tilde{s}_{j,n}^{(1)}(\kappa)+o_P(1),
	\end{align*}
	where
	\begin{align*}
		s_{j,n}^{(1)}(\kappa)&=\ip{B_{\kappa}(\lm{Y}{j-1})}{
			a(X_{j\Delta_{n}})\parens{\frac{1}{\Delta_{n}}\parens{\zeta_{j+1,n}}^{\otimes 2}-m_{n}I_r}a(X_{j\Delta_{n}})^{T}}\\
		&\qquad+\ip{B_{\kappa}(\lm{Y}{j-2})}{
			a(X_{(j-1)\Delta_{n}})\parens{\frac{1}{\Delta_{n}}\parens{\zeta_{j+1,n}'}^{\otimes 2}-m_{n}'I_r}a(X_{(j-1)\Delta_{n}})^{T}}\\
		&\qquad+2\ip{\bar{B}_{\kappa}(\lm{Y}{j-2})}{a(X_{(j-1)\Delta_{n}})
			\parens{\frac{1}{\Delta_{n}}\parens{\zeta_{j,n}\parens{\zeta_{j+1,n}'}^T}}a(X_{(j-1)\Delta_{n}})^{T}},\\
		\tilde{s}_{j,n}^{(1)}(\kappa)&=\parens{\frac{1}{2p_{n}}+\frac{1}{6p_{n}^2}}
		\ip{B_{\kappa}(\lm{Y}{j-1})}{A(X_{j\Delta_{n}})}\\
		&\qquad+\parens{-\frac{1}{2p_{n}}+\frac{1}{6p_{n}^2}}\ip{B_{\kappa}(\lm{Y}{j-2})}{A(X_{(j-1)\Delta_{n}})}.
	\end{align*}
	Actually the second term converge to 0 in $L^1$ and hence
	\begin{align*}
		U_{n}^{(1)}(\kappa)=\frac{1}{\sqrt{k_{n}}}\sum_{j=2}^{k_{n}-2}s_{j,n}^{(1)}(\kappa)+o_P(1)
	\end{align*}
	and it is enough to examine the first term of the right hand side.
	Firstly, Lemma \ref*{lem732} leads to
	\begin{align*}
		\CE{s_{j,n}^{(1)}(\kappa)}{\mathcal{H}_j^n}=0.
	\end{align*}
	Note the fact that for $\Re^r$-valued random vectors $\mathbf{x}$ and $\mathbf{y}$ such that
	\begin{align*}
		\crotchet{\begin{matrix}
			\mathbf{x}\\
			\mathbf{y}
			\end{matrix}}\sim N\parens{\mathbf{0}, \crotchet{\begin{matrix}
				\sigma_{11}I_r & \sigma_{12}I_r\\
				\sigma_{12}I_r & \sigma_{22}I_r
				\end{matrix}}},
	\end{align*}
	where $\sigma_{11}>0$, $\sigma_{22}>0$ and $\abs{\sigma_{12}}^2\le \sigma_{11}\sigma_{22}$, it holds for any $\Re^r\times\Re^r$-valued matrix $M$,
	\begin{align*}
		\E{\mathbf{y}\mathbf{x}^TM\mathbf{x}\mathbf{y}^T}=\sigma_{12}^2\parens{M+M^T}+\sigma_{11}\sigma_{22}\tr\parens{M}I_r
	\end{align*}
	and also the fact that for any square matrices $A$ and $B$ whose dimensions coincide,
	\begin{align*}
		\tr\parens{AB}+\tr\parens{AB^T}=2\tr\parens{\bar{A}\bar{B}},
	\end{align*}
	where $\bar{A}=\parens{A+A^T}/2$ and $\bar{B}=\parens{B+B^T}/2$.
	For all $\kappa_1,\ \kappa_2$,
	we obtain
	\begin{align*}
		&\CE{\parens{s_{j,n}^{(1)}(\kappa_1)}\parens{s_{j,n}^{(1)}(\kappa_2)}}{\mathcal{H}_j^n}\\
		&=m_{n}^2\tr\parens{B_{\kappa_1}(\lm{Y}{j-1})A(X_{j\Delta_{n}})B_{\kappa_2}(\lm{Y}{j-1})A(X_{j\Delta_{n}})}\\
		&\qquad+m_{n}^2\tr\parens{B_{\kappa_1}(\lm{Y}{j-1})^TA(X_{j\Delta_{n}})B_{\kappa_2}(\lm{Y}{j-1})A(X_{j\Delta_{n}})}\\
		&\qquad+\parens{m_{n}'}^2\tr\parens{B_{\kappa_1}(\lm{Y}{j-2})A(X_{(j-1)\Delta_{n}})B_{\kappa_2}(\lm{Y}{j-2})A(X_{(j-1)\Delta_{n}})}\\
		&\qquad+\parens{m_{n}'}^2\tr\parens{B_{\kappa_1}(\lm{Y}{j-2})^TA(X_{(j-1)\Delta_{n}})B_{\kappa_2}(\lm{Y}{j-2})A(X_{(j-1)\Delta_{n}})}\\
		&\qquad+\frac{4m_{n}'}{\Delta_{n}}\parens{\zeta_{j,n}}^Ta(X_{(j-1)\Delta_{n}})^{T}\bar{B}_{\kappa_1}(\lm{Y}{j-2})A(X_{(j-1)\Delta_{n}})
		\bar{B}_{\kappa_2}(\lm{Y}{j-2})a(X_{(j-1)\Delta_{n}})
		\parens{\zeta_{j,n}}\\
		&\qquad+\chi_{n}^2 \tr\left\{a(X_{j\Delta_{n}})^{T}B_{\kappa_1}(\lm{Y}{j-1})
		a(X_{j\Delta_{n}})a(X_{(j-1)\Delta_{n}})^{T}B_{\kappa_2}(\lm{Y}{j-2})a(X_{(j-1)\Delta_{n}})\right\}\\
		&\qquad+\chi_{n}^2 \tr\left\{a(X_{j\Delta_{n}})^{T}B_{\kappa_1}(\lm{Y}{j-1})^T
		a(X_{j\Delta_{n}})a(X_{(j-1)\Delta_{n}})^{T}B_{\kappa_2}(\lm{Y}{j-2})a(X_{(j-1)\Delta_{n}})\right\}\\
		&\qquad+\chi_{n}^2 \tr\left\{a(X_{j\Delta_{n}})^{T}B_{\kappa_2}(\lm{Y}{j-1})
		a(X_{j\Delta_{n}})a(X_{(j-1)\Delta_{n}})^{T}B_{\kappa_1}(\lm{Y}{j-2})a(X_{(j-1)\Delta_{n}})\right\}\\
		&\qquad+\chi_{n}^2 \tr\left\{a(X_{j\Delta_{n}})^{T}B_{\kappa_2}(\lm{Y}{j-1})^T
		a(X_{j\Delta_{n}})a(X_{(j-1)\Delta_{n}})^{T}B_{\kappa_1}(\lm{Y}{j-2})a(X_{(j-1)\Delta_{n}})\right\}.
	\end{align*}
	We have the following evaluation
	\begin{align*}
		&\frac{1}{k_{n}}\sum_{j=2}^{k_{n}}\CE{\frac{4m_{n}'}{\Delta_{n}}\parens{\zeta_{j,n}}^Ta(X_{(j-1)\Delta_{n}})^{T}\bar{B}_{\kappa_1}(\lm{Y}{j-2})A(X_{(j-1)\Delta_{n}})
			\bar{B}_{\kappa_2}(\lm{Y}{j-2})a(X_{(j-1)\Delta_{n}})
			\parens{\zeta_{j,n}}}{\mathcal{H}_j^n}\\
		&\cp \frac{4}{9}\nu_0\parens{\tr\parens{\bar{B}_{\kappa_1}(\cdot)A(\cdot)\bar{B}_{\kappa_2}(\cdot)A(\cdot)}}
	\end{align*}
	and
	\begin{align*}
		&\frac{1}{k_{n}^2}\sum_{j=2}^{k_{n}}\CE{\abs{\frac{4m_{n}'}{\Delta_{n}}\parens{\zeta_{j,n}}^Ta(X_{(j-1)\Delta_{n}})^{T}\bar{B}_{\kappa_1}(\lm{Y}{j-2})A(X_{(j-1)\Delta_{n}})
			\bar{B}_{\kappa_2}(\lm{Y}{j-2})a(X_{(j-1)\Delta_{n}})
			\parens{\zeta_{j,n}}}^2}{\mathcal{H}_j^n}\\
		&=o_P(1).
	\end{align*}
	Then we have
	\begin{align*}
		\frac{1}{k_{n}}\sum_{j=2}^{k_{n}-2}\CE{\parens{s_{j,n}^{(1)}(\kappa_1)}\parens{s_{j,n}^{(1)}(\kappa_2)}}{\mathcal{H}_j^n}
		\cp\nu_0\parens{\tr\parens{\bar{B}_{\kappa_1}(\cdot)A(\cdot)\bar{B}_{\kappa_2}(\cdot)A(\cdot)}}.
	\end{align*}
	Now let us consider the fourth conditional expectation. 
	But it is easy to see
	\begin{align*}
		\E{\abs{\frac{1}{k_{n}^2}\sum_{j=2}^{k_{n}-2}\CE{\parens{s_{j,n}^{(1)}(\kappa)}^4}{\mathcal{H}_j^n}}}\to0.
	\end{align*}
	For $U_{n}^{(2)}(\kappa)$, it holds $U_{n}^{(2)}(\kappa)=o_P(1)$ shown by Lemma 9 in \citep{GeJ93} and Corollary \ref*{cor736}. We will see the asymptotic behaviour of $U_{n}^{(3)}(\kappa)$ in the next place. As $u_{n}^{(1)}(\kappa)$, $u_{j,n}^{(3)}(\kappa)$ contains $\mathcal{H}_{j+1}^{n}$-measurable $\lm{\epsilon}{j}$
	 and $\mathcal{H}_{j+2}^{n}$-measurable $\lm{\epsilon}{j+1}$. Hence we rewrite the summation as follows:
	\begin{align*}
		U_{n}^{(3)}(\kappa)=\frac{1}{\sqrt{k_{n}}}\sum_{j=1}^{k_{n}-2}u_{j,n}^{(3)}(\kappa)
		=\frac{1}{\sqrt{k_{n}}}\sum_{j=2}^{k_{n}-2}s_{j,n}^{(3)}(\kappa)+o_P(1),
	\end{align*}
	where
	\begin{align*}
		s_{j,n}^{(3)}(\kappa)&:=\ip{\parens{B_{\kappa}(\lm{Y}{j-2})+B_{\kappa}(\lm{Y}{j-1})}}{\Lambda_{\star}^{1/2}
			\parens{\frac{1}{\Delta_{n}}\parens{\lm{\epsilon}{j}}^{\otimes 2}-\Delta_{n}^{\frac{2-\tau}{\tau-1}}I_d}\Lambda_{\star}^{1/2}}\\
		&\qquad+\ip{B_{\kappa}(\lm{Y}{j-2})}{\Lambda_{\star}^{1/2}
			\parens{-\frac{2}{\Delta_{n}}\parens{\lm{\epsilon}{j-1}}\parens{\lm{\epsilon}{j}}^T}\Lambda_{\star}^{1/2}}.
	\end{align*}
	We can obtain
	\begin{align*}
		\CE{s_{j,n}^{(3)}(\kappa)}{\mathcal{H}_j^n}
		=0.
	\end{align*}
	For all $\kappa_1$ and $\kappa_2$,
	\begin{align*}
		&\CE{s_{j,n}^{(3)}(\kappa_1)s_{j,n}^{(3)}(\kappa_2)}{\mathcal{H}_j^n}\\
		&=\frac{1}{\Delta_{n}^2}\mathbf{E}\left[\parens{\lm{\epsilon}{j}}^T\Lambda_{\star}^{1/2}\parens{B_{\kappa_1}(\lm{Y}{j-2})+B_{\kappa_1}(\lm{Y}{j-1})}\Lambda_{\star}^{1/2}
			\parens{\lm{\epsilon}{j}}^{\otimes 2}\right.\\
		&\hspace{3cm}\left.\left.\times\Lambda_{\star}^{1/2}\parens{B_{\kappa_2}(\lm{Y}{j-2})+B_{\kappa_2}(\lm{Y}{j-1})}\Lambda_{\star}^{1/2}
		\parens{\lm{\epsilon}{j}}\right|\mathcal{H}_j^n\right]\\
		&\qquad-\parens{\Delta_{n}^{\frac{2-\tau}{\tau-1}}}^2\ip{\parens{B_{\kappa_1}(\lm{Y}{j-2})+B_{\kappa_1}(\lm{Y}{j-1})}}{\Lambda_{\star}}
		\ip{\parens{B_{\kappa_2}(\lm{Y}{j-2})+B_{\kappa_2}(\lm{Y}{j-1})}}{\Lambda_{\star}}\\
		&\qquad+\frac{4\Delta_{n}^{\frac{2-\tau}{\tau-1}}}{\Delta_{n}}\parens{\lm{\epsilon}{j-1}}^T\Lambda_{\star}^{1/2}B_{\kappa_1}(\lm{Y}{j-2})\Lambda_{\star}\parens{B_{\kappa_2}(\lm{Y}{j-2})}^T\Lambda_{\star}^{1/2}\parens{\lm{\epsilon}{j-1}}.
	\end{align*}
	Note the fact that for any $\Re^d\times\Re^d$-valued matrix $A$
	\begin{align*}
		\E{\parens{\lm{\epsilon}{j}}^{\otimes2}A\parens{\lm{\epsilon}{j}}^{\otimes2}}
		&=\frac{1}{p_{n}^2}\parens{2\bar{A}}+\frac{1}{p_{n}^2}\tr\parens{A}I_d
		+\crotchet{\parens{\frac{\E{\parens{\epsilon_{0}^{(i)}}^4}-3}{p_{n}^3}}A^{(i,i)}}_{i,i},
	\end{align*}
	where $\bar{A}=\parens{A+A^T}/2$. Therefore,
	\begin{align*}
		&\mathbf{E}\left[\parens{\lm{\epsilon}{j}}^T\Lambda_{\star}^{1/2}\parens{B_{\kappa_1}(\lm{Y}{j-2})+B_{\kappa_1}(\lm{Y}{j-1})}\Lambda_{\star}^{1/2}
		\parens{\lm{\epsilon}{j}}^{\otimes 2}\right.\\
		&\qquad\qquad\left.\left.\times\Lambda_{\star}^{1/2}\parens{B_{\kappa_2}(\lm{Y}{j-2})+B_{\kappa_2}(\lm{Y}{j-1})}\Lambda_{\star}^{1/2}
		\parens{\lm{\epsilon}{j}}\right|\mathcal{H}_j^n\right]\\
		&=\frac{2}{p_{n}^2}\tr\tuborg{\Lambda_{\star}^{1/2}\parens{\bar{B}_{\kappa_1}(\lm{Y}{j-2})+\bar{B}_{\kappa_1}(\lm{Y}{j-1})}\Lambda_{\star}^{1/2}\Lambda_{\star}^{1/2}\parens{B_{\kappa_2}(\lm{Y}{j-2})+B_{\kappa_2}(\lm{Y}{j-1})}\Lambda_{\star}^{1/2}}\\
		&\qquad+\frac{1}{p_{n}^2}\tr\tuborg{\Lambda_{\star}^{1/2}\parens{B_{\kappa_1}(\lm{Y}{j-2})+B_{\kappa_1}(\lm{Y}{j-1})}\Lambda_{\star}^{1/2}}
		\tr\tuborg{\Lambda_{\star}^{1/2}\parens{B_{\kappa_2}(\lm{Y}{j-2})+B_{\kappa_2}(\lm{Y}{j-1})}\Lambda_{\star}^{1/2}}\\
		&\qquad+\sum_{i=1}^{d}\parens{\frac{\E{\parens{\epsilon_{0}^{(i)}}^4}-3}{p_{n}^3}}\parens{\Lambda_{\star}^{1/2}\parens{B_{\kappa_1}(\lm{Y}{j-2})+B_{\kappa_1}(\lm{Y}{j-1})}\Lambda_{\star}^{1/2}}^{(i,i)}\\
		&\qquad\qquad\qquad\times\parens{\Lambda_{\star}^{1/2}\parens{B_{\kappa_2}(\lm{Y}{j-2})+B_{\kappa_2}(\lm{Y}{j-1})}\Lambda_{\star}^{1/2}}^{(i,i)}.
	\end{align*}
	Hence
	\begin{align*}
		&\frac{1}{k_{n}}\sum_{j=2}^{k_{n}-2}\CE{s_{j,n}^{(3)}(\kappa_1)s_{j,n}^{(3)}(\kappa_2)}{\mathcal{H}_j^n}\\
		&=\frac{1}{k_{n}}\sum_{j=2}^{k_{n}-2}\frac{2}{p_{n}^2\Delta_{n}^2}\tr\tuborg{\parens{\bar{B}_{\kappa_1}(\lm{Y}{j-2})+\bar{B}_{\kappa_1}(\lm{Y}{j-1})}\Lambda_{\star}\parens{B_{\kappa_2}(\lm{Y}{j-2})+B_{\kappa_2}(\lm{Y}{j-1})}\Lambda_{\star}}\\
		&\qquad+\frac{1}{k_{n}}\sum_{j=2}^{k_{n}-2}\frac{4\Delta_{n}^{\frac{2-\tau}{\tau-1}}}{\Delta_{n}}\parens{\lm{\epsilon}{j-1}}^T\Lambda_{\star}^{1/2}B_{\kappa_1}(\lm{Y}{j-2})\Lambda_{\star}\parens{B_{\kappa_1}(\lm{Y}{j-2})}^T\Lambda_{\star}^{1/2}\parens{\lm{\epsilon}{j-1}}\\
		&\qquad+o_P(1)\\
		&\cp \begin{cases}
		0 & \text{ if }\tau\in(1,2),\\
		12\nu_0\parens{\tr\tuborg{\bar{B}_{\kappa_1}(\cdot)\Lambda_{\star}\bar{B}_{\kappa_2}(\cdot)\Lambda_{\star}}} & \text{ if }\tau=2.
		\end{cases}
	\end{align*}
	Therefore, $U_{n}^{(3)}(\kappa)=o_P(1)$ if $\tau\in(1,2)$. The conditional fourth expectation of $s_{j,n}^{(3)}$ can be evaluated easily as
	\begin{align*}
		\E{\abs{\frac{1}{k_{n}^2}\sum_{j=2}^{k_{n}-2}\CE{\parens{s_{j,n}^{(3)}(\kappa)}^4}{\mathcal{H}_j^n}}}
		\to0
	\end{align*}
	using Lemma \ref{lem739}. In the next place, we see the asymptotic behaviour of $U_{n}^{(4)}(\kappa)$. We again rewrite the summation as follows:
	\begin{align*}
		U_{n}^{(4)}(\kappa)&:=\frac{1}{\sqrt{k_{n}}}\sum_{j=1}^{k_{n}-2}\frac{2}{\Delta_{n}}\ip{\bar{B}_{\kappa}(\lm{Y}{j-1})}{
			a(X_{j\Delta_{n}})\parens{\zeta_{j+1,n}+\zeta_{j+2,n}'}
			\parens{\lm{\epsilon}{j+1}-\lm{\epsilon}{j}}^T\Lambda_{\star}^{1/2}}\\
		&=\frac{1}{\sqrt{k_{n}}}\sum_{j=2}^{k_{n}-2}s_{j,n}^{(4)}(\kappa)+o_P(1),
	\end{align*}
	where
	\begin{align*}
		s_{j,n}^{(4)}(\kappa)&:=\frac{2}{\Delta_{n}}\ip{\bar{B}_{\kappa}(\lm{Y}{j-2})}{
		a(X_{(j-1)\Delta_{n}})\parens{\zeta_{j,n}}
		\parens{\lm{\epsilon}{j}}^T\Lambda_{\star}^{1/2}}\\
		&\qquad-\frac{2}{\Delta_{n}}\ip{\bar{B}_{\kappa}(\lm{Y}{j-1})}{
			a(X_{j\Delta_{n}})\parens{\zeta_{j+1,n}}
			\parens{\lm{\epsilon}{j}}^T\Lambda_{\star}^{1/2}}\\
		&\qquad+\frac{2}{\Delta_{n}}\ip{\bar{B}_{\kappa}(\lm{Y}{j-2})}{
			a(X_{(j-1)\Delta_{n}})\parens{\zeta_{j+1,n}'}
			\parens{\lm{\epsilon}{j}}^T\Lambda_{\star}^{1/2}}\\
		&\qquad-\frac{2}{\Delta_{n}}\ip{\bar{B}_{\kappa}(\lm{Y}{j-2})}{
			a(X_{(j-1)\Delta_{n}})\parens{\zeta_{j+1,n}'}
			\parens{\lm{\epsilon}{j-1}}^T\Lambda_{\star}^{1/2}}.
	\end{align*}
	Hence it is enough to examine $\frac{1}{\sqrt{k_{n}}}\sum_{j=2}^{k_{n}-2}s_{j,n}^{(4)}(\kappa)$. It is obvious that
	\begin{align*}
		\CE{s_{j,n}^{(4)}(\kappa)}{\mathcal{H}_j^n}=0.
	\end{align*}
	For all $\kappa_1$ and $\kappa_2$,
	\begin{align*}
		&\frac{1}{k_{n}}\sum_{j=2}^{k_{n}-2}\CE{s_{j,n}^{(4)}(\kappa_1)s_{j,n}^{(4)}(\kappa_2)}{\mathcal{H}_j^n}\\
		&=\frac{1}{k_{n}}\sum_{j=2}^{k_{n}-2}\parens{\frac{2}{\Delta_{n}}}^{2}\frac{1}{p_{n}}\parens{\zeta_{j,n}}^Ta(X_{(j-1)\Delta_{n}})^T\bar{B}_{\kappa_1}(\lm{Y}{j-2})\Lambda_{\star}\bar{B}_{\kappa_2}(\lm{Y}{j-2})a(X_{(j-1)\Delta_{n}})\parens{\zeta_{j,n}}\\
		&\qquad+\frac{1}{k_{n}}\sum_{j=2}^{k_{n}-2}\parens{\frac{2}{\Delta_{n}}}^{2}\frac{m_{n}\Delta_{n}}{p_{n}}\tr\tuborg{\bar{B}_{\kappa_1}(\lm{Y}{j-1})\Lambda_{\star}\bar{B}_{\kappa_2}(\lm{Y}{j-1})A(X_{j\Delta_{n}})}\\
		&\qquad+\frac{1}{k_{n}}\sum_{j=2}^{k_{n}-2}\parens{\frac{2}{\Delta_{n}}}^{2}\frac{m_{n}'\Delta_{n}}{p_{n}}\tr\tuborg{\bar{B}_{\kappa_1}(\lm{Y}{j-2})\Lambda_{\star}\bar{B}_{\kappa_2}(\lm{Y}{j-2})A(X_{(j-1)\Delta_{n}})}\\
		&\qquad+\frac{1}{k_{n}}\sum_{j=2}^{k_{n}-2}\parens{\frac{2}{\Delta_{n}}}^{2}m_{n}'\Delta_{n}\parens{\lm{\epsilon}{j-1}}^T\Lambda_{\star}^{1/2}\bar{B}_{\kappa_1}(\lm{Y}{j-2})
		A(X_{(j-1)\Delta_{n}})\bar{B}_{\kappa_2}(\lm{Y}{j-2})\Lambda_{\star}^{1/2}\parens{\lm{\epsilon}{j-1}}\\
		&\qquad-\frac{1}{k_{n}}\sum_{j=2}^{k_{n}-2}\parens{\frac{2}{\Delta_{n}}}^{2}\frac{\chi_{n}\Delta_{n}}{p_{n}}\tr\tuborg{\bar{B}_{\kappa_1}(\lm{Y}{j-1})\Lambda_{\star}\bar{B}_{\kappa_2}(\lm{Y}{j-2})a(X_{(j-1)\Delta_{n}})a(X_{j\Delta_{n}})^T}\\
		&\qquad-\frac{1}{k_{n}}\sum_{j=2}^{k_{n}-2}\parens{\frac{2}{\Delta_{n}}}^{2}\frac{\chi_{n}\Delta_{n}}{p_{n}}\tr\tuborg{\bar{B}_{\kappa_2}(\lm{Y}{j-1})\Lambda_{\star}\bar{B}_{\kappa_1}(\lm{Y}{j-2})a(X_{(j-1)\Delta_{n}})a(X_{j\Delta_{n}})^T}\\
		&\cp \begin{cases}
		0 & \text{ if }\tau\in(1,2)\\
		4\nu_0\parens{\tr\tuborg{\bar{B}_{\kappa_2}(\cdot)\Lambda_{\star}\bar{B}_{\kappa_1}(\cdot)A(\cdot)}} & \text{ if }\tau=2
		\end{cases}
	\end{align*}
	which can be shown by using Lemma 9 in \citep{GeJ93}.
	Hence $U_{n}^{(4)}(\kappa)=o_P(1)$ if $\tau\in(1,2)$. The conditional fourth moment can be evaluated as
	\begin{align*}
		\E{\abs{\frac{1}{k_{n}^2}\sum_{j=2}^{k_{n}-2}\CE{\abs{s_{j,n}^{(4)}(\kappa)}^4}{\mathcal{H}_j^n}}}\to 0.
	\end{align*}
	For $l=5,6,7$, it is not difficult to see $U_{n}^{(l)}(\kappa)=o_P(1)$ using Lemma 9 in \citep{GeJ93}.

	Finally we see the covariance structure among $U_{n}^{(1)}$, $U_{n}^{(3)}$ and $U_{n}^{(4)}$ when $\tau=2$. Because of the independence of $\tuborg{w_t}$ and $\tuborg{\epsilon_{ih_{n}}}$, for all $\kappa_1$ and $\kappa_2$,
	\begin{align*}
		\CE{s_{j,n}^{(1)}(\kappa_1)s_{j,n}^{(3)}(\kappa_2)}{\mathcal{H}_j^n}=\CE{s_{j,n}^{(1)}(\kappa_1)}{\mathcal{H}_j^n}\CE{s_{j,n}^{(3)}(\kappa_2)}{\mathcal{H}_j^n}=0.
	\end{align*}
	With respect to the covariance between $U_{n}^{(1)}$ and $U_{n}^{(4)}$, the independence of $\tuborg{w_t}$ and $\tuborg{\epsilon_{ih_{n}}}$ leads to
	\begin{align*}
		&\frac{1}{k_{n}}\sum_{j=2}^{k_{n}-2}\CE{s_{j,n}^{(1)}(\kappa_1)s_{j,n}^{(4)}(\kappa_2)}{\mathcal{H}_j^n}\\
		&=-\frac{1}{k_{n}}\sum_{j=2}^{k_{n}-2}\frac{4m_{n}'}{\Delta_{n}}\parens{\zeta_{j,n}}^{T}a(X_{(j-1)\Delta_{n}})^T\bar{B}_{\kappa_1}(\lm{Y}{j-2})A(X_{(j-1)\Delta_{n}})\bar{B}_{\kappa_2}(\lm{Y}{j-2})
		\Lambda_{\star}^{1/2}\parens{\lm{\epsilon}{j-1}}\\
		&\cp 0
	\end{align*}
	obtained by Lemma 9 in \citep{GeJ93} with ease.
	The analogous discussion can be conducted between $U_{n}^{(3)}(\kappa_1)$ and $U_{n}^{(4)}(\kappa_2)$; it holds
	\begin{align*}
		\frac{1}{k_{n}}\sum_{j=2}^{k_{n}-2}\CE{s_{j,n}^{(3)}(\kappa_1)s_{j,n}^{(4)}(\kappa_2)}{\mathcal{H}_j^n}=o_P(1)
	\end{align*}
	by Lemma 9 in \citep{GeJ93}.\\

	\noindent \textbf{(Step 3): } Corllary \ref{cor738} leads to 
	\begin{align*}
		\sqrt{k_{n}\Delta_{n}} \bar{D}_{n}\parens{f^{\lambda}(\cdot)}=\bar{R}_{n}^{(1)}(\lambda)+\bar{R}_{n}^{(2)}(\lambda)+\bar{R}_{n}^{(3)}(\lambda)+\bar{R}_{n}^{(4)}(\lambda),
	\end{align*}
	where
	\begin{align*}
		\bar{R}_{n}^{(1)}(\lambda)&:=\frac{1}{\sqrt{k_{n}\Delta_{n}}}\sum_{j=1}^{k_{n}-2}f^{\lambda}(\lm{Y}{j-1})a(X_{j\Delta_{n}})\parens{\zeta_{j+1,n}+\zeta_{j+2,n}'},\\
		\bar{R}_{n}^{(2)}(\lambda)&:=\frac{1}{\sqrt{k_{n}\Delta_{n}}}\sum_{j=1}^{k_{n}-2}f^{\lambda}(\lm{Y}{j-1})\Lambda_{\star}^{1/2}\parens{\lm{\epsilon}{j+1}-\lm{\epsilon}{j}},\\
		\bar{R}_{n}^{(3)}(\lambda)&:=\frac{1}{\sqrt{k_{n}\Delta_{n}}}\sum_{j=1}^{k_{n}-2}f^{\lambda}(\lm{Y}{j-1})e_{j,n},\\
		\bar{R}_{n}^{(4)}(\lambda)&:=\frac{1}{\sqrt{k_{n}\Delta_{n}}}\sum_{j=1}^{k_{n}-2}f^{\lambda}(\lm{Y}{j-1})
		\Delta_{n}\parens{b(X_{j\Delta_{n}})-b(\lm{Y}{j-1})}.
	\end{align*}
	Hence it is enough to see asymptotic behaviour of $\bar{R}$'s and firstly we examine that of $R_{n}^{(1)}$. We define the $\mathcal{H}_{j+1}^n$-measurable random variable 
	\begin{align*}
		r_{j,n}^{(1)}(\lambda)&:=\frac{1}{\sqrt{k_{n}\Delta_{n}}}f^{\lambda}(\lm{Y}{j-1})a(X_{j\Delta_{n}})\zeta_{j+1,n}+\frac{1}{\sqrt{k_{n}\Delta_{n}}}f^{\lambda}(\lm{Y}{j-2})a(X_{(j-1)\Delta_{n}})\zeta_{j+1,n}'
	\end{align*}
	and then
	\begin{align*}
		\bar{R}_{n}^{(1)}(\lambda)=\frac{1}{\sqrt{k_{n}\Delta_{n}}}\sum_{j=1}^{k_{n}-2}f^{\lambda}(\lm{Y}{j-1})a(X_{j\Delta_{n}})\parens{\zeta_{j+1,n}+\zeta_{j+2,n}'}=\sum_{j=2}^{k_{n}-2}r_{j,n}^{(1)}(\lambda)+o_P(1).
	\end{align*}
	Obviously
	\begin{align*}
		\CE{r_{j,n}^{(1)}(\lambda)}{\mathcal{H}_j^n}=0.
	\end{align*}
	With respect to the second moment, for all $\lambda_1,\lambda_2\in\{1,\ldots,m_2\}$, by Lemma \ref*{lem732}, 
	\begin{align*}
		&\CE{\parens{r_{j,n}^{(1)}(\lambda_1)}\parens{r_{j,n}^{(1)}(\lambda_2)}}{\mathcal{H}_j^n}\\
		&=\frac{1}{k_{n}\Delta_{n}}\Delta_{n}\parens{\frac{1}{3}+\frac{1}{2p_{n}}+\frac{1}{6p_{n}^2}}f_{\lambda_1}(\lm{Y}{j-1})A(X_{j\Delta_{n}})\parens{f_{\lambda_2}(\lm{Y}{j-1})}^T\\
		&\qquad+\frac{1}{k_{n}\Delta_{n}}\Delta_{n}\parens{\frac{1}{3}-\frac{1}{2p_{n}}+\frac{1}{6p_{n}^2}}f_{\lambda_1}(\lm{Y}{j-2})A(X_{(j-1)\Delta_{n}})\parens{f_{\lambda_2}(\lm{Y}{j-2})}^T\\
		&\qquad+\frac{2}{k_{n}\Delta_{n}}\frac{\Delta_{n}}{6}\parens{1-\frac{1}{p_{n}^2}}f_{\lambda_1}(\lm{Y}{j-1})a(X_{j\Delta_{n}})\parens{a(X_{(j-1)\Delta_{n}})}^T\parens{f_{\lambda_2}(\lm{Y}{j-2})}^T.
	\end{align*}
	Therefore,
	\begin{align*}
		\sum_{j=2}^{k_{n}-2}\CE{\parens{r_{j,n}^{(1)}(\lambda_1)}\parens{r_{j,n}^{(1)}(\lambda_2)}}{\mathcal{H}_j^n}
		\cp\nu_0\parens{\parens{f_{\lambda_1}}\parens{c}\parens{f_{\lambda_2}}^T(\cdot)}
	\end{align*}
	because of Lemma \ref*{lem721} and Lemma \ref*{lem739}. The fourth conditional moment can be evaluated as
	\begin{align*}
		\E{\abs{\sum_{j=2}^{k_{n}-2}\CE{\parens{r_{j,n}^{(1)}(\lambda)}^4}{\mathcal{H}_j^n}}}\to0
	\end{align*}
	by Lemma \ref*{lem739}.
	For the remaining parts, Lemma 9 in \citep{GeJ93}, Proposition \ref{pro737} and Lemma \ref{lem739} show $\bar{R}^{(l)}=o_P(1)$ for $l=2,3,4$ with simple computation.\\

	\noindent\textbf{(Step 4): } We check the covariance structures among $\sqrt{n}D_{n}$, $U_{n}^{(1)}$, $U_{n}^{(3)}$, $U_{n}^{(4)}$ $\bar{R}_{n}^{(1)}$ which have not been shown. It is easy to see
	\begin{align*}
		\CE{\parens{D_{j,n}'}^{(l_1,l_2)}s_{j,n}^{(1)}(\kappa)}{\mathcal{H}_j^n}&=0,\qquad
		\CE{\parens{D_{j,n}'}^{(l_1,l_2)}s_{j,n}^{(4)}(\kappa)}{\mathcal{H}_j^n}=0,\\
		\CE{\parens{D_{j,n}'}^{(l_1,l_2)}r_{j,n}^{(1)}(\lambda)}{\mathcal{H}_j^n}&=0,\qquad
		\CE{s_{j,n}^{(3)}(\kappa)r_{j,n}^{(1)}(\lambda)}{\mathcal{H}_j^n}=0.
	\end{align*}
	The remaining evaluation are routine with Lemma 9 in \citep{GeJ93}
	\begin{align*}
		\frac{\sqrt{n}}{k_{n}^{3/2}}\sum_{j=1}^{k_{n}-2}\CE{\parens{D_{j,n}'}^{(l_1,l_2)}s_{j,n}^{(3)}(\kappa)}{\mathcal{H}_j^n}&=o_P(1),\\
		\frac{1}{\sqrt{k_{n}}}\sum_{j=1}^{k_{n}-2}\CE{s_{j,n}^{(1)}(\kappa)r_{j,n}^{(1)}(\lambda)}{\mathcal{H}_j^n}&=o_P(1),\\
		\frac{1}{\sqrt{k_{n}}}\sum_{j=1}^{k_{n}-2}\CE{s_{j,n}^{(4)}(\kappa)r_{j,n}^{(1)}(\lambda)}{\mathcal{H}_j^n}&=o_{P}(1).
	\end{align*}
	Then we obtain the proof.
\end{proof}

\begin{corollary}\label{cor752}
	With the same assumption as Theorem \ref*{thm751}, we have
	\begin{align*}
	\crotchet{\begin{matrix}
		\sqrt{n}D_{n}\\
		\sqrt{k_{n}}\crotchet{\bar{Q}_{n}\parens{B_{\kappa,n}(\cdot)}
			-\frac{2}{3}\bar{M}_{n}\parens{\ip{B_{\kappa,n}(\cdot)}{A_{n}^{\tau}\parens{\cdot,\alpha^{\star},\Lambda_{\star}}}}}_{\kappa}\\
		\sqrt{k_{n}\Delta_{n}}\crotchet{\bar{D}_{n}\parens{f_{\lambda}(\cdot)}}^{\lambda}
		\end{matrix}}\cl N(\mathbf{0},W(\tuborg{B_{\kappa}},\tuborg{f_{\lambda}})).
	\end{align*}
\end{corollary}

The proof is essentially analogous to that of Corollary 1 in \citep{Fa16}.

\subsection{Proofs of results in Section 3.1}
\begin{proof}[Proof of Lemma \ref*{lem311}]
	Theorem \ref{thm751} and continuous mapping theorem lead to consistency.
\end{proof}

Let us define the following quasi-likelihood functions such that
	\begin{align*}
	\check{\mathbb{H}}_{1,n}(\alpha):=-\frac{1}{2}\sum_{j=1}^{k_{n}-2}
	\parens{\ip{\parens{\frac{2}{3}\Delta_{n} A(\lm{Y}{j-1},\alpha)}^{-1}}{\parens{\lm{Y}{j+1}-\lm{Y}{j}}^{\otimes 2}}
		+\log\det A(\lm{Y}{j-1},\alpha)}
	\end{align*}
	and the corresponding ML-type estimator $\check{\alpha}_{n}$ where
	\begin{align*}
	\check{\mathbb{H}}_{1,n}(\check{\alpha}_{n})=\sup_{\alpha\in\Theta_1}\check{\mathbb{H}}_{1,n}(\alpha) 
	\end{align*}
\begin{lemma}\label{lem761}
	Under (A1)-(A7), (AH) and $\tau\in(1,2)$, $\check{\alpha}_{n}$ is consistent. 
\end{lemma}

\begin{proof}
	For Proposition \ref{pro741} and Theorem \ref{thm743}, we obtain 
	\begin{align*}
		&\frac{1}{k_{n}}\check{\mathbb{H}}_{1,n}\parens{\alpha}
		\cp -\frac{1}{2}\nu_0\parens{\parens{A(\cdot,\alpha)}^{-1}A(\cdot,\alpha^{\star})+\log \det A(\cdot,\alpha)}\text{ uniformly in }\theta.
	\end{align*}
	This uniform convergence and Assumption (A6) lead to consistency of the estimators with the discussion identical to Theorem 1 in \citep{K97}.
\end{proof}

\begin{proof}[Proof of Theorem \ref*{thm312}]
	First of all, we see the consistency of $\hat{\alpha}_{n}$. When $\tau=2$,
	\begin{align*}
	\frac{1}{k_{n}}\mathbb{H}_{1,n}^{\tau}(\alpha|\Lambda)\cp -\frac{1}{2}\nu_0\parens{\parens{A^{\tau}(\cdot,\alpha,\Lambda)}^{-1}A^{\tau}(\cdot,\alpha^{\star},\Lambda_{\star})+\log \det A^{\tau}(\cdot,\alpha,\Lambda)}\text{ uniformly in }\vartheta
	\end{align*}
	because of Proposition \ref{pro741} and Theorem \ref{thm743}, and then Lemma \ref{lem311} results in 
	\begin{align*}
		\frac{1}{k_{n}}\mathbb{H}_{1,n}^{\tau}(\alpha|\hat{\Lambda}_{n})\cp -\frac{1}{2}\nu_0\parens{\parens{A^{\tau}(\cdot,\alpha,\Lambda_{\star})}^{-1}A^{\tau}(\cdot,\alpha^{\star},\Lambda_{\star})+\log \det A^{\tau}(\cdot,\alpha,\Lambda_{\star})}
	\end{align*}
	uniformly in $\theta$. Therefore we can reduce the discussion to that in the previous Lemma \ref{lem761}.
	When $\tau\in(1,2)$, it is sufficient to see
	\begin{align*}
		&\frac{1}{k_{n}}\sum_{j=1}^{k_{n}-2}
		\parens{\ip{\parens{\frac{2}{3}\Delta_{n} A_{n}^{\tau}(\lm{Y}{j-1},\alpha,\hat{\Lambda}_{n})}^{-1}}{\parens{\lm{Y}{j+1}-\lm{Y}{j}}^{\otimes 2}}
			+\log\det \parens{ A_{n}^{\tau}(\lm{Y}{j-1},\alpha,\hat{\Lambda}_{n})}}\\
		&\quad\qquad\qquad-\frac{1}{k_{n}}\sum_{j=1}^{k_{n}-2}
		\parens{\ip{\parens{\frac{2}{3}\Delta_{n} A(\lm{Y}{j-1},\alpha)}^{-1}}{\parens{\lm{Y}{j+1}-\lm{Y}{j}}^{\otimes 2}}
		+\log\det \parens{ A(\lm{Y}{j-1},\alpha)}}\\
		&\cp 0 \text{ uniformly in }\theta
	\end{align*}
	because of Lemma \ref*{lem761}.
	Note that
	\begin{align*}
		\sup_{\theta\in\Theta}\norm{\parens{A_{n}^{\tau}(x,\alpha,\hat{\Lambda}_{n})}^{-1}-\parens{A(x,\alpha)}^{-1}}&\le C\Delta_{n}^{\frac{2-\tau}{\tau-1}}\parens{1+\norm{x}^C},\\
		\sup_{\theta\in\Theta}\abs{\log\det\parens{A_{n}^{\tau}(x,\alpha,\hat{\Lambda}_{n})}-\log\det\parens{A(x,\alpha)}}&\le C\Delta_{n}^{\frac{2-\tau}{\tau-1}}\parens{1+\norm{x}^C}.
	\end{align*}
	Using these inequalities and Lemma 9 in \citep{GeJ93} lead to the evaluation above with analogous discussion for Corollary 5.3 in \citep{Fa14}. Hence we have the consistency of $\hat{\alpha}_{n}$.
	
	In the next place we consider the consistency of $\hat{\beta}_{n}$. Because of Proposition \ref{pro741} and Theorem \ref{thm742}, 
	\begin{align*}
	&\frac{1}{k_{n}\Delta_{n}}\mathbb{H}_{2,n}(\beta|\alpha)-\frac{1}{k_{n}\Delta_{n}}\mathbb{H}_{2,n}(\beta^{\star}|\alpha)\\
	&=\frac{1}{k_{n}\Delta_{n}}\sum_{j=1}^{k_{n}-2}
	\ip{\parens{\Delta_{n}A(\lm{Y}{j-1},\alpha)}^{-1}}{\parens{\lm{Y}{j+1}-\lm{Y}{j}-\Delta_{n}b(\lm{Y}{j-1},\beta)}^{\otimes 2}}\\
	&\qquad-\frac{1}{k_{n}\Delta_{n}}\sum_{j=1}^{k_{n}-2}
	\ip{\parens{\Delta_{n}A(\lm{Y}{j-1},\alpha)}^{-1}}{\parens{\lm{Y}{j+1}-\lm{Y}{j}-\Delta_{n}b(\lm{Y}{j-1},\beta^{\star})}^{\otimes 2}}\\
	&\cp -\frac{1}{2}\nu_0\parens{\ip{\parens{c(\cdot,\alpha)}^{-1}}{\left(b(\cdot,\beta)-b(\cdot,\beta^{\star})\right)^{\otimes2}}} \text{ uniformly in }\theta
	\end{align*}
	and $\hat{\alpha}_{n}\cp \alpha^{\star}$ leads to
	\begin{align*}
		&\frac{1}{k_{n}\Delta_{n}}\mathbb{H}_{2,n}(\beta|\hat{\alpha}_{n})-\frac{1}{k_{n}\Delta_{n}}\mathbb{H}_{2,n}(\beta^{\star}|\hat{\alpha}_{n})\cp \mathbb{Y}_{2}\left(\beta\right) \text{ uniformly in }\theta.
	\end{align*}
	Hence the discussion identical to \citep{K97} verifies the consistency of $\hat{\beta}_{n}$.
\end{proof}

\begin{proof}[Proof of Theorem \ref*{thm313}]
	Firstly we prepare the notation
	\begin{align*}
		\hat{J}_{n}^{(1,2)}(\hat{\vartheta}_{n})&:=-\int_{0}^{1}\frac{1}{\sqrt{nk_{n}}}\partial_{\theta_{\epsilon}}\partial_{\alpha}\mathbb{H}_{1,n}^{\tau}(\hat{\alpha}_{n}|\theta_{\epsilon}^{\star}+u(\hat{\theta}_{\epsilon,n}-\theta_{\epsilon}^{\star}))\dop u,\\
		\hat{J}_{n}^{(2,2)}(\hat{\vartheta}_{n})&:=-\int_{0}^{1}\frac{1}{k_{n}}\partial_{\alpha}^2\mathbb{H}_{1,n}^{\tau}(\alpha^{\star}+u(\hat{\alpha}_{n}-\alpha^{\star})|\theta_{\epsilon}^{\star})\dop u,\\
		\hat{J}_{n}^{(2,3)}(\hat{\theta}_{n})&:=-\int_{0}^{1}\frac{1}{k_{n}\sqrt{\Delta_{n}}}\partial_{\alpha}\partial_{\beta}
		\mathbb{H}_{2,n}(\hat{\beta}_{n}|\alpha^{\star}+u(\hat{\alpha}_{n}-\alpha^{\star}))\dop u,\\
		\hat{J}_{n}^{(3,3)}(\hat{\theta}_{n})&:=-\int_{0}^{1}\frac{1}{k_{n}\Delta_{n}}\partial_{\beta}^2
		\mathbb{H}_{2,n}(\beta^{\star}+u(\hat{\beta}_{n}-\beta^{\star})|\alpha^{\star})\dop u,\\
		\hat{J}_{n}(\hat{\vartheta}_{n})&:=\crotchet{\begin{matrix}
			I & O & O \\
			\hat{J}_{n}^{(1,2)}(\hat{\vartheta}_{n}) & \hat{J}_{n}^{(2,2)}(\hat{\vartheta}_{n}) & O\\
			O & \hat{J}_{n}^{(2,3)}(\hat{\theta}_{n}) & \hat{J}_{n}^{(3,3)}(\hat{\theta}_{n})
			\end{matrix}}.
	\end{align*}
	Taylor's theorem gives
	\begin{align*}
		&\parens{\frac{1}{\sqrt{k_{n}}}\partial_{\alpha}\mathbb{H}_{1,n}^{\tau}(\hat{\alpha}_{n}|\hat{\theta}_{\epsilon,n})-\frac{1}{\sqrt{k_{n}}}\partial_{\alpha}\mathbb{H}_{1,n}^{\tau}(\alpha^{\star}|\theta_{\epsilon}^{\star})}^T\\
		&=\parens{\int_{0}^{1}\frac{1}{\sqrt{nk_{n}}}\partial_{\theta_{\epsilon}}\partial_{\alpha}\mathbb{H}_{1,n}^{\tau}(\hat{\alpha}_{n}|\theta_{\epsilon}^{\star}+u(\hat{\theta}_{\epsilon,n}-\theta_{\epsilon}^{\star}))\dop u }\sqrt{n}\parens{\hat{\theta}_{\epsilon,n}-\theta_{\epsilon}^{\star}}\\
		&\qquad+\parens{\int_{0}^{1}\frac{1}{k_{n}}\partial_{\alpha}^2\mathbb{H}_{1,n}^{\tau}(\alpha^{\star}+u(\hat{\alpha}_{n}-\alpha^{\star})|\theta_{\epsilon}^{\star})\dop u 
		}\sqrt{k_{n}}\parens{\hat{\alpha}_{n}-\alpha^{\star}}
	\end{align*}
	and the definition of $\hat{\alpha}_{n}$ leads to
	\begin{align*}
		\parens{-\frac{1}{\sqrt{k_{n}}}\partial_{\alpha}\mathbb{H}_{1,n}^{\tau}(\alpha^{\star}|\theta_{\epsilon}^{\star})}^T
		&=-\hat{J}_{n}^{(1,2)}(\hat{\vartheta}_{n})\sqrt{n}\parens{\hat{\theta}_{\epsilon,n}-\theta_{\epsilon}^{\star}}-\hat{J}_{n}^{(2,2)}(\hat{\vartheta}_{n})
		\sqrt{k_{n}}\parens{\hat{\alpha}_{n}-\alpha^{\star}}.
	\end{align*}
	Similarly we have
	\begin{align*}
		&\frac{1}{\sqrt{k_{n}\Delta_{n}}}
		\parens{\partial_{\beta}\mathbb{H}_{2,n}(\hat{\beta}_{n}|\hat{\alpha}_{n})
			-\partial_{\beta}\mathbb{H}_{2,n}(\beta^{\star}|\alpha^{\star})}^T\\
		&=\parens{\int_{0}^{1}\frac{1}{k_{n}\sqrt{\Delta_{n}}}\partial_{\alpha}\partial_{\beta}
			\mathbb{H}_{2,n}(\hat{\beta}_{n}|\alpha^{\star}+u(\hat{\alpha}_{n}-\alpha^{\star}))\dop u}\sqrt{k_{n}}
		\parens{\hat{\alpha}_{n}-\alpha^{\star}}\\
		&\qquad+\parens{\int_{0}^{1}\frac{1}{k_{n}\Delta_{n}}\partial_{\beta}^2
			\mathbb{H}_{2,n}(\beta^{\star}+u(\hat{\beta}_{n}-\beta^{\star})|\alpha^{\star})\dop u}\sqrt{k_{n}\Delta_{n}}
		\parens{\hat{\beta}_{n}-\beta^{\star}}
	\end{align*}
	and hence
	\begin{align*}
		\parens{-\frac{1}{\sqrt{k_{n}\Delta_{n}}}\partial_{\beta}\mathbb{H}_{2,n}(\beta^{\star}|\theta_{\epsilon}^{\star},\alpha^{\star})}^T
		&=-\hat{J}_{n}^{(2,3)}(\hat{\theta}_{n})\sqrt{k_{n}}
		\parens{\hat{\alpha}_{n}-\alpha^{\star}}\\
		&\qquad-\hat{J}_{n}^{(3,3)}(\hat{\theta}_{n})\sqrt{k_{n}\Delta_{n}}
		\parens{\hat{\beta}_{n}-\beta^{\star}}.
	\end{align*}
	Here we obtain
	\begin{align*}
		&\crotchet{\begin{matrix}
			-\sqrt{n}D_{n}\\
			-\frac{1}{\sqrt{k_{n}}}\crotchet{\partial_{\alpha^{(i_1)}}\mathbb{H}_{1,n}^{\tau}(\alpha^{\star}|\theta_{\epsilon}^{\star})}_{i_1=1,\ldots,m_1}\\
			-\frac{1}{\sqrt{k_{n}\Delta_{n}}}\crotchet{\partial_{\beta^{(i_2)}}\mathbb{H}_{2,n}(\beta^{\star}|\alpha^{\star})}_{i_2=1,\ldots,m_2}
			\end{matrix}}
			=-\hat{J}_{n}(\hat{\vartheta}_{n})\crotchet{\begin{matrix}
			\sqrt{n}
			\parens{\hat{\theta}_{\epsilon,n}-\theta_{\epsilon}^{\star}}\\
			\sqrt{k_{n}}
			\parens{\hat{\alpha}_{n}-\alpha^{\star}}\\
			\sqrt{k_{n}\Delta_{n}}
			\parens{\hat{\beta}_{n}-\beta^{\star}}
			\end{matrix}}
	\end{align*}
	and we check the asymptotics of the left hand side and the right one.
	
	\noindent\textbf{(Step 1): } For $i=1,\ldots,m_1$, we can evaluate
	\begin{align*}
		&-\frac{1}{\sqrt{k_{n}}}\partial_{\alpha^{(i)}}\mathbb{H}_{1,n}^{\tau}(\alpha^{\star}|\theta_{\epsilon}^{\star})\\
		&=-\frac{\sqrt{k_{n}}}{2}\bar{Q}_{n}\parens{\parens{\frac{2}{3}}^{-1}\parens{ A_{n}^{\tau}(\cdot,\alpha^{\star},\Lambda_{\star})}^{-1}
			\partial_{\alpha^{(i)}}A(\cdot,\alpha^{\star})\parens{ A_{n}^{\tau}(\cdot,\alpha^{\star},\Lambda_{\star})}^{-1}}\\
		&\qquad+\frac{\sqrt{k_{n}}}{2}\bar{M}_{n}\parens{\ip{\parens{\parens{A{n}^{\tau}(\cdot,\alpha^{\star},\Lambda_{\star})}^{-1}
			\partial_{\alpha^{(i)}}A(\cdot,\alpha^{\star})\parens{ A_{n}^{\tau}(\cdot,\alpha^{\star},\Lambda_{\star})}^{-1}
		}}{A_{n}^{\tau}(\cdot,\alpha^{\star},\Lambda_{\star})}}\\
		&=-{\sqrt{k_{n}}}\bar{Q}_{n}\parens{\frac{3}{4}\parens{ A_{n}^{\tau}(\cdot,\alpha^{\star},\Lambda_{\star})}^{-1}
			\partial_{\alpha^{(i)}}A(\cdot,\alpha^{\star})\parens{ A_{n}^{\tau}(\cdot,\alpha^{\star},\Lambda_{\star})}^{-1}}\\
		&\qquad+\frac{2\sqrt{k_{n}}}{3}\bar{M}_{n}\parens{\frac{3}{4}\ip{\parens{\parens{A{n}^{\tau}(\cdot,\alpha^{\star},\Lambda_{\star})}^{-1}
					\partial_{\alpha^{(i)}}A(\cdot,\alpha^{\star})\parens{ A_{n}^{\tau}(\cdot,\alpha^{\star},\Lambda_{\star})}^{-1}
			}}{A_{n}^{\tau}(\cdot,\alpha^{\star},\Lambda_{\star})}}.
	\end{align*}
	For $i=1,\ldots,m_2$, we have
	\begin{align*}
		-\frac{1}{\sqrt{k_{n}\Delta_{n}}}\partial_{\beta^{(i)}}\mathbb{H}_{2,n}(\beta^{\star}|\alpha^{\star})
		&=-\sqrt{k_{n}\Delta_{n}}\bar{D}_{n}\parens{\parens{\partial_{\beta^{(i)}}b(\cdot,\beta^{\star})}^T\parens{A(\cdot,\alpha^{\star})}^{-1}}.
	\end{align*}
	As shown in the proof of Theorem 3.1.2, if $\tau\in(1,2)$
	\begin{align*}
		\norm{\parens{A_{n}^{\tau}(x,\alpha^{\star},\Lambda_{\star})}^{-1}-\parens{A(x,\alpha^{\star})}^{-1}}\le C\Delta_{n}^{\frac{2-\tau}{\tau-1}}\parens{1+\norm{x}^C}
	\end{align*}
	and if $\tau=2$,
	\begin{align*}
		A_{n}^{\tau}(x,\alpha^{\star},\Lambda_{\star})=A(x,\alpha^{\star})+3\Lambda_{\star}=A^{\tau}(x,\alpha^{\star},\Lambda_{\star}).
	\end{align*}
	Therefore, by Theorem \ref*{thm751} and Corollary \ref*{cor752}, we obtain
	\begin{align*}
		\crotchet{\begin{matrix}
			-\sqrt{n}D_{n}\\
			-\frac{1}{\sqrt{k_{n}}}\crotchet{\partial_{\alpha^{(i_1)}}\mathbb{H}_{1,n}^{\tau}(\alpha^{\star}|\theta_{\epsilon}^{\star})}_{i_1=1,\ldots,m_1}\\
			-\frac{1}{\sqrt{k_{n}\Delta_{n}}}\crotchet{\partial_{\beta^{(i_2)}}\mathbb{H}_{2,n}(\beta^{\star}|\alpha^{\star})}_{i_2=1,\ldots,m_2}
			\end{matrix}}\cl N(\mathbf{0},\mathcal{I}^{\tau}(\vartheta^{\star})).
	\end{align*}
	
	\noindent\textbf{(Step 2): } We can compute 
	\begin{align*}
		\E{\norm{\int_{0}^{1}\frac{1}{\sqrt{nk_{n}}}\partial_{\theta_{\epsilon}}\partial_{\alpha}\mathbb{H}_{1,n}^{\tau}(\hat{\alpha}_{n}|\theta_{\epsilon}^{\star}+u(\hat{\theta}_{\epsilon,n}-\theta_{\epsilon}^{\star}))\dop u}} 
		&\to 0,
	\end{align*}
	and
	\begin{align*}
		\frac{1}{k_{n}\sqrt{\Delta_{n}}}\partial_{\alpha}\partial_{\beta}
		\mathbb{H}_{2,n}(\beta|\alpha)&\cp 0 \text{ uniformly in }\theta.
	\end{align*}
	We also have for $i_1,i_2\in\tuborg{1,\ldots,m_1}$
	\begin{align*}
		&-\frac{1}{k_{n}}\partial_{\alpha^{(i_1)}}\partial_{\alpha^{(i_2)}}\mathbb{H}_{1,n}^{\tau}(\alpha|\theta_{\epsilon}^{\star})\\
		&\cp 
		\crotchet{\frac{1}{2}\nu_0\parens{\tr\tuborg{\parens{A^{\tau}(\cdot,\alpha,\Lambda_{\star})}^{-1}\partial_{\alpha^{(i_1)}}A(\cdot,\alpha)\parens{A^{\tau}(\cdot,\alpha,\Lambda_{\star})}^{-1}\partial_{\alpha^{(i_2)}}A(\cdot,\alpha)}}}_{i_1,i_2}
	\end{align*}
	uniformly in $\alpha$ because of Proposition \ref*{pro741}, Theorem \ref*{thm743} and
	\begin{align*}
		\norm{\parens{A_{n}^{\tau}(x,\alpha^{\star},\Lambda_{\star})}^{-1}-\parens{A(x,\alpha^{\star})}^{-1}}\le C\Delta_{n}^{\frac{2-\tau}{\tau-1}}\parens{1+\norm{x}^C}
	\end{align*}
	for $\tau\in(1,2)$. Similarly, for $j_1,j_2\in\tuborg{1,\ldots,m_2}$
	\begin{align*}
		&-\frac{1}{k_{n}\Delta_{n}}\partial_{\beta^{(j_1)}}\partial_{\beta^{(j_2)}}
		\mathbb{H}_{2,n}(\beta|\alpha^{\star})\\
		&\cp
		\lcrotchet{\nu_0\parens{\ip{\parens{A(\cdot,\alpha^{\star})}^{-1}}
				{\parens{\partial_{\beta^{(j_1)}}\partial_{\beta^{(j_2)}}b(\cdot,\beta)}\parens{b(\cdot,\beta)-b(\cdot,\beta^{\star})}^T
		}}}\\
		&\hspace{2cm}\rcrotchet{+\nu_0\parens{\ip{\parens{A(\cdot,\alpha^{\star})}^{-1}}{\parens{\partial_{\beta^{(j_1)}}b}\parens{\partial_{\beta^{(j_2)}}b}^T
					(\cdot,\beta)}}}_{j_{1},j_{2}}
	\end{align*}
	uniformly in $\beta$. Hence these uniform convergences and consistency of $\hat{\alpha}_{n}$ and $\hat{\beta}_{n}$ lead to
	\begin{align*}
		-\int_{0}^{1}\frac{1}{k_{n}}\partial_{\alpha}^2\mathbb{H}_{1,n}^{\tau}(\alpha^{\star}+u(\hat{\alpha}_{n}-\alpha^{\star})|\theta_{\epsilon}^{\star})\dop u&\cp \mathcal{J}^{(2,2),\tau}(\vartheta^{\star}),\\
		-\int_{0}^{1}\frac{1}{k_{n}\Delta_{n}}\partial_{\beta}^2
		\mathbb{H}_{2,n}(\beta^{\star}+u(\hat{\beta}_{n}-\beta^{\star})|\alpha^{\star})\dop u&\cp \mathcal{J}^{(3,3)}(\theta^{\star}),
	\end{align*}
	and $\hat{J}_{n}(\hat{\vartheta}_{n})\cp \mathcal{J}^{\tau}(\vartheta^{\star})$.
\end{proof}

\subsection{Proofs of results in Section 3.2}
First of all, we define
\begin{align*}
	c_S(x):=\sum_{\ell_1=1}^{d}\sum_{\ell_2=1}^{d}A^{(\ell_1,\ell_2)}(x).
\end{align*}

{
\begin{proposition}\label{pro771}
	Assume $\Lambda_{\star}^{\left(\ell_{1},\ell_{2}\right)}=\left(h_{n}/\sqrt{n}\right)\mathfrak{M}$ for some matrix $\mathfrak{M}\ge 0$. Under (A1)-(A4) and $nh_{n}^2\to0$,
	\begin{align*}
	&\sqrt{n}\parens{\frac{1}{nh_{n}}\sum_{i=0}^{n-1}\parens{\mathscr{S}_{(i+1)h_{n}}-\mathscr{S}_{ih_{n}}}^2
		-\frac{1}{nh_{n}}\sum_{0\le 2i\le n-2}\parens{\mathscr{S}_{(2i+2)h_{n}}-\mathscr{S}_{2ih_{n}}}^2}\\
	&\qquad\cl N\parens{\delta,2\nu_0\parens{c_S^2(\cdot)}}
	\end{align*}
	where $\delta:=\sum_{\ell_{1}}\sum_{\ell_{2}}\mathfrak{M}$.
\end{proposition}
}

\begin{proof}
	We have
	\begin{align*}
	&\sqrt{n}\parens{\frac{1}{nh_{n}}\sum_{i=0}^{n-1}\parens{\mathscr{S}_{(i+1)h_{n}}-\mathscr{S}_{ih_{n}}}^2
		-\frac{1}{nh_{n}}\sum_{0\le 2i\le n-2}\parens{\mathscr{S}_{(2i+2)h_{n}}-\mathscr{S}_{2ih_{n}}}^2}\\
	&=\sum_{0\le 2i\le n-2}\left(\mathfrak{A}_{2i}^n-\frac{2\delta}{n}\right)+\delta+o_P(1),
	\end{align*}
	where
	\begin{align*}
	\mathfrak{A}_{2i}^n=\frac{\parens{-2}}{\sqrt{n}h_{n}}\parens{\mathscr{S}_{(2i+2)h_{n}}-\mathscr{S}_{(2i+1)h_{n}}}\parens{\mathscr{S}_{(2i+1)h_{n}}-\mathscr{S}_{2ih_{n}}}.
	\end{align*}
	For independence among $\left\{X_{t}\right\}_{t}$ and $\left\{\epsilon_{ih_{n}}\right\}_{i}$ being i.i.d., and 
	\begin{align*}
	&\CE{\mathfrak{A}_{2i}^n}{\mathcal{H}_{2ih_{n}}^n}\\
	&=\CE{\frac{\parens{-2}}{\sqrt{n}h_{n}}\parens{S_{(2i+2)h_{n}}-S_{(2i+1)h_{n}}}\parens{S_{(2i+1)h_{n}}-S_{2ih_{n}}}}{\mathcal{H}_{2ih_{n}}^n}\\
	&\quad+\E{\frac{\parens{-2}}{\sqrt{n}h_{n}}\sum_{\ell_{1}}\left(\Lambda_{\star}^{1/2}\left(\epsilon_{\left(2i+2\right)h_{n}}-\epsilon_{\left(2i+1\right)h_{n}}\right)\right)^{\left(\ell_{1}\right)}
	\sum_{\ell_{2}}\left(\Lambda_{\star}^{1/2}\left(\epsilon_{\left(2i+1\right)h_{n}}-\epsilon_{2ih_{n}}\right)\right)^{\left(\ell_{2}\right)}},
	\end{align*}
	and the first term has the evaluation
	\begin{align*}
		\left|\CE{\frac{\parens{-2}}{\sqrt{n}h_{n}}\parens{S_{(2i+2)h_{n}}-S_{(2i+1)h_{n}}}\parens{S_{(2i+1)h_{n}}-S_{2ih_{n}}}}{\mathcal{H}_{2ih_{n}}^n}\right|\le \frac{Ch_{n}}{\sqrt{n}}\left(1+\left\|X_{2ih_{n}}\right\|\right)^{C},
	\end{align*}
	and the second one has
	\begin{align*}
		&\E{\frac{\parens{-2}}{\sqrt{n}h_{n}}\sum_{\ell_{1}}\left(\Lambda_{\star}^{1/2}\left(\epsilon_{\left(2i+2\right)h_{n}}-\epsilon_{\left(2i+1\right)h_{n}}\right)\right)^{\left(\ell_{1}\right)}
			\sum_{\ell_{2}}\left(\Lambda_{\star}^{1/2}\left(\epsilon_{\left(2i+1\right)h_{n}}-\epsilon_{2ih_{n}}\right)\right)^{\left(\ell_{2}\right)}}\\
		&=\frac{2}{\sqrt{n}h_{n}}\sum_{\ell_{1}}\sum_{\ell_{2}}\Lambda_{\star}^{\left(\ell_{1},\ell_{2}\right)}
		=\frac{2\delta}{n}.
	\end{align*}
	These verify
	\begin{align*}
		\sum_{0\le 2i\le n-2}\E{\abs{\CE{\mathfrak{A}_{2i}^n-\frac{2\delta}{n}}{\mathcal{H}_{2ih_{n}}^n}}}
		&\to 0.
	\end{align*}
	We also have the evaluation
	\begin{align*}
		\sum_{0\le 2i\le n-2}\CE{\left(\mathfrak{A}_{2i}^n-\frac{2\delta}{n}\right)^{2}}{\mathcal{H}_{2ih_{n}}^n}
		&=\sum_{0\le 2i\le n-2}\CE{\left(\mathfrak{A}_{2i}^n\right)^{2}}{\mathcal{H}_{2ih_{n}}^n}+o_{P}\left(1\right),
	\end{align*}
	and the order of $\Lambda_{\star}=\frac{h_{n}}{\sqrt{n}}\mathfrak{M}$ leads to
	\begin{align*}
		&\sum_{0\le 2i\le n-2}\CE{\left(\mathfrak{A}_{2i}^n\right)^{2}}{\mathcal{H}_{2ih_{n}}^n}\\
		&=\sum_{0\le 2i\le n-2}\CE{\frac{4}{nh_{n}^{2}}\parens{S_{(2i+2)h_{n}}-S_{(2i+1)h_{n}}}^{2}\parens{S_{(2i+1)h_{n}}-S_{2ih_{n}}}^{2}}{\mathcal{H}_{2ih_{n}}^n}
		+o_{P}\left(1\right)\\
		&\cp 2\nu_0\left(c_{S}^{2}\left(\cdot\right)\right).
	\end{align*}
	The identical discussion verifies
	\begin{align*}
		\mathbf{E}\left[\left|\sum_{0\le 2i\le n-2}\CE{\left(\mathfrak{A}_{2i}^n\right)^{2}}{\mathcal{H}_{2ih_{n}}^n}\right|\right]\to 0.
	\end{align*}
	Hence martingale CLT verifies the result.
\end{proof}

\begin{lemma}\label{lem772}
	Under (A1)-(A4) and (AH),
	\begin{align*}
		&\frac{1}{k_{n}\Delta_{n}^2}\sum_{j=1}^{k_{n}-2}\parens{\lm{\mathscr{S}}{j+1}-\lm{\mathscr{S}}{j}}^4\\
		&\cp \begin{cases}
		\frac{4}{3}\nu_0\parens{c_S^2(\cdot)} &\text{ if }\tau\in(1,2)\\
		\frac{4}{3}\nu_0\parens{c_S^2(\cdot)}+\frac{2}{3}\nu_0\parens{c_S(\cdot)}\parens{\sum_{i_1=1}^{d}\sum_{i_2=1}^{d}\Lambda_{\star}^{(i_1,i_2)}}\\
		\qquad+\parens{2d^2+10d}\parens{\sum_{i_1=1}^{d}\sum_{i_2=1}^{d}\Lambda_{\star}^{(i_1,i_2)}}^2&\text{ if }\tau=2.
		\end{cases}
	\end{align*}
\end{lemma}

\begin{proof} We prove with the identical way as Theorem \ref*{thm743}. Note the evaluation
	\begin{align*}
		&\frac{1}{k_{n}\Delta_{n}^2}\sum_{j=1}^{k_{n}-2}\parens{\lm{\mathscr{S}}{j+1}-\lm{\mathscr{S}}{j}}^4\\
		&=\frac{1}{k_{n}\Delta_{n}^2}\sum_{j=1}^{k_{n}-2}\parens{\sum_{i=1}^{d}\parens{a^{(i,\cdot)}(X_{j\Delta_{n}})\parens{\zeta_{j+1,n}+\zeta_{j+2,n}'}+\parens{\Lambda_{\star}^{1/2}}^{(i,\cdot)}\parens{\lm{\epsilon}{j+1}-\lm{\epsilon}{j}}}}^4+o_{P}(1),
	\end{align*}
	using Proposition \ref{pro737} and Lemma \ref{lem739}. We have
	\begin{align*}
		\mathfrak{C}_{j,n}&:=\CE{\parens{\sum_{i=1}^{d}\parens{a^{(i,\cdot)}(X_{j\Delta_{n}})\parens{\zeta_{j+1,n}+\zeta_{j+2,n}'}+\parens{\Lambda_{\star}^{1/2}}^{(i,\cdot)}\parens{\lm{\epsilon}{j+1}-\lm{\epsilon}{j}}}}^4}{\mathcal{H}_{j}^{n}}\\
		&=3\parens{m_{n}+m_{n}'}^2\Delta_{n}^2c_S^2(X_{j\Delta_{n}})
		+\parens{m_{n}+m_{n}'}\Delta_{n}c_S(X_{j\Delta_{n}})\frac{1}{p_{n}}\parens{\sum_{i_1=1}^{d}\sum_{i_2=1}^{d}
			\Lambda_{\star}^{(i_1,i_2)}}\\
		&\qquad+\parens{\frac{2d^2+10d}{p_{n}^2}+\frac{2}{p_{n}^3}\parens{\sum_{i=1}^{d}\E{\parens{\epsilon_{0}^{(i)}}^4}-3d}}\parens{\sum_{i_1=1}^{d}\sum_{i_2=1}^{d}
			\Lambda_{\star}^{(i_1,i_2)}}^2
	\end{align*}
	and hence
	\begin{align*}
		\frac{1}{k_{n}\Delta_{n}^2}\sum_{1\le 3j\le k_{n}-2}\mathfrak{C}_{3j,n}
		&\cp \begin{cases}
		\frac{4}{9}\nu_0\parens{c_S^2(\cdot)} &\text{ if }\tau\in(1,2)\\
		\frac{4}{9}\nu_0\parens{c_S^2(\cdot)}+\frac{2}{9}\nu_0\parens{c_S(\cdot)}\parens{\sum_{i_1=1}^{d}\sum_{i_2=1}^{d}\Lambda_{\star}^{(i_1,i_2)}}\\
		\qquad+\frac{2d^2+10d}{3}\parens{\sum_{i_1=1}^{d}\sum_{i_2=1}^{d}\Lambda_{\star}^{(i_1,i_2)}}^2&\text{ if }\tau=2.
		\end{cases}
	\end{align*}
	Finally we obtain
	\begin{align*}
		\E{\abs{\frac{1}{k_{n}^2\Delta_{n}^4}\sum_{1\le 3j\le k_{n}-2}\CE{\parens{\lm{\mathscr{S}}{3j+1}-\lm{\mathscr{S}}{3j}}^8}{\mathcal{H}_{3j}^n}}}
		\le \frac{C}{k_{n}}
		\to 0.
	\end{align*}
	and the proof is obtained because of the Lemma 9 in \citep{GeJ93}.
\end{proof}

\begin{proof}[Proof of Theorem \ref*{thm321}]
	Under $H_0$, the result of Lemma \ref*{lem772} is equivalent to
	\begin{align*}
		\frac{3}{4k_{n}\Delta_{n}^2}\sum_{j=1}^{k_{n}-2}\parens{\lm{\mathscr{S}}{j+1}-\lm{\mathscr{S}}{j}}^4 \cp
		\nu_0\parens{c_S^2(\cdot)},\ ^{\forall} \tau\in(1,2].
	\end{align*}
	Therefore, Proposition \ref*{pro771}, Lemma \ref*{lem772} and Slutsky's theorem {complete} the proof.
\end{proof}

\begin{proof}[Proof of Theorem \ref*{thm322}]
	Assumption (T1) verifies $\sum_{l_1}\sum_{l_2}\Lambda_{\star}^{(l_1,l_2)}>0$.
	We firstly show 
	\begin{align*}
		\frac{1}{2n}\sum_{i=0}^{n-1}\parens{\mathscr{S}_{(i+1)h_{n}}- \mathscr{S}_{ih_{n}}}^2
		\cp\sum_{l_1}\sum_{l_2}\Lambda_{\star}^{(l_1,l_2)}
	\end{align*}
	under $H_1$ and (T1).
	We can decompose
	\begin{align*}
		&\frac{1}{2n}\sum_{i=0}^{n-1}\parens{\mathscr{S}_{(i+1)h_{n}}-\mathscr{S}_{ih_{n}}}^2
		-\sum_{l_1}\sum_{l_2}\Lambda_{\star}^{(l_1,l_2)}\\
		&=\frac{1}{2n}\sum_{i=0}^{n-1}\parens{\sum_{l=1}^{d}\parens{X_{(i+1)h_{n}}^{(l)}-X_{ih_{n}}^{(l)}}}^2\\
		&\qquad+\frac{1}{2n}\sum_{i=0}^{n-1}\parens{\sum_{l=1}^{d}\parens{\Lambda_{\star}^{1/2}}^{(l,\cdot)}\parens{\epsilon_{(i+1)h_{n}}-\epsilon_{ih_{n}}}}^2-\sum_{l_1}\sum_{l_2}\Lambda_{\star}^{(l_1,l_2)}\\
		&\qquad+\frac{1}{n}\sum_{i=0}^{n-1}\parens{\sum_{l=1}^{d}\parens{X_{(i+1)h_{n}}^{(l)}-X_{ih_{n}}^{(l)}}}
		\parens{\sum_{l=1}^{d}\parens{\Lambda_{\star}^{1/2}}^{(l,\cdot)}\parens{\epsilon_{(i+1)h_{n}}-\epsilon_{ih_{n}}}}.
	\end{align*}
	The first term and the fourth one of the right hand side are $o_P(1)$ which can be shown by $L^1$-evaluation.
	We also can evaluate the second and third term of the right hand side as
	\begin{align*}
		&\frac{1}{2n}\sum_{i=0}^{n-1}\parens{\sum_{l=1}^{d}\parens{\Lambda_{\star}^{1/2}}^{(l,\cdot)}\parens{\epsilon_{(i+1)h_{n}}-\epsilon_{ih_{n}}}}^2-\sum_{l_1}\sum_{l_2}\Lambda_{\star}^{(l_1,l_2)}\\
		&=-\frac{1}{n}\sum_{i=0}^{n-1}\sum_{l_1=1}^{d}\sum_{l_2=1}^{d}\parens{\Lambda_{\star}^{1/2}}^{(l_1,\cdot)}\parens{\epsilon_{(i+1)h_{n}}}
		\parens{\epsilon_{ih_{n}}}^T\parens{\Lambda_{\star}^{1/2}}^{(\cdot,l_2)}
		+o_P(1)
	\end{align*}
	because of law of large numbers. The first term can be evaluated as
	\begin{align*}
		\E{\abs{\frac{1}{n}\sum_{i=0}^{n-1}\sum_{l_1=1}^{d}\sum_{l_2=1}^{d}\parens{\Lambda_{\star}^{1/2}}^{(l_1,\cdot)}\parens{\epsilon_{(i+1)h_{n}}}
		\parens{\epsilon_{ih_{n}}}^T\parens{\Lambda_{\star}^{1/2}}^{\cdot,l_2}}^2}\to 0.
	\end{align*}
	With identical computation, we obtain
	\begin{align*}
		\frac{1}{n}\sum_{0\le 2i\le n-2}\parens{\mathscr{S}_{(2i+2)h_{n}}-\mathscr{S}_{2ih_{n}}}^2&\cp \sum_{l_1,l_2}\Lambda_{\star}^{(l_1,l_2)}
	\end{align*}
	There exists a constant $C_1$ such that
	\begin{align*}
		\frac{3}{4k_{n}\Delta_{n}^2}\sum_{j=1}^{k_{n}-2}\parens{\lm{\mathscr{S}}{j+1}-\lm{\mathscr{S}}{j}}^4 &\cp C_1.
	\end{align*}
	These convergences in probability and some computations {verify} the result.
\end{proof}

\begin{proof}[Proof for Theorem \ref{thm323}]
	Proposition \ref{pro771} and the discussion similar to that in Lemma \ref{lem772} verify the result.
\end{proof}